\documentclass[11pt]{amsart}
\usepackage{amssymb}
\usepackage{amsthm}
\usepackage{amsmath}
\usepackage{mathtools}
\usepackage{latexsym}
\usepackage{comment}
\usepackage{hyperref}
\usepackage{cleveref}
\usepackage{verbatim}
\usepackage{xypic}
\usepackage{graphicx}
\usepackage{tikz-cd}
\usepackage[utf8]{inputenc}
\usepackage[margin=1in]{geometry}
\usepackage[shortlabels]{enumitem}
\usepackage{xcolor}

\usetikzlibrary{decorations.pathmorphing}

\newtheorem{thm}{Theorem}[section]
\newtheorem{prop}[thm]{Proposition}

\newtheorem{lemma}[thm]{Lemma}
\newtheorem{cor}[thm]{Corollary}
\newtheorem{dfn}[thm]{Definition}

\theoremstyle{definition}

\theoremstyle{definition} \newtheorem{rmk}[thm]{Remark}

\newcommand{\cc}{\mathbb{C}}
\newcommand{\rr}{\mathbb{R}}
\newcommand{\qq}{\mathbb{Q}}
\newcommand{\zz}{\mathbb{Z}}

\newcommand{\ggr}{\mathbb{G}}
\newcommand{\aff}{\mathbb{A}}
\newcommand{\proj}{\mathbb{P}}

\newcommand{\Gal}{\mathrm{Gal}}
\newcommand{\GL}{\mathrm{GL}}
\newcommand{\PGL}{\mathrm{PGL}}

\newcommand{\Hom}{\mathrm{Hom}}
\newcommand{\End}{\mathrm{End}}
\newcommand{\Aut}{\mathrm{Aut}}

\newcommand{\Div}{\mathrm{Div}}
\newcommand{\Prin}{\mathrm{Prin}}
\newcommand{\Pic}{\mathrm{Pic}}

\newcommand{\ev}{\mathrm{ev}}
\newcommand{\mon}{\mathrm{mon}}

\newcommand{\Berk}{\mathbb{P}_{\cc_K}^{1, \mathrm{an}}}
\newcommand{\Hyp}{\mathbb{H}_{\cc_K}^{1, \mathrm{an}}}

\newcommand{\bbLambda}{\mathbin{\text{\reflectbox{\raisebox{\depth}{\scalebox{1}[-1]{$\mathbb V$}}}}}}

\graphicspath{{pics/}}

\title{Split degenerate superelliptic curves and $\ell$-adic images of inertia}
\author{Jeffrey Yelton}

\begin{document}

\maketitle

\begin{abstract}

Let $K$ be a field with a discrete valuation, and let $p$ and $\ell$ be (possibly equal) primes which are not necessarily different from the residue characteristic.  Given a superelliptic curve $C : y^p = f(x)$ which has split degenerate reduction over $K$, with Jacobian denoted by $J / K$, we describe the action of an element of the inertia group $I_K$ on the $\ell$-adic Tate module $T_\ell(J)$ as a product of powers of certain transvections with respect to the $\ell$-adic Weil pairing and the canonical principal polarization of $J$.  The powers to which the transvections are taken are given by a formula depending entirely on the cluster data of the roots of the defining polynomial $f$.  This result is demonstrated using Mumford's non-archimedean uniformization of the curve $C$.

\end{abstract}

\section{Introduction} \label{sec intro}

Let $K$ be a field which is Henselian with respect to a discrete valuation $v : K \twoheadrightarrow \zz$; let $\pi \in K$ denote a uniformizer (so that we have $v(\pi) = 1$).  Let $\bar{K}$ denote an algebraic closure of $K$.  We write $I_K$ for the inertia subgroup of the absolute Galois group $\Gal(\bar{K} / K)$ of $K$.  For this paper, we fix a prime $p$ and consider a \emph{superelliptic curve} $C / K$ which is a $p$-cover of the projective line $\proj_K^1$.  Our goal is to describe, for any prime $\ell$, the action of a particular element of $I_K$ on the $\ell$-adic Tate module of the Jacobian of $C$.  We develop the necessary set-up before presenting our main result as \Cref{thm main} in \S\ref{sec intro main result} below.

\subsection{Superelliptic curves, their Jacobians, and the associated $\ell$-adic Galois actions} \label{sec intro superelliptic}
The type of superelliptic curve we are interested in is a smooth projective curve determined by the affine model given by an equation of the form 
\begin{equation} \label{eq superelliptic}
C: y^p = f(x) = c \prod_{i = 1}^d (x - z_i)^{r_i},
\end{equation}
where $f(x) \in K[x]$ is a polynomial of degree $d$ and we have $1 \leq r_i \leq p - 1$ for $1 \leq i \leq d$.  (In the special case that $p = 2$, the curve $C$ is \emph{hyperelliptic}, and the polynomial $f$ is squarefree.)  The set $\mathcal{B}$ of branch points of the degree-$p$ map $x : C \to \proj_K^1$ coincides with $\{z_i\}_{1 \leq i \leq d}$ (resp. $\{z_i\}_{1 \leq i \leq d} \cup \{\infty\}$) if we have $p \mid d$ (resp. $p \nmid d$).  For simplicity, we assume throughout this paper that the polynomial $f$ is monic of degree not divisible by $p$ and that its roots $z_i$ all lie in $K$, noting that, given an equation of the form (\ref{eq superelliptic}), such a model can always be obtained after replacing $K$ with a finite algebraic extension, applying a fractional linear transformations to $x$ that moves one of the branch points to $\infty$, and applying an appropriate transformation to $y$.  In this case, one can easily compute from the Riemann-Hurwitz formula that the genus $g$ of $C$ is given by $g = \frac{1}{2}(p - 1)(d - 1)$.

Throughout this paper, for any group variety $G / K$ and integer $m \geq 1$, we write $G[m]$ for the $m$-torsion subgroup of $G(\bar{K})$.  We denote the Jacobian variety of $C$ by $J$; it is an abelian variety over $K$ of dimension $g$.  For each prime $\ell$, we let $\displaystyle T_\ell(J) := \lim_{\leftarrow n} J[\ell^n]$ denote the \emph{$\ell$-adic Tate module of $J$}; it is a free $\zz_\ell$-module of rank $2g$.  We write 
\begin{equation*}
\rho_{\ell} : \Gal(\bar{K} / K) \to \Aut(T_{\ell}(J))
\end{equation*}
 for the natural $\ell$-adic Galois action on this Tate module.  We write $J^\vee$ for the dual abelian variety of $J$ and $T_\ell(J^\vee)$ for its $\ell$-adic Tate module.  The Jacobian variety $J$ is equipped with the canonical principal polarization $J \stackrel{\sim}{\to} J^\vee$, which induces an isomorphism $T_\ell(J) \stackrel{\sim}{\to} T_\ell(J^\vee)$ of $\ell$-adic Tate modules.  We have the \emph{Weil pairing} given by 
\begin{equation*}
\mathfrak{e}_\ell : T_\ell(J) \times T_\ell(J^\vee) \to T_\ell(\bar{K}^\times) := \lim_{\leftarrow n} \boldsymbol{\mu}_{\ell^n}
\end{equation*}
 (where $\boldsymbol{\mu}_{\ell^n}$ denotes the group of $\ell^n$th roots of unity); it is a Galois-equivariant $\zz_\ell$-bilinear pairing.  On composing with the aforementioned isomorphism $T_\ell(J) \stackrel{\sim}{\to} T_\ell(J^\vee)$ on the second argument, this gives us the Galois-equivariant $\zz_\ell$-bilinear skew-symmetric pairing 
\begin{equation*}
e_\ell : T_\ell(J) \times T_\ell(J) \to T_\ell(\bar{K}^\times).
\end{equation*}
 
For each integer $n \geq 0$, we fix a primitive $n$th root of unity denoted by $\zeta_n$, and we choose these roots of unity so that $\zeta_{mn}^m = \zeta_n$ for any $m, n \in \zz_{\geq 0}$.  We make the additional assumption throughout this paper that we have $\zeta_p \in K$.  The curve $C$ then has an order-$p$ automorphism over $K$ which respects its degree-$p$ map onto $\proj_K^1$ and which is given by $(x, y) \mapsto (x, \zeta_p y)$.  This induces an order-$p$ $K$-automorphism of the Jacobian $J$, which we also denote by $\zeta_p$ and which realizes an embedding of the ring $\zz[\zeta_p]$ into the $K$-endomorphism ring of $J$.  In turn, the automorphism $\zeta_p$ of $J$ acts on the $\ell$-adic Tate module $T_\ell(J)$.  We denote the resulting automorphism again by $\zeta_p \in \Aut(T_\ell(J))$; in fact, this realizes $T_\ell(J)$ as a $\displaystyle \zz[\zeta_p]_\ell := \lim_{\leftarrow n} \zz[\zeta_p] / \ell^n \zz[\zeta_p]$-module.  (See \cite{garnek2024exponential} for more details.)

For any element $\mathfrak{w} \in T_\ell(J)$, we denote by $t_{\mathfrak{w}} \in \Aut(T_\ell(J))$ the \emph{transvection} with respect to $\mathfrak{w}$, given by $t_{\mathfrak{w}} : \mathfrak{v} \mapsto \mathfrak{v} + e_\ell(\mathfrak{v}, \mathfrak{w}) \mathfrak{w}$, where we are viewing $e_\ell$ as taking values in the additive group $\zz_\ell$ via the obvious group isomorphism $T_\ell(\bar{K}^\times) \stackrel{\sim}{\to} \zz_\ell$.
 
In \cite[\S2]{yelton2021boundedness}, the author used topological methods to describe the image in $\Aut(T_\ell(J))$ of an element of $I_K$ as a product of transvections under certain hypotheses in the hyperelliptic case when $p = 2$ is different from the residue characteristic of $K$: \cite[Proposition 2.6]{yelton2021boundedness} describes this image of inertia as being generated by a power of a single transvection under a particular hypothesis on the geometry of the roots $z_i$.  It is implicit from the results of \cite[\S2]{yelton2021boundedness} that, replacing the hypothesis that, under the weaker hypothesis that the only proper non-singleton \emph{clusters} (see \Cref{dfn cluster} below) have even cardinality, the image of inertia is generated by a product of powers of (commuting) transvections with respect to elements of $T_\ell(J)$ which each correspond to one of these clusters.  The object of this paper is to prove a generalization of this for superelliptic curves, under the new hypothesis that the curve satisfy a condition called \emph{split degenerate reduction}, but without any restriction on the residue characteristic of $K$.

\subsection{Split degenerate superelliptic curves and cluster data} \label{sec intro clusters}

A curve $C$ of genus $g$ is said to have \emph{split degenerate reduction} over $K$ if it has a semistable model over the ring of integers of $K$ whose special fiber consists only of components which are each isomorphic to the projective line over the reside field.  There are several equivalent formulations for this condition on the special fiber of a semistable model of $C$, one of them being that the graph of components of its special fiber has $g$ loops, which in turn is equivalent to saying that this special fiber has (maximal) toric rank $g$.  By definition, this last condition means that identity component $(\mathcal{J}_s)_0$ of the special fiber of the N\'{e}ron model $\mathcal{J}$ of the Jacobian $J$ is an split torus (of rank $g$) over the residue field; in other words, the resulting abelian variety $J$ has \emph{split toric reduction}.

When a superelliptic curve $C$ has split degenerate reduction, we can say more about the form of the equation (\ref{eq superelliptic}) used to define the curve.  Letting $h = \frac{d - 1}{2} = \frac{g}{p - 1}$, by \cite[Proposition 3.1(a)]{van1982galois}, we may label the roots of the defining polynomial $f$ as $\alpha_0, \alpha_1, \beta_1, \dots, \alpha_h, \beta_h$ in such a way that $C$ has an equation of the form 
\begin{equation} \label{eq split degenerate superelliptic}
y^p = (x - \alpha_0)^{m_0} \prod_{i = 1}^h (x - \alpha_i)^{m_i} (x - \beta_i)^{p - m_i}
\end{equation}
for some integers $m_i$ each satisfying $1 \leq m_i \leq p - m_i$.  Letting $\beta_0 = \infty$, the set $\mathcal{B}$ of branch points of the map $C \to \proj_K^1$ is then given by $\{\alpha_i, \beta_i\}_{0 \leq i \leq h}$.

The $\ell$-adic action described by \Cref{thm main} below is determined by the \emph{cluster data} of this set of branch points; in order to make this precise, we introduce the language of \emph{clusters}, following \cite{dokchitser2022arithmetic}.  (In the following definition and elsewhere below, we use the notation $\proj_L^1$ for the set $L \cup \{\infty\}$ for any extension $L / K$.)

\begin{dfn} \label{dfn cluster}

Let $A \subset \proj_K^1$ be a finite subset.  A subset $\mathfrak{s} \subseteq A$ is called a \emph{cluster} (of $A$) if there is some subset $D \subset K$ which is a disc under the metric induced by the valuation $v : K^\times \to \zz$ such that $\mathfrak{s} = A \cap D$.  The \emph{depth} of a non-singleton cluster $\mathfrak{s}$ is the integer 
\begin{equation}
d(\mathfrak{s}) := \min_{z, z' \in \mathfrak{s}, \, z \neq z'} v(z - z').
\end{equation}

Given any two clusters $\mathfrak{s}_1, \mathfrak{s}_2 \subset A$, we denote by $\mathfrak{s}_1 \vee \mathfrak{s}_2$ the smallest cluster which contains both $\mathfrak{s}_1$ and $\mathfrak{s}_2$ (it is easy to show that such a cluster always exists).

\end{dfn}

In our context, the set $A$ is the set of branch points $\mathfrak{B} = \{\alpha_i, \beta_i\}_{0 \leq i \leq h}$.  For $1 \leq i \leq h$, we write $\mathfrak{s}_i$ for the smallest cluster containing both $\alpha_i$ and $\beta_i$ (noting that we have $\alpha_i, \beta_i \in K$).

\begin{dfn} \label{dfn ubereven}

We say that a cluster $\mathfrak{s} \subset \mathcal{B}$ is \emph{\"{u}bereven} if it satisfies the property that, for each cluster $\mathfrak{c}$ which is not the disjoint union of $\geq 2$ even-cardinality clusters, we have 
\begin{equation} \label{eq ubereven}
d(\mathfrak{s}) + d(\mathfrak{c}) - 2d(\mathfrak{s} \vee \mathfrak{c}) > \tfrac{pv(p)}{p-1}.
\end{equation}
(In particular, an \"{u}bereven cluster itself is the disjoint union of $\geq 2$ even-cardinality clusters.)

\end{dfn}

From now on, let $\mathfrak{C}$ denote the set of non-singleton clusters of $\mathcal{B}$ which either are \"{u}bereven or are not the disjoint union of $\geq 2$ even-cardinality clusters, and let $\mathfrak{C}_0 \subseteq \mathfrak{C}$ be the subset consisting of those clusters with even cardinality.

\subsection{Main result} \label{sec intro main result}

We are now ready to present our main result, which describes a symplectic automorphism of the $\ell$-adic Tate module $T_\ell(J)$ which is in the image of inertia.

\begin{thm} \label{thm main}

Let $C$ be a split degenerate superelliptic curve over $K$ which is determined by an equation of the form (\ref{eq split degenerate superelliptic}).  For any subset $\mathfrak{E} \subseteq \mathfrak{C}$, let $u_{\mathfrak{E}} \in \{0, \dots, \#\mathfrak{E}\}$ denote the number of clusters in $\mathfrak{E}$ which are \"{u}bereven.  Given a cluster $\mathfrak{s} \in \mathfrak{C}_0$, among all clusters $\mathfrak{c} \in \mathfrak{C}$ satisfying $d(\mathfrak{s}) - d(\mathfrak{s} \vee \mathfrak{c}) > (1 - u_{\{\mathfrak{s}\}})\frac{pv(p)}{p-1}$, let $\mathfrak{s}'$ be one which minimizes the value $d(\mathfrak{s}) + d(\mathfrak{c}) - 2d(\mathfrak{s} \vee \mathfrak{c}) - (2 - u_{\{\mathfrak{s}, \mathfrak{c}\}})\frac{pv(p)}{p-1}$ (such a cluster exists and is unique by \Cref{cor uniqueness used in main theorem} below).  Let 
\begin{equation}
m_{\mathfrak{s}} = d(\mathfrak{s}) + d(\mathfrak{s}') - 2d(\mathfrak{s} \vee \mathfrak{s}') - (2 - u_{\{\mathfrak{s}, \mathfrak{s}'\}})\tfrac{pv(p)}{p-1}.
\end{equation}

Choose a prime $\ell$.  Let $\sigma \in I_K$ be a Galois automorphism which, for each integer $m \geq 0$, fixes all $m$th roots in $\bar{K}$ of elements of $K$ of valuation $0$ and maps any $m$th root $\pi^{1/m}$ of any uniformizer $\pi$ of $K$ to $\zeta_m \pi^{1/m}$.  There is a $\zz[\zeta_p]_\ell$-submodule $\mathbb{T}_\ell \subset T_\ell(J)$, which coincides with the submodule $T_\ell(J)^{I_K}$ of elements fixed by inertia in the case that $\ell$ is not the residue characteristic, with a $\zz[\zeta_p]_\ell$-basis $\{\mathfrak{v}_1, \dots, \mathfrak{v}_h\}$ such that the action of $\sigma$ on $T_\ell(J)$ is given by 
\begin{equation} \label{eq main product of transvections}
\rho_\ell(\sigma) = \prod_{\mathfrak{s} \in \mathfrak{C}_0} \Big(\prod_{0 \leq n \leq p - 1} t_{\zeta_p^n \mathfrak{w}_{\mathfrak{s}}}^{m_{\mathfrak{s}}}\Big).
\end{equation}
For the formula above, we define $\mathfrak{w}_{\mathfrak{s}} = \sum_i \mathfrak{v}_i \in T_\ell(J)^{I_K}$, with the sum taken over all $i$ such that we have $\mathfrak{s}_i \subseteq \mathfrak{s}$ and do not have $\mathfrak{s}_i \subsetneq \mathfrak{s}_l \subseteq \mathfrak{s}$ for any $l \neq i$.

Moreover, if $\ell = p$, the modulo-$p$ image of the element $(\zeta_p + 2\zeta_p^2 + \dots + (p - 1)\zeta_p^{p - 1})\mathfrak{v}_i \in \mathbb{T}_\ell \subset T_\ell(J)$ is the $p$-torsion point of $J(K)$ represented by a multiple\footnote{\Cref{cor integral combination of v_i's} below specifies that the integer multiplier is $1$ when $\zeta_p \in \Aut(C)$ is defined as in \S\ref{sec background jacobians}.} of the divisor $\sum_j m_j ((\alpha_j, 0) - (\beta_j, 0)) \in \Div_0(C)$, where the sum is taken over all indices $j$ satisfying $d(\mathfrak{s}_i) - d(\mathfrak{s}_i \vee \mathfrak{s}_j) \leq \frac{pv(p)}{p-1}$ and $m_j$ is the exponent appearing in (\ref{eq split degenerate superelliptic}).

\end{thm}

\begin{rmk} \label{rmk g exists}

One sees that an automorphism $\sigma \in I_K$ with the properties specified in the statement of \Cref{thm main} exists by the theory of Kummer extensions on observing that the multiplicative group $K^\times$ is generated over its subgroup of valuation-$0$ elements by $\pi$.  If $\ell$ is not the residue characteristic of $K$, the image of the element $\sigma$ in the maximal pro-$\ell$ quotient of $I_K$ topologically generates this quotient, and the image $\rho_\ell(\sigma) \in \Aut(T_\ell(J))$ generates the $\ell$-adic image of inertia.

\end{rmk}

\begin{rmk} \label{rmk commuting transvections}

It follows from \Cref{rmk transvections on toric part} below that the pairing $e_\ell$ restricted to the submodule $\mathbb{T}_\ell \subset T_\ell(J)$ is trivial.  One checks as an elementary exercise that for any pair of elements $\mathfrak{w}, \mathfrak{w}' \in T_\ell(J)$, the transvections $t_{\mathfrak{w}}$ and $t_{\mathfrak{w}'}$ commute if and only if $e_\ell(\mathfrak{w}, \mathfrak{w}')$ is trivial.  It follows that the transvections in the product in (\ref{eq main product of transvections}) commute, and so that product as written is well defined.

\end{rmk}

\begin{rmk} \label{rmk cases of main theorem}

In particular cases, the conclusion of the above theorem becomes considerably simpler.  We point out several of these below.  For these examples, we define the \emph{parent cluster} of a cluster $\mathfrak{s}$ to be the smallest cluster properly containing $\mathfrak{s}$, if one exists; in this situation, we write $\delta(\mathfrak{s})$ for the depth of $\mathfrak{s}$ minus the depth of its parent.

\begin{enumerate}[(a)]

\item Suppose that $p$ is not the residue characteristic of $K$.  Then all clusters which are the disjoint union of $\geq 2$ even-cardinality clusters are \"{u}bereven; the set $\mathfrak{C}$ consists of all non-singleton clusters; and one can show without much difficulty that for each cluster $\mathfrak{s} \in \mathfrak{C}_0$, the cluster $\mathfrak{s}'$ is just the parent of $\mathfrak{s}$.  Then the formula for $m_{\mathfrak{s}}$ becomes much simpler: we get $m_{\mathfrak{s}} = \delta(\mathfrak{s})$.

\item Suppose instead that no cluster of $\mathcal{B}$ is a union of $\geq 2$ even-cardinality clusters, so that in particular there are no \"{u}bereven clusters.  Then one sees using \Cref{prop dictionary}(c) below that the set $\mathfrak{C}_0$ consists only of the clusters $\mathfrak{s}_1, \dots, \mathfrak{s}_h$, which comprise all of the even-cardinality clusters.  Moreover, as in part (a), for $1 \leq i \leq h$, the cluster $\mathfrak{s}'_i$ is just the parent of $\mathfrak{s}_i$.  The main conclusion of \Cref{thm main} now simplifies to stating that 
\begin{equation} \label{eq main product of transvections no ubereven}
\rho_\ell(\sigma) = \prod_{i = 1}^h \Big(\prod_{0 \leq n \leq p - 1} t_{\zeta_p^n \mathfrak{v}_i}^{r_i}\Big),
\end{equation}
with $r_i = \delta(\mathfrak{s}_i) - \frac{2pv(p)}{p-1}$.

\item Finally, we consider the hyperelliptic case (where $p = 2$).  Noting that (as an elementary property of transvections) for any $v \in T_\ell(J)$, we have $t_{-v} = t_v$, we may rewrite (\ref{eq main product of transvections}) as 
\begin{equation}
\rho_\ell(\sigma) = \prod_{\mathfrak{s} \in \mathfrak{C}_0} t_{\mathfrak{w}_{\mathfrak{s}}}^{2m_{\mathfrak{s}}}.
\end{equation}
Moreover, if we assume that the residue characteristic of $K$ is not $2$, then we even get 
\begin{equation}
\rho_\ell(\sigma) = \prod_{\mathfrak{s} \in \mathfrak{C}_0} t_{\mathfrak{w}_{\mathfrak{s}}}^{2\delta(\mathfrak{s}) - (4 - 2u_{\mathfrak{s}})v(2)}, 
\end{equation}
where $u_{\mathfrak{s}}$ is now just the number of clusters among $\mathfrak{s}$ and its parent $\mathfrak{s}'$ which are themselves the disjoint union of $\geq 2$ even-cardinality clusters.

If we further assume that $\ell = 2$, the statement of \Cref{thm main}(a) says that the image of each element $\mathfrak{v}_i \in T_2(J)$ is represented by the divisor $(\alpha_i, 0) - (\beta_i, 0) \in \Div_0(C)$.  The result for this last special case can also be shown from the results of \cite[\S2]{yelton2021boundedness}, although the analog appearing as \cite[Proposition 2.6]{yelton2021boundedness} applies not to the situation of split degenerate reduction (unless $g = 1$) but to the case of only one proper non-singleton cluster, that cluster having cardinality $2$.

\item We finally mention that if we further assume that $C : y^2 = (x - \alpha_0) (x - \alpha_1) (x - \beta_1)$ is an elliptic curve with split multiplicative reduction over $K$, then the only cluster in $\mathfrak{C}_0$ is $\mathfrak{s}_1 = \{\alpha_1, \beta_1\}$.  One may compute using \cite[Proposition 1.7(b)]{silverman2009arithmetic} that the $j$-invariant $j(C)$ satisfies $v(j(C)) = 8v(2) - 2\delta(\mathfrak{s}_1)$; in particular, as the multiplicative reduction property implies $v(j(C)) < 0$, we get $\delta(\mathfrak{s}_1) > 4v(2)$.  (One can even show this independently by applying \Cref{thm disjointness of neighborhoods} to this special case.)  Now we simply get that $\mathfrak{v}_1$ reduces modulo $2$ to the order-$2$ divisor class represented by $(\alpha_1, 0) - (\beta_1, 0) \in \Div_0(C)$ and 
\begin{equation}
\rho_\ell(\sigma) = t_{\mathfrak{v}_1}^{2\delta(\mathfrak{s}_1) - 8v(2)} = t_{\mathfrak{v}_1}^{-j(C)}.
\end{equation}
Apart from the explicit description of the modulo-$2$ image of $\mathfrak{v}_i$, this can be more or less deduced from \cite[Exercise 5.13(b)]{silverman1994advanced}.  But knowing the modulo-$2$ image of $\mathfrak{v}_i$ allows us to apply this result to a global setting using the methods of \cite[\S3,4]{yelton2021boundedness} and for instance to strengthen the result given in Example 4.3 of that paper (in which $\alpha_0 := 1, \alpha_1 := 0, \beta_1 := \lambda$ lie in a global ground field $L$) to one that allows for either of the prime ideals $\mathfrak{p}_1 \mid (\lambda)$ and $\mathfrak{p}_2 \mid (\lambda - 1)$ of $\mathcal{O}_L$ to have residue characteristic $2$.  We expect that one may be able to use \Cref{thm main} to pave the way toward similar results in a global setting for $g \geq 2$ or $p \geq 3$.

\end{enumerate}

\end{rmk}

\subsection{Outline of the paper}

In \S\ref{sec background}, we provide the necessary background on non-archimedean uniformization of split degenerate curves (as a subset $\Omega \subset \proj_K^1$ modulo the action of a subgroup $\Gamma < \PGL_2(K)$ with certain properties) and of their Jacobians (as a torus $T := \Hom(\Gamma, K^\times)$ modulo a lattice $\Lambda := \{c_\gamma\}_{\gamma \in \Gamma}$ generated by certain characters $c_\gamma$ defined there), emphasizing particular properties that hold when the curve is superelliptic.  We also discuss the dual $J^\vee \cong T^\vee / \Lambda^\vee$ of such a Jacobian $J \cong T / \Lambda$ as well as the Weil pairings of $J$ with $J^\vee$, culminating in an explicit description of the Weil pairing given by \Cref{prop explicit Weil pairing}.  In \S\ref{sec berk} we develop the necessary background theory on viewing the set $\mathcal{B}$ through certain subspaces of the Berkovich projective line $\Berk$ and show that the cluster $\mathfrak{s}'$ in the statement of \Cref{thm main} is well defined.  The goal of \S\ref{sec period matrix} is then to build on the set-up established in \S\ref{sec berk} to get an explicit formula (given as \Cref{thm values of pairing}) for the pairing $(\gamma, \gamma') \mapsto v(c_\gamma(\gamma'))$; this is done with respect to a convenient basis and therefore provides the valuations of the entries of a \emph{period matrix} of the lattice $\Lambda$ defining the Jacobian.  We then proceed in \S\ref{sec monodromy} to associate unipotent operators in $\Aut(T_\ell(J))$ to pairings of $\Lambda \otimes \zz_\ell$ with $\Lambda^\vee \otimes \zz_\ell$, showing (as \Cref{prop monodromy pairing}) that the pairing associated to $\rho_\ell(\sigma)$ is the same as $(\gamma, \gamma') \mapsto v(c_\gamma(\gamma'))$, and providing (as \Cref{prop explicit transvection}) formulas for the pairings associated to transvections with respect to certain elements of $T_\ell(J)$.  In \S\ref{sec proof of main}, we use our formulas for pairings established in \S\ref{sec period matrix} and \S\ref{sec monodromy} to show that the automorphism $\rho_\ell(\sigma)$ is equal to the product of transvections claimed by \Cref{thm main}.  We finish the paper with \S\ref{sec torsion}, which strengthens the results of \cite[\S1]{van1983note} and proves the last claim of \Cref{thm main} (concerning each generator $\mathfrak{v}_i$ in the $\ell = p$ case).

\subsection{Acknowledgments}

The author would like to thank Yaacov Kopeliovich for useful discussions on non-archimedean uniformization of Jacobians and in particular for pointing him towards the article \cite{van1983note}, enabling him to produce the results of \S\ref{sec torsion}.

\section{Background on non-archimedean uniformization of curves and their Jacobians} \label{sec background}

This section lays out all of the necessary background information on the uniformization of curves in the non-archimedean setting and the construction of their Jacobians as analytic tori, as well more explicit background results in the special case of superelliptic curves.  Our main sources of the theory of non-archimedean uniformization for general curves (with split degenerate reduction) and their Jacobians are \cite[Chapters III, IV, VI]{gerritzen2006schottky} and \cite[\S6, 7]{papikian2013non}, while our main source regarding the superelliptic case \cite{van1982galois} and \cite{van1983note}; many more details can be found in these sources.  (Our definition of theta functions follows the formula given in \cite[\S II.2]{gerritzen2006schottky} and \cite{van1983note} rather than in \cite{drinfeld1973periods} or \cite[\S7]{papikian2013non}.)

\subsection{Schottky groups and non-archimedean uniformization of split degenerate curves} \label{sec background Schottky}

Mumford showed in his groundbreaking paper \cite{mumford1972analytic} that any curve $C / K$ (not necessarily superelliptic) of genus $g \geq 1$ can be realized as a quotient of a certain subset $\Omega \subset \proj_K^1$ by the action of a free subgroup $\Gamma < \PGL_2(K)$ of $g$ generators via fractional linear transformations if and only if the the curve $C$ has split degenerate reduction.  The free subgroup $\Gamma < \PGL_2(K)$ must act discontinuously on $\proj_K^1$ (\textit{i.e.} the set of limit points under its action must not coincide with all of $\proj_K^1$), and the subset $\Omega \subset \proj_K^1$ such that $C$ can be uniformized as the quotient $\Omega / \Gamma$ coincides with the set of non-limit points.  Such a subgroup $\Gamma < \PGL_2(K)$ is called a \emph{Schottky group}.  This main result on non-archimedean uniformization of curves is given as \cite[Theorem 4.20]{mumford1972analytic} and \cite[Theorems III.2.2, III.2.12.2, and IV.3.10]{gerritzen2006schottky}.  We comment that in the special case of $g = 1$, after applying an appropriate automorphism of $\proj_K^1$ we get $\Omega = \proj_K^1 \smallsetminus \{0, \infty\} = K^\times$ and that the Schottky group $\Gamma$ is generated by the fractional linear transformation $z \mapsto qz$ for some element $q \in K^\times$ of positive valuation, and thus we recover the Tate uniformization $C \cong K^\times / \langle q \rangle$ established in \cite{tate1995review}.

It is shown in \cite[\S9.2]{gerritzen2006schottky} and \cite[\S1]{van1983non} (for the $p = 2$ case) and in \cite[\S2]{van1982galois} (for general $p$) that given a prime $p$ and a split degenerate curve $C / K$ of genus $g = (p - 1)h$ realized as such a quotient $\Omega / \Gamma$, the curve $C$ is superelliptic and a degree-$p$ cover of $\proj_K^1$ if and only if $\Gamma$ is normally contained in a larger subgroup $\Gamma_0 < \PGL_2(K)$ generated by $h + 1$ elements $s_0, \dots, s_h$ whose only relations are $s_0^p = \dots = s_h^p = 1$.  In this situation, the Schottky group is given by 
\begin{equation} \label{eq Gamma}
\Gamma = \langle \gamma_{i, 0, n} := s_0^{n - 1} s_i s_0^{-n} \rangle_{1 \leq i \leq h, \, 1 \leq n \leq p - 1}.
\end{equation}
It is an elementary group-theoretic exercise to show from these definitions that the generators $\gamma_{i, 0, n}$ of $\Gamma$ satisfy no group relations and that we have $\Gamma \lhd \Gamma_0$ and $[\Gamma : \Gamma_0] = p$.  (This implies in particular that $\Omega$ is the set of non-limit points under the action of $\Gamma_0$ as well.)  In fact, it turns out that we have $\Omega / \Gamma_0 \cong \proj_K^1$ so that the natural surjection $\Omega / \Gamma \twoheadrightarrow \Omega / \Gamma_0$ is just the degree-$p$ covering map $C \to \proj_K^1$.

As each element $s_i \in \Gamma_0 \smallsetminus \Gamma$ has order $p$ and normalizes $\Gamma$, it induces an order-$p$ automorphism on the quotient $C \cong \Omega / \Gamma$ which respects the covering map $C \to \proj_K^1 \cong \Omega / \Gamma_0$, by sending the image modulo $\Gamma$ of each point $z \in \Omega$ to the image modulo $\Gamma$ of $s_i(z) \in \Omega$.  It follow from the easily-verified fact that we have $s_j s_i^{-1} \in \Gamma$ that this automorphism of $C$ does not depend on the index $i$.  The group of such automorphisms of $C$ is generated by the order-$p$ automorphism given by $(x, y) \mapsto (x, \zeta_p y)$; let us now assume that the primitive root $\zeta_p$ has been chosen such that each $s_i$ induces this automorphism.  It it then shown as \cite[Proposition 3.2]{van1982galois} that for $0 \leq i \leq h$, there is a matrix in $\GL_2(K)$ representing the element $s_i \in \PGL_2(K)$ which is similar to $\begin{bsmallmatrix} \zeta_p^{n_i} & 0 \\ 0 & 1 \end{bsmallmatrix}$, where $n_i$ is an integer such that $m_i n_i \equiv 1$ (mod $p$) (with $m_i$ the exponent appearing in (\ref{eq split degenerate superelliptic}).

Each order-$p$ element $s_i \in \PGL_2(K)$ fixes exactly $2$ points of $\proj_K^1$, which we denote by $a_i$ and $b_i$.  One can show that we have $a_i, b_i \in \Omega$ for $0 \leq i \leq h$ (see for instance \cite[Corollary 4.3]{yelton2024branch2}).  Writing $S = \{a_0, b_0, \dots, a_h, b_h\} \subset \Omega$, it is easy to verify that the set-theoretic image of $S$ modulo the action of the group $\Gamma_0$ coincides with the set of branch points $\mathcal{B} \subset \proj_K^1 \cong \Omega / \Gamma_0$.  By \cite[Proposition 3.1(a)]{van1982galois}, for $0 \leq i \leq h$, each of the branch points $\alpha_i, \beta_i \in \mathcal{B}$ as defined in \S\ref{sec intro superelliptic} is respectively the image of the element $a_i, b_i \in \Omega$ after possibly permuting the indices and permuting $a_i, b_i$ for a fixed index $i$.

\subsection{Non-archimedean uniformization of Jacobians of split degenerate curves} \label{sec background jacobians}

The discussion in this and the next subsection applies to any split degenerate curve $C / K$, regardless of whether it is superelliptic.  In order to analytically construct the Jacobian of a split degenerate curve, it is necessary to define a certain kind of automorphic form on $\Omega$; proofs for the properties of these automorphic forms stated in the following discussion can be found in \cite[\S II.2.3]{gerritzen2006schottky}.

As in \S\ref{sec background Schottky}, let $\Gamma < \PGL_2(K)$ be the Schottky group used to uniformize the split degenerate curve $C$, and write $\Omega \subset \proj_K^1$ for the subset of limit points of the action of $\Gamma$.  For any choice of $a, b \in \Omega$ not both in the $\Gamma$-orbit of $\infty$, we define the \emph{theta function} $\Theta_{a, b}$ (with respect to $a, b \in \Omega$ and to $\Gamma < \PGL_2(K)$) as 
\begin{equation*}
\Theta_{a, b}(z) = \prod_{\gamma \in \Gamma} \frac{z - \gamma(a)}{z - \gamma(b)}.
\end{equation*}
We adopt the convention that if exactly one of the terms in the numerator (resp. denominator) is $\infty$, then the numerator (resp. denominator) is replaced by $1$ and that if the denominator comes out to $0$, then the infinite product equals $\infty \in \proj_K^1$.  It is easy to see that as long as we have $b \notin \Gamma a$, the set of zeros (resp. poles) of $\Theta_{a, b}$ coincides with $\{\gamma(a_i)\}_{\gamma \in \Gamma}$ (resp. $\{\gamma(b_i)\}_{\gamma \in \Gamma}$), whereas if on the other hand we have $b \in \Gamma a$, then the function $\Theta_{a, b}$ has no zeros or poles.

One can show that any theta function with respect to $\Gamma$ is well-defined and meromorphic on $\Omega$ and moreover that it has an automorphy property: for any $\gamma \in \Gamma$, there is an element $c_{a, b}(\gamma) \in K^\times$ such that we have 
\begin{equation} \label{eq automorphy}
\Theta_{a, b}(\gamma(z)) = c_{a, b}(\gamma)\Theta_{a, b}(z).
\end{equation}
(We note that the formula for the automorphy factor $c_{a, b} \in K^\times$ in most sources defines it as the multiplicative inverse of the factor we have defined above.)  It is clear that the mapping $\gamma \mapsto c_{a, b}(\gamma)$ respects group multiplication and so defines a homomorphism $c_{a, b} : \Gamma \to K^\times$.  Since the target group is abelian, in fact the homomorphism $c_{a, b}$ factors through the abelianization $\bar{\Gamma} \cong \zz^g$ of $\Gamma$; abusing notation slightly, we denote the resulting homomorphism again by $c_{a, b} : \bar{\Gamma} \to K^\times$.

Moreover, for any $\gamma \in \Gamma$, one can show that the function $c_{a, \gamma(a)}$ does not depend on the choice of an element $a \in \Omega$; we denote this function by $c_\gamma$.  It is clear that we have $c_{\gamma \gamma'} \equiv c_\gamma c_{\gamma'}$ for any $\gamma, \gamma' \in \Gamma$; therefore, the set $\Lambda := \{c_\gamma \ | \ \gamma \in \Gamma\}$ forms a group of homomorphisms from $\bar{\Gamma}$ to $K^\times$.  Below we will often abuse notation slightly by writing $c_\gamma$ for the homomorphism $\bar{\Gamma} \to K^\times$ defined as above with respect to an element of $\Gamma$ but in a context where $\gamma$ refers not to that element of $\Gamma$ but to its reduction in $\bar{\Gamma}$.

The full homomorphism group $\Hom(\bar{\Gamma}, K^\times)$ is isomorphic to $(K^\times)^g$ and so has the structure of a split analytic torus over $K$ which we denote by $T$ (see \cite[Corollary 2 to Theorem 4]{drinfeld1973periods} or \cite[\S VI.1]{gerritzen2006schottky}).  The subgroup $\Lambda < T$ defined above is a rank-$g$ lattice.  The quotient $T / \Lambda$ can be identified with (the analytification of) the Jacobian $J / K$ by defining an $K$-analytic Abel-Jacobi map as follows.  Any degree-$0$ $K$-divisor can be written as $\sum_{i = 1}^s ((P_i) - (Q_i))$ for some $s \geq 0$ and some (not necessarily distinct) points $P_i, Q_i \in C(\bar{K})$ such that the above formal sum is fixed by $\Gal(\bar{K} / K)$.  For $1 \leq i \leq s$, choose elements $p_i, q_i \in \Omega(\bar{K}) := \Omega \cup (\bar{K} \smallsetminus K)$ mapping respectively to $P_i, Q_i \in C(\bar{K}) \cong \Omega(\bar{K}) / \Gamma$ under reduction modulo the action of $\Gamma$.  Then we define the $K$-analytic Abel-Jacobi map 
\begin{equation} \label{eq Abel-Jacobi}
\mathrm{AJ} : \Div_0(C) \to T, \ \ \sum_{i = 1}^s ((P_i) - (Q_i)) \mapsto \prod_{i = 1}^s c_{p_i, q_i}.
\end{equation}
It is shown in \cite[\S3,4]{drinfeld1973periods} and \cite[\S VI.1]{gerritzen2006schottky} that this map $\mathrm{AJ}$ is surjective (and sends $K$-divisors in $\Div_0(C)$ to $K$-points of $T$) and that the image of the subgroup $\Prin(C) < \Div_0(C)$ of principal divisors under $\mathrm{AJ}$ coincides with the lattice $\Lambda < T$.  In this way, the map $\mathrm{AJ}$ induces an isomorphism $\Pic_0(C) \stackrel{\sim}{\to} T / \Lambda$ and thus identifies the quotient $T / \Lambda$ with the Jacobian $J$ of $C$.

Although the above construction describes the Jacobian variety $J$ as a multiplicative group, we will use additive notation for points of $J$.

As in \S\ref{sec intro superelliptic}, we denote by $\zeta_p \in \End(J)$ the automorphism of $J$ induced from the automorphism of $C$ given by $(x, y) \mapsto (x, \zeta_p y)$.  If we retain our assumption from \S\ref{sec background Schottky} that $\zeta_p$ is chosen so that this automorphism of $C$ is induced by the element $s_i \in \Gamma_0$ for some (any) index $i$, then we may explicitly describe the automorphism $\zeta_p \in \End(J)$ as follows.

\begin{prop} \label{prop automorphism zeta_p}

For some (any) index $i$, the automorphism $\zeta_p \in \End(J)$ is the one induced by the automorphism of $T$ which acts on each homomorphism $\Gamma \to K^\times$ in $T(K)$ by composing it with the conjugation map $^{s_i}(\cdot) : \Gamma \to \Gamma$ given by $\gamma \mapsto s_i \gamma s_i^{-1}$.

\end{prop}

\begin{proof}

Since the group $J(K)$ is generated by equivalence classes of divisors of the form $(P) - (Q) \in \Div_0(C)$, it follows from the above construction of the Jacobian that it suffices to verify the claimed description of the action of $\zeta_p \in \End(J)$ on images modulo $\Lambda$ of elements of $T$ of the form $c_{a, b}$ for $a, b \in \Omega$.  By considering the pullback of the divisor $(P) - (Q)$ via the automorphism $(x, y) \mapsto (x, \zeta_p y)$ of $C$, the automorphism $\zeta_p$ sends $c_{a, b}$ to $c_{s_i^{-1}(a), s_i^{-1}(b)}$.

Now, using \Cref{lemma automorphy van steen} below to get a constant $c \in K^\times$, we get 
\begin{equation}
c_{s_i^{-1}(a), s_i^{-1}(b)}(\gamma) \equiv \frac{\Theta_{s_i^{-1}(a), s_i^{-1}(b)}(\gamma(z))}{\Theta_{s_i^{-1}(a), s_i^{-1}(b)}(z)} \equiv \frac{c^{-1}\Theta_{a, b}([s_i \gamma](z))}{c^{-1}\Theta_{a, b}(s_i(z))} \equiv \frac{\Theta_{a, b}([s_i \gamma s_i^{-1}](s_i(z)))}{\Theta_{a, b}(s_i(z))} \equiv c_{a, b}(s_i \gamma s_i^{-1}).
\end{equation}

\end{proof}

\subsection{Duals and Weil pairings for split toric Jacobians} \label{sec background pairings}

Let $T^\vee$ denote the group $\Hom(\Lambda, \ggr_m)$ of characters of the lattice $\Lambda$, and let $\Lambda^\vee$ denote the group $\Hom(T, \ggr_m)$ of characters of the torus $T$, where \emph{characters} are understood to be morphisms of group varieties; we may view both $T^\vee$ and $\Lambda^\vee$ as rigid analytic group varieties over $K$.  There is an obvious map $\Lambda^\vee \to T^\vee$ given by restriction of characters to $\Lambda \subset T$.  One can see, as in the discussion in \cite[\S2]{bosch1991degenerating}, that this realizes $\Lambda^\vee$ as a lattice in $T^\vee$ of maximal rank, and it is shown as \cite[Theorem 2.1]{bosch1991degenerating} that the resulting quotient $T^\vee / \Lambda^\vee$ may be identified with the dual $J^\vee$ of the abelian variety $J$.

\begin{rmk} \label{rmk Lambda dual is barGamma}

Let $L / K$ be an algebraic extension.  Any element $\gamma \in \bar{\Gamma}$ induces an element $\ev_{\gamma} \in \Lambda^\vee(L) = \Hom(T_L, \ggr_{m, L})$ mapping a point $\mu \in T(L)$, viewed as a homomorphism $\chi_\mu : \bar{\Gamma} \to L^\times$, to its evaluation $\chi_\mu(\gamma) \in L^\times$.  One checks that the mapping $\gamma \mapsto \ev_{\gamma}$ is injective and surjective and thus allows us to identify the lattice $\Lambda^\vee_L$ with the free $\zz$-module $\bar{\Gamma}$, an identification that we will make freely below.

\end{rmk}

For any algebraic extension $L / K$ and an $L$-point $\mu$ of $T^\vee$ (resp. ($K$-)point of $\Lambda^\vee$), we write $\chi_\mu : \Lambda \to L^\times$ (resp. $\chi_\mu : T \to K^\times$) for the corresponding character.  Thanks to \Cref{rmk Lambda dual is barGamma}, each $L$-point $\lambda$ of $T$ (resp. ($K$-)point of $\Lambda$) can similarly be viewed as a character on $\Lambda^\vee$ (resp. $T^\vee$) taking values in $L^\times$ (resp. $K^\times$), and we write $\chi_\lambda$ for this character.  (Note that we have $\chi_\lambda(\mu) = \chi_\mu(\lambda)$ for any $\lambda \in \Lambda$ and $\mu \in \Lambda^\vee$.)

Let $m \geq 1$ be an integer.  By applying the functor $\Hom(\zz / m\zz, \cdot)$ to the short exact sequence 
\begin{equation}
1 \to \Lambda \to T \to J \to 1
\end{equation}
 and observing that $\Lambda$, being a free $\zz$-module, has trivial torsion, we get an injection $T[m] \hookrightarrow J[m]$.  Similarly, we get an injection $T^\vee[m] \hookrightarrow J^\vee[m]$.  We write $\bar{\mathbb{T}}_m$ and $\bar{\mathbb{T}}^\vee_m$ for the images of $T[m]$ and $T^\vee[m]$ under these respective injections, making them respectively rank-$g$ $\zz/m\zz$-submodules of the rank-$2g$ $\zz/m\zz$-modules $J[m]$ and $J^\vee[m]$.  Now letting $m = \ell^n$ be a power of a prime $\ell$ and taking inverse limits gives us injections of free $\zz_\ell$-modules 
\begin{equation} \label{eq T_ell}
 \mathbb{T}_\ell := T_\ell(T) := \lim_{\leftarrow n} T[\ell^n] \hookrightarrow T_\ell(J), \ \ \ \ \mathbb{T}_\ell^\vee := T_\ell(T^\vee) := \lim_{\leftarrow n} T^\vee[\ell^n] \hookrightarrow T_\ell(J^\vee).
\end{equation}

\begin{rmk} \label{rmk submodule fixed by inertia}

Suppose that $m$ is prime to the residue characteristic of $K$.  Then since all $m$-power roots of unity in $\bar{K}$ are fixed by the inertia subgroup $I_K$, the submodule $\bar{\mathbb{T}}_m \subset J[m]$ is contained in the submodule $J[m]^{I_K} \subset J[m]$ fixed by the action of inertia.  On taking inverse limits, it follows that we have the similar inclusion $\mathbb{T}_\ell \subseteq T_\ell(J)^{I_K}$ for any $\ell$ different from the residue characteristic.

One can show from the fact that $J$ has split toric reduction, using Grothendieck's Orthogonality Theorem (\cite[Th\'{e}or\`{e}me 2.4]{grothendieck1972modeles}), that the ranks of $T_\ell(J)^{I_K}$ and $T_\ell(J^\vee)^{I_K}$ must both equal $g$, using the following argument.  The fact that the abelian variety $J$ is split toric implies that its dual abelian variety $J^\vee$ is as well.  Grothendieck's Orthogonality Theorem states that the submodule $T_\ell(J)^{\mathrm{t}} \subset T_\ell(J)^{I_K}$ mapping to the toric part of the identity component of the special fiber of the N\'{e}ron model of $J$ and the submodule $T_\ell(J^\vee)^{I_K} \subset T_\ell(J^\vee)$ are orthogonal complements to each other with respect to the Weil pairing $\mathfrak{e}_\ell : T_\ell(J) \times T_\ell(J^\vee) \to T_\ell(\bar{K}^\times)$.  As the rank of $T_\ell(J)^{\mathrm{t}}$ is $g$, the rank of its orthogonal complement $T_\ell(J^\vee)^{I_K}$ is $2g - g = g$, and by a symmetric argument, the rank of $T_\ell(J)^{I_K}$ is $g$ as well.  We may therefore make the identifications $\mathbb{T}_\ell = T_\ell(J)^{I_K} = T_\ell(J)^{\mathrm{t}}$ and $\mathbb{T}_\ell^\vee = T_\ell(J^\vee)^{I_n} = T_\ell(J^\vee)^{\mathrm{t}}$.  We then have that $T_\ell(T)$ and $T_\ell(T^\vee)$ are orthogonal complements to each other under the pairing $\mathfrak{e}_\ell$; this is implicit in \Cref{prop explicit Weil pairing} below, which generalizes this fact to the case that $\ell$ is the residue characteristic.

\end{rmk}

For any integer $m \geq 1$, there is a homomorphism $\bar{\phi}_m: J[m] \to \Lambda \otimes \zz/m\zz$ given by lifting an element $z \in J[m]$ to a point $\tilde{z} \in T_{\bar{K}}$ and mapping it to the reduction of $\tilde{z}^m \in \Lambda$ modulo $m$th powers of elements of $\Lambda$.  This map $\bar{\phi}_m$ is clearly well defined, and its kernel coincides with $T[m] \subset J[m]$.  Meanwhile, we may make the identifications $T^\vee[m] = \Hom(\Lambda, \boldsymbol{\mu}_m) = \Hom(\Lambda \otimes \zz/m\zz, \boldsymbol{\mu}_m)$.  This allows us to define a pairing 
\begin{equation} \label{eq explicit Weil pairing mod m}
J[m] \times T^\vee[m] \to \boldsymbol{\mu}_m, \ \ (z, \xi) \mapsto (\chi_\xi \circ \bar{\phi}_m)(z),
\end{equation}
Now, letting $m = \ell^n$ and taking inverse limits, we get a homomorphism $\phi_\ell : T_\ell(J) \to \Lambda \otimes \zz_\ell$ whose kernel coincides with $\mathbb{T}_\ell$, an identification $\mathbb{T}_\ell = T_\ell(T^\vee) = \Hom(\Lambda \otimes \zz_\ell, T_\ell(\bar{K}^\times))$, and the pairing 
\begin{equation} \label{eq explicit Weil pairing}
T_\ell(J) \times T_\ell(T^\vee) \to T_\ell(\bar{K}^\times), \ \ (z, \xi) \mapsto (\chi_\xi \circ \phi_\ell)(z).
\end{equation}

The fact that the pairings in (\ref{eq explicit Weil pairing mod m}) and (\ref{eq explicit Weil pairing}) are respectively explicit formulations of certain restrictions of the Weil pairings $\bar{\mathfrak{e}}_m$ and $\mathfrak{e}_\ell$ is more or less asserted in \cite[Proof of Theorem 2.1]{coleman2000monodromy} (as well as in \cite[Proof of Theorem 5.8]{papikian2013non}) without justification provided; for the sake of completeness, it is given as \Cref{cor explicit Weil pairing} below.  The following proposition provides a formula for the mod-$m$ Weil pairing on all of $J[m] \times J^\vee[m]$.

\begin{prop} \label{prop explicit Weil pairing}

Given an integer $m \geq 1$, the mod-$m$ Weil pairing $\bar{\mathfrak{e}}_m : J[m] \times J^\vee[m] \to \boldsymbol{\mu}_m$ is given by the following formula.  Given an element $z \in J[m]$ (resp. $\xi \in J^\vee[m]$), choose a lifting $\tilde{z} \in T(\bar{K})$ (resp. $\tilde{\xi} \in T^\vee(\bar{K}) = \Hom(\Lambda, \bar{K}^\times)$) and note that we have $\tilde{z}^m \in \Lambda$ (resp. $\tilde{\xi}^m \in \Lambda^\vee = \Hom(T, K^\times)$).  Then we have 
\begin{equation} \label{eq Weil pairing formula}
\bar{\mathfrak{e}}_m(z, \xi) = \frac{\chi_{\tilde{\xi}}(\tilde{z}^m)}{\chi_{\tilde{\xi}^m}(\tilde{z})}.
\end{equation}

\end{prop}

\begin{proof}
In \cite[Proof of Theorem 2.1]{bosch1991degenerating}, in which the dual $J^\vee$ of the split toric abelian variety $J$ is constructed as $T^\vee / \Lambda^\vee$, any $L$-point $\xi \in J^\vee(L)$ (where $L / K$ is an algebraic extension) is identified with the line bundle $\mathcal{L}$ on $J$ given by the trivial line bundle $\tilde{\mathcal{L}} := T \times \aff^1$ of the torus $T$ modulo the action of $\Lambda$ given by 
\begin{equation} \label{eq action of Lambda for L}
\lambda: (\tilde{w}, a) \mapsto (\lambda \tilde{w}, \chi_{\tilde{\xi}}(\lambda) a), \ \ \ \ \lambda \in \Lambda, \ \tilde{w} \in T, \ a \in \aff^1,
\end{equation}
where $\tilde{\xi} \in T^\vee(L) = \Hom(\Lambda, L^\times)$ is a lifting of $\xi$ (one checks that choosing a different lifting $\tilde{\xi}'$ of $\xi$ leads to an isomorphic line bundle on $J$ using the fact that the quotient $\tilde{\xi}' / \tilde{\xi} \in \Lambda^\vee = \Hom(T, \ggr_m)$ yields the global section $\chi_{\tilde{\xi}' / \tilde{\xi}}$ of $T$).  Let $[m]$ denote the $m$th-power endomorphism on $T$ or on $J$.  We construct the pullback line bundle $[m]^*\mathcal{L}$ on $J$ by pulling back the trivial bundle $\tilde{\mathcal{L}} = T \times \aff^1$ via the map $[m] : T \to T$ and defining an action of $\Lambda$ on $[m]^*\tilde{\mathcal{L}}$ which is compatible under the $m$th power map $[m]$ with the action of $\Lambda$ on $\tilde{\mathcal{L}}$ defined above.  One checks that such an action has to be given by 
\begin{equation} \label{eq action of Lambda for pullback}
\lambda : (\tilde{w}, a) \mapsto (\lambda^m \tilde{w}, \chi_{\tilde{\xi}}(\lambda^m) a) = (\lambda^m \tilde{w}, \chi_{\tilde{\xi}}^m(\lambda) a), \ \ \ \ \lambda \in \Lambda, \ \tilde{w} \in T, \ a \in \aff^1;
\end{equation}
 indeed, this is the only action compatible with the pullback of the translation action of $\Lambda$ on $T$ via the map $[m]$.  The line bundle $[m]^*\mathcal{L}$ is then given by $T \times \aff^1$ modulo the action given by (\ref{eq action of Lambda for pullback}).

Now suppose that $\xi$ is $m$-torsion.  Then we have $\tilde{\xi}^m \in \Lambda^\vee = \Hom(T, K^\times)$, and the line bundle $[m]^*\mathcal{L}$ is trivial.  More precisely, there is an isomorphism of line bundles $J \times \aff^1 \stackrel{\sim}{\to} [m]^*\mathcal{L}$ induced by the isomorphism 
\begin{equation} \label{eq trivialization of pullback}
T \times \aff^1 \stackrel{\sim}{\to} [m]^*\tilde{\mathcal{L}}, \ (\tilde{w}, a) \mapsto (\tilde{w}, \chi_{\tilde{\xi}^m}(\tilde{w}) a)
\end{equation}
 and by carrying over the action in (\ref{eq action of Lambda for pullback}) of $\Lambda$ on $[m]^*\tilde{\mathcal{L}}$ via (the inverse of) this isomorphism to the action of $\Lambda$ on $T \times \aff^1$ given by 
\begin{equation} \label{eq action of Lambda trivial}
\lambda : (\tilde{w}, a) \mapsto (\lambda^m \tilde{w}, a), \ \ \ \ \lambda \in \Lambda, \ \tilde{w} \in T, \ a \in \aff^1.
\end{equation}
The situation is summarized in the below diagram (the arrows $\rightsquigarrow$ point to quotients by an action).

\begin{equation}
\begin{tikzcd} [column sep = {.05in, 1em}]
\Lambda \arrow[loop right, distance = .25in, "(\ref{eq action of Lambda trivial})", pos=.15] \arrow[d, equal] & T \times \aff^1 \arrow[rr, rightsquigarrow] \arrow[d, "(\ref{eq trivialization of pullback})"] &  & J \times \aff^1 \arrow[d, "\wr"] \\
\Lambda \arrow[loop right, distance = .25in, "(\ref{eq action of Lambda for pullback})", pos=.15] \arrow[d, "{[}m{]}"] & T \times \aff^1 \arrow[rr, rightsquigarrow] \arrow[d, "({[}m{]}{,}\,\text{id})"] &  & \left[m\right]^*\mathcal{L} \\
\Lambda \arrow[loop right, distance = .25in, "(\ref{eq action of Lambda for L})", pos=.15] & T \times \aff^1 \arrow[rr, rightsquigarrow] &  & \mathcal{L}
\end{tikzcd}
\end{equation}

Now one can readily verify that the line bundle $\mathcal{L}$ can be constructed directly from its pullback $[m]^*\mathcal{L}$ as the quotient of $[m]^*\mathcal{L}$ by the action of $J[m]$ given by $z : (w, a) \mapsto (w + z, \chi_{\tilde{\xi}}(\tilde{z}^m)a)$ for any $\tilde{z} \in T$ reducing to a point $z \in J[m]$ (and any $w \in J$ and $a \in \aff^1$): indeed, this gives us the quotient of $T \times \aff^1$ by the action of $[m]^{-1}(\Lambda)$ given by $\tilde{z} : (w, a) \mapsto (\tilde{z}^m w, \chi_{\tilde{\xi}}(\tilde{z}^m)a)$ for any $\tilde{z} \in [m]^{-1}(\Lambda) \subset T$, which is isomorphic to $\mathcal{L}$ as constructed using (\ref{eq action of Lambda for L}) via applying $[m]$.  Now, using the trivialization of $[m]^*\mathcal{L}$ given by (\ref{eq trivialization of pullback}), we may view $\mathcal{L}$ as the quotient of the trivial bundle $J \times \aff^1$ by 
\begin{equation}
z : (w, a) \mapsto (w + z, \chi_{\tilde{\xi}}(\tilde{z}^m) \chi_{\tilde{\xi}^{-m}}(\tilde{z}) a), \ \ \ \ z \in J[m], \ w \in J, \ a \in \aff^1.
\end{equation}
Now according to the construction in \cite[\S20]{mumford1974abelian}, the character $z \mapsto \chi_{\tilde{\xi}}(\tilde{z}^m) \chi_{\tilde{\xi}^{-m}}(\tilde{z})$ equals $\bar{\mathfrak{e}}_m(\cdot, \xi)$, thus proving the claimed formula.
\end{proof}

\begin{cor} \label{cor explicit Weil pairing}

The mod-$m$ Weil pairing $\bar{\mathfrak{e}}_m : J[m] \times J^\vee[m] \to \boldsymbol{\mu}_m$, restricted to $T^\vee[m] = \bar{\mathbb{T}}_m^\vee \subset J^\vee[m]$ in the second argument, is the same as the pairing in (\ref{eq explicit Weil pairing mod m}).

The $\ell$-adic Weil pairing $\mathfrak{e}_\ell$, restricted to $T_\ell(T^\vee) = \mathbb{T}_\ell^\vee \subset T_\ell(J^\vee)$ in the second argument, is the same as the pairing in (\ref{eq explicit Weil pairing}).

\end{cor}

\begin{proof}
Choose elements $z \in J[m]$ and $\xi \in T^\vee[m] = \bar{\mathbb{T}}_m \subset J^\vee[m]$; choose a lifting $\tilde{z} \in T$; and choose a lifting $\tilde{\xi} \in T^\vee$ which is $m$-torsion.  The first statement now follows immediately from the formula in (\ref{eq Weil pairing formula}) given by \Cref{prop explicit Weil pairing} on observing that the character $\chi_{\xi^m}$ is trivial.  The second statement clearly follows from the first on considering how the $\ell$-adic pairing $\mathfrak{e}_\ell$ is derived via inverse limits from the mod-$\ell^n$ pairings $\bar{\mathfrak{e}}_{\ell^n}$.
\end{proof}

\section{Clusters and the Berkovich projective line} \label{sec berk}

Given the completion $\cc_K$ of an algebraic closure of $K$, we write $v : \cc_K \to \rr$ for an extension of the valuation $v : K^\times \to \zz$.  Below when we speak of a \emph{disc} $D \subset \cc_K$, we mean that $D$ is a closed disc with respect to the metric induced by $v : \cc_K \to \rr$; in other words, $D = \{z \in \cc_K \ | \ v(z - c) \geq r\}$ for some center $c \in \cc_K$ and for some real number $r \in \rr$ which is the \emph{(logarithmic) radius} of $D$.  Given a disc $D \subset \cc_K$, we denote its (logarithmic) radius by $d(D)$.

\subsection{The Berkovich projective line and related notation} \label{sec berk berk}

The \emph{Berkovich projective line} $\Berk$ over the complete algebraically closed field $\cc_K$ is a type of rigid analytification of the projective line $\proj_{\cc_K}^1$ and is typically defined in terms of multiplicative seminorms on $\cc_K[x]$ as in \cite[\S1]{baker2008introduction} and \cite[\S6.1]{benedetto2019dynamics}.  Points of $\Berk$ are identified with multiplicative seminorms which are each classified as Type I, II, III, or IV.  For the purposes of this paper, as in \cite{yelton2024branch}, we may safely ignore points of Type IV and need only adopt a fairly rudimentary construction which does not directly involve seminorms.

\begin{dfn} \label{dfn berk}

Define the \textit{Berkovich projective line}, denoted $\Berk$, to be the topological space with points and topology given as follows.  The points of $\Berk$ are identified with  
\begin{enumerate}[(i)]
\item $\cc_K$-points of $\proj_{\cc_K}^1$, which we will call \emph{points of Type I}; and 
\item discs $D \subset \cc_K$; if $d(D) \in \qq$ (resp. $d(D) \notin \qq$), we call this a \emph{point of Type II} (resp. a \emph{point of Type III}).
\end{enumerate}
Let $\Hyp \subset \Berk$ be the subset consisting of the points of Type II or III.

A point of $\Berk$ which is identified with a point $z \in \cc_K \cup \{\infty\}$ (resp. a disc $D \subset \cc_K$) is denoted $\eta_z \in \Berk$ (resp. $\eta_D \in \Berk$).  Given two points $\eta = \eta_D, \eta' = \eta_{D'} \in \Hyp$, we write $\eta \vee \eta' \in \Hyp$ for the point corresponding to the smallest disc containing both $D$ and $D'$.  We define an order relation (denoted by $>$) on $\Hyp$ by decreeing that $\eta > \eta'$ if and only if we have $\eta = \eta_D$ and $\eta' = \eta_{D'}$ for discs $D \supsetneq D'$.  Note that for any $\eta, \eta' \in \Hyp$, the point $\eta \vee \eta'$ satisfies $\eta \vee \eta' > \eta, \eta'$ and is minimal (with respect to the ``greater than" relation) for this property.

We define a metric on $\Hyp$ given by the distance function 
\begin{equation*}
\delta : \Hyp \times \Hyp \to \rr
\end{equation*}
 as follows.  For points $\eta = \eta_D, \eta' = \eta_{D'} \in \Hyp$ satisfying $\eta > \eta'$, we let 
\begin{equation}
\delta(\eta, \eta') = d(D) - d(D'), 
\end{equation}
and for general $\eta, \eta' \in \Hyp$, we let 
\begin{equation}
\delta(\eta, \eta') = \delta(\eta, \eta \vee \eta') + \delta(\eta', \eta \vee \eta').
\end{equation}

We endow the subset $\Hyp \subset \Berk$ with the topology induced by the metric given by $\delta$, and we extend this to a topology on all of $\Berk$ in such a way that, for any $w \in \cc_K$, the sequence $\{\eta_{(w), i}\}_{i = 1, 2, 3, \dots}$ (resp. $\{\eta_{(w), i}\}_{i = -1, -2, -3, \dots}$) converges to $\eta_w$ (resp. $\eta_\infty$), where $\eta_{(w), i}$ corresponds to the disc containing $w$ with logarithmic radius $i$ for all $i \in \zz$ -- see \cite[Definition 2.1, Remark 2.2]{yelton2024branch} for another (more detailed) formulation.

\end{dfn}

As is discussed in \cite[Remark 2.2]{yelton2024branch}, the space $\Berk$ is path-connected, and there is a unique non-backtracking path between any pair of points in $\Berk$.  This allow us to set the following notation.  Below we denote the image in $\Berk$ of the non-backtracking path between two points $\eta, \eta' \in \Berk$ by $[\eta, \eta'] \subset \Berk$, and we will often refer to this image itself as ``the path'' from $\eta$ to $\eta'$; note that with this notation we have $[\eta, \eta'] = [\eta', \eta]$.  The above observations imply that, given a point $\eta \in \Berk$ and a connected subspace $U \in \Berk$, there is a (unique) point $\xi \in U$ such that every path from $\eta$ to a point in $U$ contains $\xi$; we will often speak of ``the closest point in $U$ to $\eta$'' in referring to this point $\xi$.  In a similar way, if $U, U' \in \Berk$ are connected subspaces, we will speak of ``the closest point in $U$ to $U'$'' (and vice versa).  Given a point $\eta \in \Hyp$ and subspaces $U, U' \subset \Berk$ (which have nontrivial intersection with $\Hyp$), we write $\delta(\eta, U)$ (resp. $\delta(U, U')$) for the distance between $\eta$ and the closest point in $U'$ to $\eta$ (resp. between the closest point in $U$ to $U'$ and the closest point in $U'$ to $U$).

\begin{rmk} \label{rmk ordering}

One easily sees that the order relation $>$ defined above can be equivalently defined as follows.  For points $\eta, \eta' \in \Hyp$, we have $\eta > \eta'$ if $\eta$ lies in the interior of the path $[\eta', \eta_\infty] \subset \Berk$.  From this one sees that given points $\eta', \eta'' > \eta$, the points $\eta', \eta'' \in \Hyp$ are comparable under the ordering (that is, we have the trichotomy $\eta'' > \eta'$, $\eta'' < \eta'$, or $\eta'' = \eta'$).

\end{rmk}

Given any subset $A \subset \proj_K^1$, we write $\Sigma_A \subset \Berk$ for the convex hull of the subspace $\{\eta_z\}_{z \in A} \subset \Berk$, \textit{i.e.} the smallest connected subspace of $\Berk$ containing $\{\eta_z\}_{z \in A}$.

We recall the set-up given in \S\ref{sec background Schottky} involving the groups $\Gamma \lhd \Gamma_0 < \PGL_2(K)$, the subset $\Omega \subset \proj_K^1$ of non-limit points under the action of $\Gamma$ (or $\Gamma_0$), the generators $s_0, \dots, s_h$ of $\Gamma_0$, and the fixed points $a_0, b_0, \dots, a_h, b_h \in \Omega$; let us write $S = \{a_0, b_0, \dots, a_h, b_h\}$.  Recall that the image of $S$ modulo the action of $\Gamma_0$ coincides with the set $\mathcal{B}$ of branch points of our superelliptic curve with the points $a_i, b_i \in S$ mapping respectively to $\alpha_i, \beta_i \in \mathcal{B}$ for $0 \leq i \leq h$.  For $0 \leq i \leq h$, define $\Lambda_{(i)} = \Sigma_{\{a_i, b_i\}} \subset \Berk$ (that is, the subspace $\Lambda_{(i)} \subset \Hyp$ is the \emph{axis} connecting $z_{a_i}$ to $z_{b_i}$), and define $\bbLambda_{(i)}$ to be the tubular neighborhood given by 
\begin{equation}
\bbLambda_{(i)} = \{\eta \in \Hyp \ | \ \delta(\eta, \Lambda_{(i)}) \leq \tfrac{v(p)}{p-1}\}.
\end{equation}
We set analogous notation for the set $\mathcal{B}$ of branch points: for $0 \leq i \leq h$, define $\bar{\Lambda}_{(i)} = \Sigma_{\{\alpha_i, \beta_i\}} \subset \Berk$, and define $\bar{\bbLambda}_{(i)}$ to be the tubular neighborhood given by 
\begin{equation}
\bar{\bbLambda}_{(i)} = \{\eta \in \Hyp \ | \ \delta(\eta, \bar{\Lambda}_{(i)}) \leq \tfrac{pv(p)}{p-1}\}
\end{equation}
(note the difference in radius of this tubular neighborhood from that of $\bbLambda_{(i)}$.  We see from these definitions that when $p$ is not the residue characteristic of $K$, we have $\bbLambda_{(i)} = \Lambda_{(i)}$ and $\bar{\bbLambda}_{(i)} = \bar{\Lambda}_{(i)}$.)

\begin{thm} \label{thm disjointness of neighborhoods}

We have $\bbLambda_{(i)} \cap \bbLambda_{(j)} = \varnothing$ and $\bar{\bbLambda}_{(i)} \cap \bar{\bbLambda}_{(j)} = \varnothing$ for indices $i \neq j$.

\end{thm}

\begin{proof}
The mutual disjointness of the subspaces $\bbLambda_{(i)} \subset \Berk$ is \cite[Proposition 2.13]{yelton2024branch}, while the mutual disjointness of the subspaces $\bar{\bbLambda}_{(i)} \subset \Berk$ is \cite[Corollary 5.1]{yelton2024branch2}.
\end{proof}

We set some notation that will be used from now on.  For any indices $i \neq j$, let $\tilde{v}_{i, j}$ (resp. $\tilde{v}_{j, i}$) be the closest point in $\bbLambda_{(i)}$ (resp. $\bbLambda_{(j)}$) to $\bbLambda_{(j)}$ (resp. $\bbLambda_{(i)}$), so that $[\tilde{v}_{i, j}, \tilde{v}_{j, i}]$ is the shortest (non-backtracking) path between $\bbLambda_{(i)}$ and $\bbLambda_{(j)}$, and define the points $\tilde{\bar{v}}_{i, j}, \tilde{\bar{v}}_{j, i}$ analogously with respect to $\bar{\bbLambda}_{(i)}, \bar{\bbLambda}_{(j)}$.  One sees from unique path-connectedness that the points $\tilde{v}_{i, j}, \tilde{v}_{j, i}$ lie on the shortest path connecting $\Lambda_{(i)}$ to $\Lambda_{(j)}$, so that we have $\tilde{v}_{i, j}, \tilde{v}_{j, i} \in \Sigma_S$; similarly, we have $\tilde{\bar{v}}_{i, j}, \tilde{\bar{v}}_{j, i} \in \Sigma_{\mathcal{B}}$.  For $1 \leq i \leq h$, let $v_i, \hat{v}_i \in \Berk$ denote the points corresponding respectively to the smallest disc $D_i$ which contains the points $a_i, b_i$ and the (unique) disc $\hat{D}_i \supseteq D_i$ satisfying $d(\hat{D}_i) = d(D_i) - \frac{v(p)}{p-1}$.  Analogously, let $\bar{v}_i, \hat{\bar{v}}_i \in \Berk$ denote the points corresponding respectively to the smallest disc $\bar{D}_i$ which contains the points $\alpha_i, \beta_i$ and the (unique) disc $\hat{\bar{D}}_i \supseteq \bar{D}_i$ satisfying $d(\hat{\bar{D}}_i) = d(\bar{D}_i) - \frac{pv(p)}{p-1}$.  (Note that we have $\hat{v}_i = v_i$ and $\hat{\bar{v}}_i = \bar{v}_i$ if $p$ is not the residue characteristic of $K$.)  Set $\hat{v}_0 = \hat{\bar{v}}_0 = \eta_\infty$.

Let $\mathcal{I}$ be the directed graph whose vertices are the indices $i \in \{0, \dots, h\}$ and such that any pair of indices $i \neq j$ are connected by an edge directed toward $i$ if the path $[\hat{\bar{v}}_i, \hat{\bar{v}}_j] \subset \Berk$ does not pass through the point $\hat{\bar{v}}_l$ for any $l \neq i, j$ and that we have $\hat{\bar{v}}_j > \hat{\bar{v}}_i$.  We remark that this graph $\mathcal{I}$ is clearly a tree rooted at the vertex $0$: we have $\hat{\bar{v}}_0 > \hat{\bar{v}}_i$ for $1 \leq i \leq h$, so $\mathcal{I}$ is connected with every vertex other than $0$ in a ``downward" path from $0$, and one sees using \Cref{rmk ordering} that there cannot be cycles in $\mathcal{I}$.  Given any index $i \in \{1, \dots, h\}$, let $i'$ be the (unique) index such that the corresponding vertices $i, i'$ of $\mathcal{I}$ are connected by an edge directed towards $i$.

\begin{prop} \label{prop v_i etc.}

Given an index $i \in \{1, \dots, h\}$, for any index $j \neq i$, we have the following.

\begin{enumerate}[(a)]

\item The point $v_i$ (resp. $\hat{v}_i$) is the unique maximal point in $\Lambda_{(i)} \subset \Sigma_S$ (resp. $\bbLambda_{(i)}$) with respect to the relation $>$, and the analogous statement holds for $\bar{v}_i \in \bar{\Lambda}_{(i)} \subset \Sigma_{\mathcal{B}}$ (resp. $\hat{\bar{v}}_i \in \bar{\bbLambda}_{(i)}$).

\item We have $\hat{v}_i = \tilde{v}_{i, i'}$ and $\hat{\bar{v}}_i = \tilde{\bar{v}}_{i, i'}$.  (In particular, we have $\hat{v}_i \in \bbLambda_{(i)} \cap \Sigma_S$ and $\hat{\bar{v}}_i \in \bar{\bbLambda}_{(i)} \cap \Sigma_{\mathcal{B}}$.)

\item We have $\delta(\tilde{v}_{j, i}, \Lambda_{(j)}) = \frac{v(p)}{p-1}$ and $\delta(\tilde{\bar{v}}_{j, i}, \bar{\Lambda}_{(j)}) = \frac{pv(p)}{p-1}$.

\end{enumerate}

\end{prop}

\begin{proof}
The fact that $v_i$ is maximal in the axis $\Lambda_{(i)}$ follows directly from the easily verified fact that the points in this axis are those corresponding to discs that either contain exactly one of the points $a_i, b_i$ or minimally contain both points $a_i, b_i$.  It follows from \Cref{rmk ordering} that given any $\xi \in \Lambda_{(i)}$ and any $v \notin \Lambda_{(i)}$ with $v \geq \xi$, we have $v \geq v_i$.

Now we clearly have $\hat{v}_i \geq v_i$ with $\delta(\hat{v}_i, v_i) = \frac{v(p)}{p-1}$.  In particular, this implies $\hat{v}_i \in \bbLambda_{(i)}$.  Now take any point $w \in \bbLambda_{(i)}$, and let $\xi$ be the closest point in $\Lambda_{(i)}$ to $w$.  We have $w \vee \xi \geq \xi$ and thus it follows that $w \vee \xi = \xi \in \Lambda_{(i)}$ or $w \vee \xi \geq v_i$.  In the former case, we get $\hat{v}_i \geq v_i \geq \xi = w \vee \xi \geq w$; in the latter case, we get $\delta(w \vee \xi, v_i) \leq \delta(w, v_i) \leq \frac{v(p)}{p-1} = \delta(\hat{v}_i, v_i)$ and therefore $v_i \geq w \vee \xi \geq w$.  Either way, we have shown that $\hat{v}_i \geq w$ for any $w \in \bbLambda_{(i)}$, proving the statement of part (a) regarding the points $v_i, \hat{v}_i$.  The statements for the points $\bar{v}_i, \hat{\bar{v}}_i$ follow from a completely analogous argument.

Since we have $\hat{v}_{i'} > \hat{v}_i$, every point in the path $[\hat{v}_{i'}, \hat{v}_i]$ is $\geq \hat{v}_i$.  This means by part (a) that $\hat{v}_i$ is the only point of $[\hat{v}_{i'}, \hat{v}_i]$ lying in $\bbLambda_{(i)}$ and therefore is the closest point in $\bbLambda_{(i)}$ to $\hat{v}_{i'}$ and thus to all of $\bbLambda_{(i')}$.  Part (b) follows from this and from a completely analogous argument involving $\hat{\bar{v}}_i$.

Now for any index $j \neq i$, we have $\bbLambda_{(i)} \cap \bbLambda_{(j)} = \varnothing$ by \Cref{thm disjointness of neighborhoods}, which implies the inequality $\delta(\tilde{v}_{i, j}, v_{j, i}) > \frac{v(p)}{p-1}$, where $v_{j, i}$ is the closest point in $\Lambda_{(j)}$ to $\tilde{v}_{i, j}$ (and thus to $\bbLambda_{(i)}$).  Write $\xi$ for the (unique) point in the path $[\tilde{v}_{i, j}, v_{j, i}]$ with $\delta(\xi, v_{j, i}) = \frac{v(p)}{p-1}$.  By definition we have $\xi \in \bbLambda_{(j)}$, and meanwhile, we have $\xi' \notin \bbLambda_{(j)}$ for any point $\xi'$ in the interior of the path $[\xi, \tilde{v}_{i, j}]$.  Therefore, we have $\xi = \tilde{v}_{j, i}$ by definition, proving the first equality asserted by part (c); the second equality follows from a completely analogous argument.
\end{proof}

After applying a suitable automorphism in $\PGL_2(K)$ to our set $S = \{a_i, b_i\}_{0 \leq i \leq h}$ and replacing $\Gamma \lhd \Gamma_0$ with their conjugates by that automorphism, we may assume $b_0 = \infty$: see \cite[Remark 1.2]{yelton2024branch2}.  A condition on the set $S$ called \emph{optimality} is defined in \cite[Definition 3.12]{yelton2024branch}, and it is shown as \cite[Proposition 3.8 combined with Theorem 3.13]{yelton2024branch} that one can always modify $S$ into an optimal set without changing $b_0$ or affecting the groups $\Gamma \lhd \Gamma_0$.  In light of all of this, we assume from now on that our set $S = \{a_i, b_i\}_{0 \leq i \leq h}$ is optimal and satisfies $b_0 = \infty$.

The author's previous work in \cite{yelton2024branch2} leads to the following crucial result relating the sets $S$ and $\mathcal{B}$.

\begin{thm} \label{thm distances between neighboring axes}

There is a homeomorphism $\pi_* : \Sigma_S \to \Sigma_{\mathcal{B}}$ extending the map given by $\eta_{a_i} \mapsto \eta_{\alpha_i}; \ \eta_{b_i} \mapsto \eta_{\beta_i}$, preserving the ordering relation $>$, mapping $\Lambda_{(i)}$ (resp. $\bbLambda_{(i)}$) onto $\bar{\Lambda}_{(i)}$ (resp. $\bar{\bbLambda}_{(i)}$) for each index $i$, and mapping the points $v_i, \hat{v}_i, \tilde{v}_{i, j}$ to $\bar{v}_i, \hat{\bar{v}}_i, \tilde{\bar{v}}_{i, j}$ respectively for all pairs of indices $i \neq j$.  We have $[\tilde{v}_{i, j}, \tilde{v}_{j, i}] \cap \bbLambda_{(l)} = \varnothing$ for all $l \neq i, j$ if and only if any $[\tilde{\bar{v}}_{i, j}, \tilde{\bar{v}}_{j, i}] \cap \bar{\bbLambda}_{(i)} = \varnothing$ for all $l \neq i$, and for such indices $i \neq j$, the homeomorphism $\pi_*$ restricts to an isometry $[\tilde{v}_{i, j}, \tilde{v}_{j, i}] \to [\tilde{\bar{v}}_{i, j}, \tilde{\bar{v}}_{j, i}]$.  In particular, we have 
\begin{equation} \label{eq distances between neighboring axes}
\delta(\bbLambda_{(i)}, \bbLambda_{(j)}) = \delta(\tilde{v}_{i, j}, \tilde{v}_{j, i}) = \delta(\tilde{\bar{v}}_{i, j}, \tilde{\bar{v}}_{j, i}) = \delta(\bar{\bbLambda}_{(i)}, \bar{\bbLambda}_{(j)}).
\end{equation}

\end{thm}

\begin{proof}
This result is obtained as a corollary of \cite[Theorem 1.3]{yelton2024branch2} as follows.  The homeomorphism $\pi_* : \Sigma_S \to \Sigma_{\mathcal{B}}$ furnished by that theorem clearly maps $\Lambda_{(i)}$ onto $\bar{\Lambda}_{(i)}$ for each $i$.  Fixing an index $i$ and a point $w \in \Sigma_S$, let $\xi$ be the closest point in $\Lambda_{(i)}$ to $w$.  Then $\pi_*(v)$ is the closest point in $\bar{\Lambda}_{(i)}$ to $\pi_*(w)$.  If $\delta(v, w) \leq \frac{v(p)}{p-1}$, then, in the notation of that theorem, following definitions we get $\mu(v, w) = \delta(v, w)$; applying the formula 
\begin{equation} \label{eq dilation}
\delta(\pi_*(v), \pi_*(w)) = \delta(v, w) + (p - 1)\mu(v, w)
\end{equation}
 given by the theorem yields $\delta(\pi_*(v), \pi_*(w)) = p \delta(v, w) \leq \frac{pv(p)}{p-1}$ and so we have $\pi_*(w) \in \bar{\bbLambda}_{(i)}$.  If on the other hand we have $\delta(v, w) > \frac{v(p)}{p-1}$, then we have $\mu(v, w) = \frac{v(p)}{p-1} < \delta(v, w)$ and the formula (\ref{eq dilation}) implies that $\delta(\pi_*(v), \pi_*(w)) > \frac{pv(p)}{p-1}$ and so we have $\pi_*(w) \notin \bar{\bbLambda}_{(i)}$.  It follows that the homeomorphism $\pi_*$ maps each $\bbLambda_{(i)}$ onto $\bar{\bbLambda}_{(i)}$.  (See also \cite[Remark 1.4]{yelton2024branch2}.)

Suppose that $\eta, \eta' \in \Berk$ are points with $\eta' > \eta$.  Then the fact that $\pi_*(\eta') > \pi_*(\eta)$ follows from applying \Cref{rmk ordering} and considering that we have $\pi_*(\eta_\infty) = \pi_*(\eta_{b_0}) = \eta_{\beta_0} = \eta_\infty$.

Now, given any indices $i \neq j$, it is clear from the fact that $\pi_*$ is a homeomorphism of uniquely path-connected topological spaces that it maps $\tilde{v}_{i, j}$ (resp. $\tilde{v}_{j, i}$) to the closest point in $\bar{\bbLambda}_{(i)}$ (resp. $\bar{\bbLambda}_{(j)}$) to $\bar{\bbLambda}_{(j)}$ (resp. $\bar{\bbLambda}_{(i)}$), which by definition is $\tilde{\bar{v}}_{i, j}$ (resp. $\tilde{\bar{v}}_{j, i}$).  Similarly, using the characterization of the points $v_i, \hat{v}_i$ given by \Cref{prop v_i etc.}(a), we see that $\pi_*$ maps them respectively to $\bar{v}_i, \hat{\bar{v}}_i$.

The statement about $[\tilde{v}_{i, j}, \tilde{v}_{j, i}] \cap \bbLambda_{(l)}$ is now immediate.  If $i \neq j$ are chosen so that this intersection is empty for all $l \neq i, j$, then, in the notation of \cite[Theorem 1.3]{yelton2024branch2}, it follows from definitions that we have $\mu(v, w) = 0$ for any $v, w \in [\tilde{v}_{i, j}, \tilde{v}_{j, i}]$.  Thus, applying the formula (\ref{eq dilation}) given by that theorem to such points $v, w$ shows that $\pi_*$ restricts to the claimed isometry of paths.
\end{proof}

\begin{cor} \label{cor isometric connected components}

Let $\pi_* : \Sigma_S \to \Sigma_{\mathcal{B}}$ be the map given by \Cref{thm distances between neighboring axes}.  Let $J \subset \{1, \dots, h\}$ be a subset of indices such that we have $\tilde{v}_{i', i} = \tilde{v}_{j', j}$ (which by \Cref{thm disjointness of neighborhoods} implies $i' = j'$) for any $i, j \in J$, and write 
\begin{equation*}
U = \bigcup_{i \in J} [\tilde{v}_{i, i'}, \tilde{v}_{i', i}]; \ \ \bar{U} = \bigcup_{i \in J} [\tilde{\bar{v}}_{i, i'}, \tilde{\bar{v}}_{i', i}].
\end{equation*}
The map $\pi_*$ restricts to an isometry $U \to \bar{U}$ which preserves the ordering relation $>$.

\end{cor}

\begin{proof}
The statement follows quickly from \Cref{thm distances between neighboring axes} by noting that the restriction of the homeomorphism $\pi_*$ to $U$ is given by gluing together the isometries $[\tilde{v}_{i, i'}, \tilde{v}_{i', i}] \to [\pi_*(\tilde{v}_{i, i'}), \pi_*(\tilde{v}_{i', i})]$ along their (non-empty) intersections and thus is itself an isometry.  The order-preserving property is inherited from $\pi_*$.  (Alternately, this corollary can be proved directly from \cite[Theorem 1.3]{yelton2024branch2}.)
\end{proof}

\begin{dfn} \label{dfn vertices}

Given a finite subset $A \subset \proj_K^1$, a \emph{vertex} of $\Sigma_A$ is a point $v \in \Sigma_A$ whose open neighborhoods contain a star shape centered at $v$ (with $\geq 3$ edges coming out of it).

With the above notation, a \emph{distinguished vertex} of the convex hull $\Sigma_S$ (resp. $\Sigma_{\mathcal{B}}$) is a vertex which lies in $\Lambda_{(i)}$ (resp. $\bar{\Lambda}_{(i)}$) for some index $i \in \{0, \dots, h\}$.

\end{dfn}

We now provide a quick dictionary between the language of clusters and that of convex hulls in the Berkovich projective line which will be suitable for our purposes.

\begin{prop} \label{prop dictionary}

There is a one-to-one correspondence between the non-singleton clusters $\mathfrak{s}$ of $\mathcal{B}$ and the set of vertices of $\Sigma_{\mathcal{B}}$ given by sending a cluster $\mathfrak{s}$ to the point $\eta_{D_{\mathfrak{s}}}$, where $D_{\mathfrak{s}} \subset \cc_K$ is the smallest disc containing $\mathfrak{s}$.  We moreover have the following.

\begin{enumerate}[(a)]

\item For clusters $\mathfrak{s}, \mathfrak{c}$, we have $\eta_{D_{\mathfrak{s} \vee \mathfrak{c}}} = \eta_{D_{\mathfrak{s}}} \vee \eta_{D_{\mathfrak{c}}}$ and $\delta(\eta_{D_{\mathfrak{s}}}, \eta_{D_{\mathfrak{c}}}) = d(\mathfrak{s}) + d(\mathfrak{c}) - 2d(\mathfrak{s} \vee \mathfrak{c})$.

\item A cluster which is not (resp. is) the disjoint union of $\geq 2$ even-cardinality clusters corresponds to a distinguished vertex (resp. a non-distinguished vertex).

\item An cluster which is not the disjoint union of $\geq 2$ even-cardinality clusters coincides with $\mathfrak{s}_i$ for some index $i$ (and thus, by \Cref{prop v_i etc.}(a), corresponds to the maximal point in $\bar{\Lambda}_{(i)}$) if and only if that cluster has even cardinality.

\item An \"{u}bereven cluster corresponds to a non-distinguished vertex which does not lie in $\bbLambda_{(i)}$ for any index $i$.

\end{enumerate}

\end{prop}

\begin{proof}
The claimed one-to-one correspondence and the claim of part (b) are given by \cite[Proposition 3.8(b)]{yelton2024branch2}.  The first identity given by part (a) comes from an immediate verification, and the claimed distance formula then follows directly from the definition of the distance function $\delta$.

Note that each cluster $\mathfrak{s}_i$ corresponds to the point $\bar{v}_i$ by definition of $\bar{v}_i$ and that $\bar{v}_i$ is the maximal point of $\bar{\Lambda}_{(i)}$ by \Cref{prop v_i etc.}(a).  Now let $\mathfrak{s}$ be a cluster which is not the disjoint union of $\geq 2$ even-cardinality clusters.  By part (b), the corresponding point $\eta_{D_{\mathfrak{s}}}$ is a distinguished vertex and thus lies in $\bar{\Lambda}_{(i)}$ for some index $i$; which means we have either $\#(\mathfrak{s} \cap \{\alpha_i, \beta_i\}) = 1$ or $\mathfrak{s} = \mathfrak{s}_i$.  If the cardinality of $\mathfrak{s}$ is even and if we have $\#(\mathfrak{s} \cap \{\alpha_i, \beta_i\}) = 1$, then we also have $\#(\mathfrak{s} \cap \{\alpha_j, \beta_j\}) = 1$ for some $j \neq i$, implying $\eta_{D_{\mathfrak{s}}} \in \bar{\Lambda}_{(i)} \cap \bar{\Lambda}_{(j)}$.  Similarly, if the cardinality of $\mathfrak{s}$ is odd and if we have $\mathfrak{s} = \mathfrak{s}_i$, then the fact that $\alpha_i, \beta_i \in \mathfrak{s}$ similarly implies $\#(\mathfrak{s} \cap \{\alpha_j, \beta_j\}) = 1$ for some $j \neq i$, again implying $\eta_{D_{\mathfrak{s}}} \in \bar{\Lambda}_{(i)} \cap \bar{\Lambda}_{(j)}$.  Both assumptions thus lead to a contradiction to \Cref{thm disjointness of neighborhoods}.  We conclude that the cluster $\mathfrak{s}$ has even cardinality if and only if we have $\mathfrak{s} = \mathfrak{s}_i$.  This proves part (c).

Now let $\mathfrak{s}$ be a cluster which is the disjoint union of $\geq 2$ even-cardinality clusters.  For each index $i$, let $\xi_i$ be the closest point in $\bar{\Lambda}_{(i)} \subset \Sigma_{\mathcal{B}}$ to point $\eta_{D_{\mathfrak{s}}} \in \Sigma_{\mathcal{B}}$ (which is a non-distinguished vertex by part (a)).  One sees that each point $\xi_i \in \Sigma_{\mathcal{B}}$ is a distinguished vertex by observing that the intersection of the (non-singleton) paths $[\eta_{a_i}, \xi_i], [\eta_{b_i}, \xi_i], [\eta_{D_{\mathfrak{s}}}, \xi_i] \subset \Sigma_{\mathcal{B}}$ coincides with the singleton $\{\xi_i\}$.  Thus by part (b), each vertex $\xi_i$ corresponds to a cluster $\mathfrak{c}_i$ which is not the disjoint union of $\geq 2$ even-cardinality clusters.  Suppose that we have $\eta_{D_{\mathfrak{s}}} \in \bar{\bbLambda}_{(i)}$ for some $i$.  This implies $\delta(\eta_{D_{\mathfrak{s}}}, \xi_i) \leq \frac{pv(p)}{p-1}$ for some $i$.  By part (a), this is the same inequality as $d(\mathfrak{s}) + d(\mathfrak{c}_i) - 2d(\mathfrak{s} \vee \mathfrak{c}_i) \leq \frac{pv(p)}{p-1}$ and thus violates \Cref{dfn ubereven}; therefore, the cluster $\mathfrak{s}$ is not \"{u}bereven.  Now suppose conversely that we have $\eta_{D_{\mathfrak{s}}} \notin \bar{\bbLambda}_{(i)}$ for all $i$.  Then each distinguished vertex $\xi$, being in some $\bar{\Lambda}_{(i)}$, satisfies $\delta(\eta_{D_{\mathfrak{s}}}, \xi) > \frac{pv(p)}{p-1}$.  Then by a similar use of parts (a) and (b), we get the inequality (\ref{eq ubereven}) for each such cluster $\mathfrak{c}$, and so \Cref{dfn ubereven} says that the cluster $\mathfrak{s}$ is \"{u}bereven.  This proves part (d).
\end{proof}

\begin{prop} \label{prop uniqueness used in main theorem}

Given an index $i \in \{1, \dots, h\}$, choose any point $\bar{w} \in [\bar{v}_i, \tilde{\bar{v}}_{i', i}]$; let $\bar{w}''$ be the vertex in the interior of $[\tilde{\bar{v}}_{i', i}, \tilde{v}_{i, i'} \vee \bar{w}]$ which is closest to $\bar{w}$ if one exists; and let $\bar{w}'' = \tilde{\bar{v}}_{i', i}$ otherwise.  For all indices $j \neq i, i'$ satisfying $\bar{v}_j \vee \bar{w} > \tilde{v}_{i, i'} \vee \bar{w}$, we have $\delta(\tilde{\bar{v}}_{j, i}, \tilde{v}_{i, i'} \vee \bar{w}) > \delta(\bar{w}'', \tilde{v}_{i, i'} \vee \bar{w})$.

\end{prop}

\begin{proof}
Let $j \neq i, i'$ be an index such that $\bar{v}_j \vee \bar{w} > \tilde{\bar{v}}_{i, i'} \vee \bar{w}$.  Then we must have $\tilde{\bar{v}}_{j, i} \vee \bar{w} > \tilde{\bar{v}}_{i, i'} \vee \bar{w}$, because otherwise, we would have $\tilde{v}_{i, i'} \vee \bar{w} \in [\bar{v}_j, \tilde{\bar{v}}_{j, i}] \subset \bar{\bbLambda}_{(j)}$, which contadicts the definition of $i'$ and the hypothesis on $\bar{w}$.  Suppose that we have $\tilde{\bar{v}}_{j, i} \vee \bar{w} \notin [\tilde{\bar{v}}_{i', i}, \tilde{\bar{v}}_{i, i'} \vee \bar{w}]$.  This implies the inequalities 
\begin{equation}
\delta(\tilde{\bar{v}}_{j, i}, \tilde{\bar{v}}_{i, i'} \vee \bar{w}) \geq \delta(\tilde{\bar{v}}_{j, i} \vee \bar{w}, \tilde{\bar{v}}_{i, i'} \vee \bar{w}) > \delta(\tilde{\bar{v}}_{i', i}, \tilde{\bar{v}}_{i, i'} \vee \bar{w}) \geq \delta(\bar{w}'', \tilde{\bar{v}}_{i, i'} \vee \bar{w}).
\end{equation}
Suppose instead that we have $\tilde{\bar{v}}_{j, i} \vee \bar{w} \in [\tilde{\bar{v}}_{i', i}, \tilde{\bar{v}}_{i, i'} \vee \bar{w}]$.  We then have $\tilde{\bar{v}}_{j, i} \vee \bar{w} \neq \tilde{\bar{v}}_{j, i} \in \bbLambda_{(j)}$ from the construction of $i'$.  Thus, the paths $[\tilde{\bar{v}}_{j, i}, \tilde{\bar{v}}_{j, i} \vee \bar{w}], [\bar{w}, \tilde{\bar{v}}_{j, i} \vee \bar{w}], [\eta_\infty, \tilde{\bar{v}}_{j, i} \vee \bar{w}] \subset \Sigma_S$ each have interior and all intersect at the singleton $\{\tilde{\bar{v}}_{j, i} \vee \bar{w}\}$.  This implies by definition that the point $\tilde{\bar{v}}_{j, i} \vee \bar{w}$ is a vertex.  If we have $\tilde{\bar{v}}_{j, i} \vee \bar{w} \neq \tilde{\bar{v}}_{i', i}$, then it follows from the construction of $i'$ that the point $\tilde{\bar{v}}_{j, i} \vee \bar{w}$ is not a distinguished vertex, which implies that non-distinguished vertices in $[\bar{v}_i, \tilde{v}_{i', i}]$ exist and $\tilde{\bar{v}}_{j, i} \vee \bar{w} \geq \bar{w}''$.  Now, noting that we also have $\tilde{\bar{v}}_{i', i} \geq \bar{w}''$, we are guaranteed the inequalities 
\begin{equation}
\delta(\bar{v}_j, \tilde{\bar{v}}_{i, i'} \vee \bar{w}) \geq \delta(\tilde{\bar{v}}_{j, i}, \tilde{\bar{v}}_{i, i'} \vee \bar{w}) > \delta(\tilde{\bar{v}}_{j, i} \vee \bar{w}, \tilde{\bar{v}}_{i, i'} \vee \bar{w}) \geq \delta(\bar{w}'', \tilde{\bar{v}}_{i, i'} \vee \bar{w}).
\end{equation}
\end{proof}

\begin{cor} \label{cor uniqueness used in main theorem}

Given a cluster $\mathfrak{s} \in \mathfrak{C}_0$, a cluster $\mathfrak{s}'$ defined as in the statement of \Cref{thm main} exists, is unique, and can be described as follows.  Letting $\bar{w} \in \Sigma_S$ be the point corresponding via \Cref{prop dictionary} to the cluster $\mathfrak{s}$, we have $\bar{w} = \bar{v}_i$ or $\bar{w} \in [\tilde{v}_{i', i}, \hat{\bar{v}}_i]$ for some $i \in \{1, \dots, h\}$.  Let $\bar{w}'$ be the vertex in the interior of $[\tilde{\bar{v}}_{i', i}, \bar{w}]$ which is closest to $\bar{w}$ if one exists, and let $\bar{w}'$ be the closest point in $\bar{\Lambda}_{(i')}$ to $\bar{w}$ otherwise.  The cluster $\mathfrak{s}'$ is then the one corresponding to $\bar{w}'$ via \Cref{prop dictionary}.

\end{cor}

\begin{proof}
If $\mathfrak{s} = \mathfrak{s}_i$ for some $i \in \{1, \dots, h\}$, then this cluster is not \"{u}bereven by \Cref{prop dictionary}(d) and we have $\bar{w} = \bar{v}_i$ and 
\begin{equation} \label{eq uniqueness small distance1}
\delta(\tilde{\bar{v}}_{i, i'} \vee \bar{w}, \bar{w}) = \delta(\tilde{\bar{v}}_{i, i'}, \bar{v}_i) = \tfrac{pv(p)}{p-1} = (1 - u_{\{\mathfrak{s}\}})\tfrac{pv(p)}{p-1}.
\end{equation}

If on the other hand $\mathfrak{s}$ does not equal any of the clusters $\mathfrak{s}_i$, then by definition of $\mathfrak{C}_0$, this cluster is \"{u}bereven; by \Cref{prop dictionary}(d), the corresponding point $\bar{w}$ is a vertex not lying in $\bar{\bbLambda}_{(j)}$ for any index $j$.  Given any element $\alpha_j$ or $\beta_j$ of the cluster $\mathfrak{s}$, the path $[\bar{w}, \eta_{\alpha_j}]$ or $[\bar{w}, \eta_{\beta_j}]$ (whose maximal point with respect to $>$ is $\bar{w}$) is contained in $\Sigma_{\mathcal{B}}$; such a path eventually intersects a tubular neighborhood $\bar{\bbLambda}_{(l)}$ for at least one index $l \in \{1, \dots, h\}$ (as $j$ is such an index).  Among such indices $l$, let $i$ be the one which minimizes $\delta(\bar{w}, \bar{\bbLambda}_{(l)})$.  Then it is clear not only that $\hat{\bar{v}}_i = \tilde{\bar{v}}_{i, i'}$ (see \Cref{prop v_i etc.}(b)) is the closest point in $\bar{\bbLambda}_{(i)}$ to $\bar{w}$ but that we have $\tilde{\bar{v}}_{i', i} > \bar{w} > \tilde{\bar{v}}_{i, i'}$.  In particular, we have 
\begin{equation} \label{eq uniqueness small distance2}
\delta(\tilde{\bar{v}}_{i, i'} \vee \bar{w}, \bar{w}) = \delta(\bar{w}, \bar{w}) = 0 = (1 - u_{\{\mathfrak{s}\}})\tfrac{pv(p)}{p-1}.
\end{equation}

Now, given any cluster $\mathfrak{c} \in \mathfrak{C}$ with corresponding point $\eta_{D_{\mathfrak{c}}} \in \Sigma_{\mathcal{B}}$, we see using \Cref{prop dictionary}(a) that the hypothesis $d(\mathfrak{s}) - d(\mathfrak{s} \vee \mathfrak{c}) > (1 - u_{\{\mathfrak{s}\}})\frac{pv(p)}{p-1}$ translates to the conditions $\eta_{D_{\mathfrak{c}}} \vee \bar{w} > \bar{w}$ and $\delta(\eta_{D_{\mathfrak{c}}}, \bar{w}) > (1 - u_{\{\mathfrak{s}\}}) \frac{pv(p)}{p-1}$.  Using (\ref{eq uniqueness small distance1}),(\ref{eq uniqueness small distance2}), these conditions imply $\delta(\eta_{D_{\mathfrak{c}}} \vee \bar{w}, \tilde{\bar{v}}_{i, i'} \vee \bar{w}) > 0$, in turn implying $\eta_{D_{\mathfrak{c}}} \vee \bar{w} > \tilde{\bar{v}}_{i, i'} \vee \bar{w}$.

If the cluster $\mathfrak{c}$ is \"{u}bereven, then, as $\eta_{D_{\mathfrak{c}}} \vee \bar{w}$ is the vertex corresponding to the cluster $\mathfrak{c} \vee \mathfrak{s}$ by \Cref{prop dictionary}(a), using (\ref{eq uniqueness small distance1}),(\ref{eq uniqueness small distance2}) and the definition of $\bar{w}''$ from \Cref{prop uniqueness used in main theorem}, we get 
\begin{equation}
\delta(\eta_{D_{\mathfrak{c}}}, \bar{w}) - (2 - u_{\{\mathfrak{s}, \mathfrak{c}\}}) \tfrac{pv(p)}{p-1} = \delta(\eta_{D_{\mathfrak{c}}}, \bar{w}) - (1 - u_{\{\mathfrak{s}\}}) \tfrac{pv(p)}{p-1} \geq \delta(\eta_{D_{\mathfrak{c}}} \vee \bar{w}, \tilde{v}_{i, i'} \vee \bar{w}) \geq \delta(\bar{w}'', \tilde{v}_{i, i'} \vee \bar{w}),
\end{equation}
with equalities only if $\mathfrak{c} = \mathfrak{s}'$ (in which case $\bar{w}' = \bar{w}''$ is a vertex).

If, on the other hand, the cluster $\mathfrak{c}$ is not \"{u}bereven, by definition of $\mathfrak{C}$ and by \Cref{prop dictionary}(b), its corresponding vertex $\eta_{D_{\mathfrak{c}}}$ lies in $\bar{\Lambda}_{(j)}$ for some index $j \neq i$.  Then, using (\ref{eq uniqueness small distance1}),(\ref{eq uniqueness small distance2}), \Cref{prop v_i etc.}(c), and \Cref{prop uniqueness used in main theorem}, we get 
\begin{equation}
\delta(\eta_{D_{\mathfrak{c}}}, \bar{w}) - (2 - u_{\{\mathfrak{s}, \mathfrak{c}\}}) \tfrac{pv(p)}{p-1} \geq \delta(\tilde{v}_{j, i}, \bar{w}) - (1 - u_{\{\mathfrak{s}\}}) \tfrac{pv(p)}{p-1} = \delta(\tilde{v}_{j, i}, \tilde{v}_{i, i'} \vee \bar{w}) \geq \delta(\bar{w}'', \tilde{v}_{i, i'} \vee \bar{w}),
\end{equation}
with equalities only if $\mathfrak{c} = \mathfrak{s}'$.

By definition, we have $\bar{w}' = \bar{w}''$ if one of them is a non-distinguished vertex (which by \Cref{prop dictionary}(d) happens exactly when $\mathfrak{s}'$ is \"{u}bereven), and when this is not the case, we have $\bar{w}'' \in [\bar{w}', \tilde{v}_{i, i'}]$ with $\delta(\bar{w}', \bar{w}'') = \frac{pv(p)}{p-1} = (1 - u_{\{\mathfrak{s}'\}}) \frac{pv(p)}{p-1}$ by \Cref{prop v_i etc.}(c).  This, combined with (\ref{eq uniqueness small distance1}),(\ref{eq uniqueness small distance2}), gives us 
\begin{equation}
\delta(\bar{w}'', \tilde{v}_{i, i'} \vee \bar{w}) = \delta(\bar{w}', \bar{w}) - (2 - u_{\{\mathfrak{s}, \mathfrak{s}'\}}) \tfrac{pv(p)}{p-1}.
\end{equation}
Putting all of this together and using \Cref{prop dictionary}(a) to translate distances between points of $\Sigma_{\mathcal{B}}$ into differences of depths of clusters, we get 
\begin{equation}
d(\mathfrak{s}) + d(\mathfrak{c}) - 2d(\mathfrak{s} \vee \mathfrak{c}) - (2 - u_{\{\mathfrak{s}, \mathfrak{c}\}}) \tfrac{pv(p)}{p-1} \geq d(\mathfrak{s}) + d(\mathfrak{s}') - 2d(\mathfrak{s} \vee \mathfrak{s}') - (2 - u_{\{\mathfrak{s}, \mathfrak{s}'\}}) \tfrac{pv(p)}{p-1},
\end{equation}
with equality only if $\mathfrak{c} = \mathfrak{s}'$.  The cluster $\mathfrak{s}'$ as we have defined it for this proposition therefore fits exactly the definition of $\mathfrak{s}'$ given in the statement of \Cref{thm main}.
\end{proof}

\section{Valuations of a period matrix of the Jacobian} \label{sec period matrix}

The goal of this section is to compute the values of the pairing 
\begin{equation}
\langle \cdot, \cdot \rangle_\mon^\varphi : \bar{\Gamma} \times \bar{\Gamma} \to \zz, \ \ (\gamma, \gamma') \mapsto v(c_\gamma(\gamma')).
\end{equation}
An explicit description of the values of this pairing for any given $\gamma, \gamma' \in \Gamma$ is provided in \cite[\S4]{drinfeld1973periods} in terms of the action of $\Gamma$ on the \emph{Bruhat-Tits tree} associated to $\PGL_2(K)$, defined for instance by Mumford in \cite[\S1]{mumford1972analytic} and by Serre in \cite[\S II.1]{serre2003trees}.  We denote this tree by $\mathcal{T}$ and define it in the present article as follows.

Denoting the ring of integers of $K$ by $\mathcal{O}_K$, the vertices of $\mathcal{T}$ are the equivalence classes of rank-$2$ $\mathcal{O}_K$-submodules of $K^2$, where two such $\mathcal{O}_K$-modules $M, M'$ are equivalent if there is a scalar $a \in K^\times$ such that we have $M' = aM$.  Two vertices of $\mathcal{T}$ are connected by an edge if they are represented by $\mathcal{O}_K$-modules $M \supset M'$ with $M / M'$ isomorphic to the residue field of $K$.  It is known (for instance as \cite[Theorem II.1.1.1]{serre2003trees}) that the graph $\mathcal{T}$ defined this way is a tree.  The group of automorphisms $\GL_2(K)$ acts in an obvious way on the space of rank-$2$ $\mathcal{O}_K$-submodules of $K^2$, and this clearly induces an action of $\PGL_2(K)$ on the vertices of $\mathcal{T}$ which is well known and easily verified to be edge-preserving.

A \emph{loxodromic} (or \emph{hyperbolic}) element of $\PGL_2(K)$ is one which is represented by a matrix in $\GL_2(K)$ which is conjugate to a diagonal matrix whose diagonal entries have distinct valuations.  Given a loxodromic automorphism $\gamma \in \PGL_2(K)$, let $u_\gamma$ denote the quotient of these diagonal elements when they are ordered so that $v(u_\gamma) > 0$.  Then it is well known (see for instance \cite[\S II.1.3]{serre2003trees}) that there is an infinite non-backtracking path $L_\gamma \subset \mathcal{T}$ which is stabilized by $\gamma$ and on which $\gamma$ acts by \emph{translation} with \emph{amplitude} $v(u_\gamma) \in \zz_{> 0}$ on this axis -- that is, the automorphism $\gamma$ acts on the line $L_\gamma$ by moving each vertex of $L_\gamma$ to the vertex $v(u_\gamma)$ edges away in a fixed direction.  We call the subgraph $L_\gamma \subset \mathcal{T}$ the \emph{axis} of the loxodromic element $\gamma$.  As an easy exeercise one verifies that the amplitude $v(u_\gamma)$ can also be characterized as the minimum of the distance between $v$ and $\gamma(v)$ over all vertices $v$ of $\mathcal{T}$.  Below we will also freely use the easily verified fact that, fixing an element $\gamma' \in \Gamma$, the action of $\Gamma$ on the axis $L_{\gamma'}$ is given by $\gamma : L_{\gamma'} \mapsto L_{\gamma \gamma' \gamma^{-1}}$.

It is well known that every nontrivial element of a Schottky group is loxodromic (see for instance \cite[\S I.1.5]{gerritzen2006schottky}).  With respect to our Schottky group $\Gamma$ used to uniformize our superelliptic curve, let us write $\mathcal{T}_\Gamma := \bigcup_{\gamma \in \Gamma \smallsetminus \{1\}} L_\gamma \subset \mathcal{T}$ for the union of the axes of the nontrivial elements of $\Gamma$.  As the action of $\Gamma$ on $\mathcal{T}$ permutes the axes $L_\gamma$, the subtree $\mathcal{T}_\Gamma \subset \mathcal{T}$ is stabilized by this action.  The resulting quotient graph $\mathcal{T}_\Gamma / \Gamma$ is a finite graph by \cite[Theorem 1.23]{mumford1972analytic}.  Let us write $\pi_* : \mathcal{T}_\Gamma \to \mathcal{T}_\Gamma / \Gamma$ for the associated quotient map of graphs.

Given any two vertices $v, w \in \mathcal{T}$, write $[v, w] \subset \mathcal{T}$ for the subtree consisting of the (shortest) path from $v$ to $w$; the choice of initial vertex $v$ and terminal vertex $w$ gives this path an \emph{orientation}.  Given any graph $G$ with vertices $v, w, v', w'$ and any two (possibly backtracking) paths $[v, w], [v', w'] \subset \mathcal{T}_\Gamma / \Gamma$, define $e([v, w], [v', w'])$ to be the total number of edges in $[v, w] \cap [v', w']$, counted with orientation.  More precisely, if the pairs $v, w$ and $v', w'$ are each connected by a single edge in the respective paths $[v, w]$ and $[v', w']$, then we set $e([v, w], [v', w'])$ to be $0$ if these edges are distinct and $1$ (resp. $-1$) if the edges are the same and $v' = v$ (resp. $v' = w$); then for general paths $[v, w], [v', w']$, we define the pairing $e$ by treating paths as sums of edges (counted with orientation so that opposite orientation corresponds to the additive inverse) and extending $\zz$-bilinearly.

The following result, which is a rephrasing of \cite[Theorem 5]{drinfeld1973periods}, will be our main tool for computing the valuations $v(c_\gamma(\gamma'))$ for elements $\gamma, \gamma' \in \Gamma$.

\begin{thm} \label{thm drinfeld-manin}

Choose any vertices $v, w$ of $\mathcal{T}_\Gamma$.  For any elements $\gamma, \gamma' \in \Gamma$, we have 
\begin{equation}
v(c_\gamma(\gamma')) = e(\pi_*([v, \gamma(v)]), \pi_*([w, \gamma'(w)])).
\end{equation}
In particular, we have $v(c_\gamma(\gamma)) > 0$ for all $\gamma \in \Gamma \smallsetminus \{1\}$.

\end{thm}

\begin{rmk} \label{rmk drinfeld-manin}

One may observe that the theta functions $\Theta_{a, b}$ are defined differently in \cite{drinfeld1973periods}, so that the factor corresponding to each element $\tilde{\gamma} \in \Gamma$ in the product formula is $\frac{\tilde{\gamma}(z) - a}{\tilde{\gamma}(z) - b}$ rather than $\frac{z - \tilde{\gamma}(a)}{z - \tilde{\gamma}(b)}$.  However, it is easy to show from the proof of \Cref{lemma automorphy van steen} below (putting $s = \tilde{\gamma}^{-1}$) that each such factor under Drinfeld-Manin's definition equals $\frac{A - Ca}{A - Cb}$ times the factor corresponding to $\tilde{\gamma}^{-1}$ under our definition, where $\tilde{\gamma}$ is represented by $\begin{bsmallmatrix} A & B \\ C & D \end{bsmallmatrix}$.  These terms $\frac{A - Ca}{A - Cb}$ cancel out in the definiton of $c_{a, b}$; in particular $c_\gamma = c_{a, \gamma(a)}$ is not affected, and so we may still use \Cref{thm drinfeld-manin}.

\end{rmk}

To implement the formula given in the above theorem, we want to move to the Berkovich setting that was developed in \S\ref{sec berk}.  In order to make this transition, we define $\mathcal{A}$ to be the graph whose vertices are the points of $\Hyp$ whose corresponding discs have logarithmic radius in $\zz$ and contain an element of $K$, and where two vertices $\eta_D, \eta_{D'}$ are connected by an edge if we have $\delta(\eta_D, \eta_{D'}) = 1$.

\begin{rmk}

Given any two vertices $v, w$ of $\mathcal{A}$, one easily confirms that the point $v \vee w \in \Hyp$ is also a vertex of $\mathcal{A}$; using this fact, we see that the distance between the vertices $v$ and $w$ (defined as the number of edges in the shortest path between them) equals the distance $\delta(v, w)$ used to define $\Hyp$ as a metric graph.

\end{rmk}

Given any automorphism $\gamma \in \PGL_2(K)$ of $\proj_K^1$, there is a natural way to define its action on $\Hyp$: see \cite[\S2.3]{baker2010potential} or \cite[\S7.1]{benedetto2019dynamics}, for instance.  The action of any automorphism $\gamma \in \PGL_2(K)$ on a point $\eta_D \in \Hyp$ may be described concretely as $\gamma(\eta_D) = \eta_{E}$, where $E = \gamma(D)$ if $\infty \notin \gamma(D)$ and $E$ is the smallest closed disc containing $K \smallsetminus \gamma(D)$ if $\infty \in \gamma(D)$ (as can be deduced from \cite[Proposition 7.6, Theorem 7.12]{benedetto2019dynamics}; see also \cite[Proposition 2.4]{yelton2024branch}).  The action of each element of $\PGL_2(K)$ on $\Hyp$ is a metric-preserving self-homeomorphism by \cite[Proposition 2.30]{baker2010potential}.  It follows that this action induces an edge-preserving action of $\PGL_2(K)$ on the vertices of the graph $\mathcal{A}$.

\begin{rmk} \label{rmk equivalence between Berk and T}

The trees $\mathcal{T}$ and $\mathcal{A}$ are isomorphic as graphs in a way that carries the action of $\PGL_2(K)$ on $\mathcal{T}$ to its action on $\mathcal{A}$.  This is well known and is more or less shown in \cite[\S I.2.6]{gerritzen2006schottky}.  Here is a summary of a more direct argument.  We assign to each vertex $v$ of $\mathcal{T}$ a vertex of $\mathcal{A}$ as follows.  The vertex $v$ is represented by a rank-$2$ $\mathcal{O}_K$-submodule $M \subset K^2$.  Let $r \in \zz$ be the minimal valuation among the second coordinates of elements of $M$, and let $M_0 = \{(\alpha, \beta) \in M \ | \ v(\beta) = r\}$.  Then the corresponding vertex of $\mathcal{A}$ is $\eta_D$, where $D = \{\frac{\alpha}{\beta} \ | \ (\alpha, \beta) \in M_0\}$.  The inverse of this function assigns a vertex $\eta = \eta_D$ of $\mathcal{A}$ to the equivalence class of the rank-$2$ $\mathcal{O}_K$-submodule $M \subset K^2$ generated by $\{(a, 1), (b, 1)\} \subset K^2$, where $a, b \in D$ are elements such that $D$ is the smallest disc containing $a, b$ (so that its logarithmic radius is $v(b - a)$).

One shows that the actions of $\PGL_2(K)$ on $\mathcal{T}$ and $\mathcal{A}$ are compatible by separately considering the actions of the automorphisms in $\PGL_2(K)$ represented by the matrices $\begin{bsmallmatrix} 0 & 1 \\ 1 & 0 \end{bsmallmatrix}$ and $\begin{bsmallmatrix} \delta & \epsilon \\ 0 & 1 \end{bsmallmatrix}$ with $\delta \in K^\times$, $\epsilon \in K$ (these act on $\proj_K^1$ by the reciprocal automorphism and all affine automorphisms respectively), noting that these elements generate the whole group $\PGL_2(K)$.  Showing compatibility for any affine automorphism is immediate.  To show this for the reciprocal automorphism, choose a point $\eta = \eta_D \in \Hyp$ and points $a, b \in D$ such that $D$ is the smallest disc containing $a, b$ and we have $v(a) = v(b)$.  Then one checks that the reciprocal automorphism sends $\eta_D$ to $\eta_E$, where $E$ is the smallest disc containing $a^{-1}, b^{-1}$.  It is easy to verify that $\eta_D \mapsto \eta_E$ commutes with the action of the automorphism $\begin{bsmallmatrix} 0 & 1 \\ 1 & 0 \end{bsmallmatrix}$ on representative $\mathcal{O}_K$-modules of the corresponding vertices of $\mathcal{T}$.

\end{rmk}

In light of \Cref{rmk equivalence between Berk and T}, from now on we identify the Bruhat-Tits tree $\mathcal{T}$ with the tree $\mathcal{A}$ defined above; in this way, one views the tree $\mathcal{T}$ as the underlying graph of the metric graph given by enhancing the metric space $\Sigma_{\proj_K^1} \cap \Hyp$ with vertices given by the points corresponding to discs with integer logarithmic radius and center in $K$.

Recall that the group $\Gamma_0$ containing the Schottky group with index $p$ is generated by the elements $s_i \in \PGL_2(K)$ which each have order $p$ and fix the points $a_i, b_i \in S$ for $0 \leq i \leq h$, and recall the set of (free) generators $\gamma_{i, 0, n}$ of the Schottky group $\Gamma \lhd \Gamma_0$ defined in (\ref{eq Gamma}).  It is an elementary exercise to show from this definition of the Schottky group that it is in fact the subgroup of $\Gamma_0$ consisting of those words whose total exponent is divisible by $p$.  In particular, we have  
\begin{equation}
\gamma_{i, j, n} := s_j^{n-1} s_i s_j^{-n} \in \Gamma, \ \ i, j \in \{0, \dots, h\}, n \in \zz.
\end{equation}

\begin{prop} \label{prop basis of Gamma}

The $(p - 1)h$-element set $\{\gamma_{i, i', n}\}_{1 \leq i \leq h, \, 1 \leq n \leq p - 1}$ generates the Schottky group $\Gamma$.

\end{prop}

\begin{proof}
Consider the tree $\mathcal{I}$ defined in \S\ref{sec berk}.  Given any index $i \neq 0$, let $i^{(0)} := i, i^{(1)} := i', i^{(2)} := (i')', \dots, i^{(r)} := 0$ (for some $r \geq 1$) denote the sequence of vertices on the path from $i$ to $0$ in $\mathcal{I}$.  For any $i \neq 0$, we may write 
\begin{equation} \label{eq gammai0n}
\gamma_{i, 0, n} = [s_{i^{(r-1)}}^{n - 1} s_{i^{(r)}}^{1 - n}]^{-1} \cdots [s_{i^{(1)}}^{n - 1} s_{i^{(2)}}^{1 - n}]^{-1} \gamma_{i^{(0)}, i^{(1)}, n} [s_{i^{(1)}}^n s_{i^{(2)}}^{-n}] \cdots [s_{i^{(r-1)}}^n, s_{i^{(r)}}^{-n}].
\end{equation}
On checking that for any $n \in \{1, \dots, p - 1\}$ and any $j \neq 0$, we have $\gamma_{j, j', 1} \cdots \gamma_{j, j', n - 1} = s_j^{n - 1} s_{j'}^{1 - n}$ and $\gamma_{j, j', 1} \cdots \gamma_{j, j', n} = s_j^n s_{j'}^{-n}$, we see from (\ref{eq gammai0n}) that the element $\gamma_{i, 0, n}$ for each index $i \neq 0$ is a product of elements of the form $\gamma_{j, j', m}$ with $j \neq 0$ and $1 \leq m \leq n$.  As the elements $\gamma_{i, 0, n}$ generate the group $\Gamma$, the proposition is proved.
\end{proof}

\begin{cor} \label{cor basis of barGamma}

The $(p - 1)h$-element set $\{\gamma_{i, i', n}\}_{1 \leq i \leq h, \, 1 \leq n \leq p - 1}$ of images of elements of $\Gamma$ in the abelianization $\bar{\Gamma}$ is a basis of the free $\zz$-module $\Gamma$.

\end{cor}

\begin{proof}
We know that the Schottky group $\Gamma$ is free and generated by $g = (p - 1)h$ elements.  Therefore, its abelianization $\bar{\Gamma}$ is a free $\zz$-module generated by $(p - 1)h$ elements.  By \Cref{prop basis of Gamma}, the group $\Gamma$ is generated by its elements $\gamma_{i, i', n}$ for $1 \leq i \leq h$ and $1 \leq n \leq p - 1$, so the abelianization $\bar{\Gamma}$ is generated by the images of these $(p - 1)h$ elements.  The result follows.
\end{proof}

Now thanks to the above proposition, the values of the pairing $\langle \cdot, \cdot \rangle_\mon^\varphi$ over the set of generators $\{\gamma_{i, i', n}\}_{1 \leq i \leq h, \, n \in \zz}$ determine the values of the pairing over all of $\bar{\Gamma}$ (where again we are using the same notation $\gamma_{i, i', n}$ for the element of $\Gamma$ as well as its image in the abelianization $\bar{\Gamma}$).  Our main result in this section gives the values of this pairing over this subset.

\begin{thm} \label{thm values of pairing}

With all of the above notation (in particular $v_i, \tilde{v}_i$ for indices $i \neq 0$), for indices $i, j \neq 0$ and integers $m, n \in \zz$, we have 
\begin{equation} \label{eq values of pairing}
v(c_{\gamma_{i, i', m}}(\gamma_{j, j', n})) = 
\begin{cases}
2 \delta(\tilde{v}_i \vee \tilde{v}_j, \bbLambda_{(i')}) &\text{if } i' = j', \ \tilde{v}_i \vee \tilde{v}_j \notin \bbLambda_{(i')}, \text{ and } n \equiv m \ (\text{mod } p) \\
-\epsilon \delta(\tilde{v}_i \vee \tilde{v}_j, \bbLambda_{(i')}) &\text{if } i' = j', \ \tilde{v}_i \vee \tilde{v}_j \notin \bbLambda_{(i')}, \text{ and } n \equiv m \pm 1 \ (\text{mod } p) \\
0 &\text{otherwise,}
\end{cases}
\end{equation}
where $\epsilon = 2$ if $p = 2$ and $\epsilon = 1$ if $p \geq 3$.

\end{thm}

In order to prove this theorem, we first need some lemmas.

\begin{lemma} \label{lemma action of s}

For each index $i$ and each integer $n \notin p\zz$, the subspace of $\Hyp$ fixed pointwise by the automorphism $s_i^n$ coincides with $\bbLambda_{(i)}$.  Given any point $\eta \in \Hyp$ and letting $\xi$ be the closest point in $\bbLambda_{(i)}$ to $v$, the automorphism $s_i^n$ maps the path $[\eta, \xi]$ isometrically onto $[s_i^n(\eta), \xi]$, and we have $[\eta, \xi] \cap [s_i^n(\eta), \xi] = \{\xi\}$.

\end{lemma}

\begin{proof}
As we have $p \nmid n$, the automorphism $s_i^n$ has order $p$.  The first statement is given by \cite[Proposition 2.8(b)]{yelton2024branch}.  The fact that $s_i^n$ maps the path $[\eta, \xi]$ isometrically onto $[s_i^n(\eta), \xi]$ then follows from the metric-preserving property the action of $\PGL_2(K)$ on $\Hyp$.  Now it follows from that last statement that the intersection $[\eta, \xi] \cap [s_i^n(\eta), \xi]$ is fixed pointwise by $s_i^n$ and thus coincides with $[\eta, \xi] \cap [s_i^n(\eta), \xi] \cap \bbLambda_{(i)}$.  But this intersection consists of only the point $\xi$ by definition of $\xi$, thus proving the final statement.
\end{proof}

\begin{lemma} \label{lemma off the convex hull special case}

Choose any $\eta \in \Sigma_S$ and any nontrivial element $\gamma \in \Gamma_0$, which we write as a product 
\begin{equation} \label{eq word}
\gamma = s_{i_t}^{n_t} s_{i_{t - 1}}^{n_{t - 1}} \cdots s_{i_1}^{n_1}
\end{equation}
 for some $t \geq 1$, some $n_1, \dots, n_t \in \zz \smallsetminus p\zz$, and some indices $i_l$ satisfying $i_l \neq i_{l - 1}$ for $2 \leq l \leq t$.
 
Suppose that we have $\eta \notin \bbLambda_{(i_1)}$.  Then the closest point in $\Sigma_S$ to $\gamma(\eta)$ lies in $\bbLambda_{(i_t)}$, and we have $\gamma(\eta) \notin \Sigma_S$.

\end{lemma}

\begin{proof}
This is just a special case of \cite[Lemma 3.16]{yelton2024branch} (which requires that $S$ be optimal).
\end{proof}

\begin{lemma} \label{lemma intersection of fundamental domains}

For any indices $i \neq j$, write $D_{i, j} = [\tilde{v}_{j, i}, \tilde{v}_{i, j}] \subset \Sigma_S$.

\begin{enumerate}[(a)]

\item For any integers $m, n \in \zz$ with $m \not\equiv n$ (mod $p$), the intersection $s_{i'}^m(D_{i, i'}) \cap s_{j'}^n(D_{j, j'})$ is empty or a singleton.

\item If $i' \neq j'$, the intersection $D_{i, i'} \cap D_{j, j'}$ is empty or a singleton consisting of one of the endpoints of $D_{i, i'}$ or $D_{j, j'}$.

\item If $i' = j'$, then the intersection $D_{i, i'} \cap D_{j, j'}$ coincides with the path $[\tilde{v}_{i, i'} \vee \tilde{v}_{j, j'}, \tilde{v}_{i', i}]$ if $\tilde{v}_{i', i} = \tilde{v}_{j', j}$ and is empty otherwise.

\end{enumerate}

\end{lemma}

\begin{proof}
We have $D_{i, i'} \subset \Sigma_S$.  Now an application of either \Cref{lemma action of s} or \Cref{lemma off the convex hull special case} gives us $s_{j'}^n(D_{j, j'}) \cap \Sigma_S = \{\tilde{v}_{j', j}\}$ as long as $p \nmid n$.  This proves part (a) when $m = 0$.  The result follows for general $m$ after applying the automorphism $s_{i'}^m$ to the case where $m$ and $n$ are respectively resplaced by $0$ and $n - m$.

Note that one endpoint of the path $D_{i, i'}$ is greater than the other under the relation $>$ (that is, we have $\tilde{v}_{i', i} > \tilde{v}_{i, i'}$), and the same is true of the path $D_{j, j'}$.  Applying \Cref{rmk ordering}, it follows that if the paths $D_{i, i'}$ and $D_{j, j'}$ are not disjoint, their intersection is a path $[v, w]$ with $v \geq w$ and where $v \in \{\tilde{v}_{i', i}, \tilde{v}_{j', j}\}$.  Suppose that the path $[v, w]$ is not a singleton, or equivalently, that we have $v > w$.  Then as $v, w \in D_{i, i'}$, we get that $v$ is not the lower endpoint $\tilde{v}_{i, i'}$ of $D_{i, i'}$.  By construction of the index $i'$ from $i$, for every index $l \neq i'$, we have $(D_{i, i'} \smallsetminus \{\tilde{v}_{i, i'}\}) \cap \bbLambda_{(l)} = \varnothing$, and so we have $v \notin \bbLambda_{(l)}$ for any $l \neq i'$.  Then we get $v = \tilde{v}_{j', j} \in \bbLambda_{(j')}$ only if $j' = i'$.  By a similar argument, we get $v = \tilde{v}_{i', i}$ only if $i' = j'$.  It follows that we have $i' = j'$ under the assumption that $[v, w]$ is not a singleton, proving part (b).

Conversely, assume that $i' = j'$.  From what was observed above about the structure of the intersection $D_{i, i'} \cap D_{j, j'}$, the assertion of part (b) in the case of $\tilde{v}_{i', i} = \tilde{v}_{j', j}$ is now obvious.  If $i' = j'$ and $\tilde{v}_{i', i} \neq \tilde{v}_{j', j}$, then, as each point $\tilde{v}_{i', i}, \tilde{v}_{j', j}$ is the closest point in $\bbLambda_{(i')}$ to the respective endpoint $\tilde{v}_{i, i'}, \tilde{v}_{j, j'}$, it is clear that we have $D_{i, i'} \cap D_{j, j'} = \varnothing$, completing the proof of part (c).
\end{proof}

\begin{lemma} \label{lemma axis}

Choose indices $i, j \neq 0$, and adopt the notation of \Cref{lemma intersection of fundamental domains}.

\begin{enumerate}[(a)]

\item Fix an integer $n \in \zz$.  A fundamental domain of the axis $L_{\gamma_{i, j, n}}$ associated to the element $\gamma_{i, j, n} \in \Gamma$ with respect to the action of that element is given by the (non-backtracking) path $[s_j^{n}(\tilde{v}_{i, j}), s_j^{n-1}(\tilde{v}_{i, j})] = [s_j^n(\tilde{v}_{i, j}), \tilde{v}_{j, i}] \cup [\tilde{v}_{j, i}, s_j^{n-1}(\tilde{v}_{i, j})] = s_j^n(D_{i, j}) \cup s_j^{n-1}(D_{i, j})$; this is a path of length $2\delta(\tilde{v}_{i, j}, \tilde{v}_{j, i}) = 2\delta(\bbLambda_{(i)}, \bbLambda_{(j)})$.

\item No two points in the interior of the subspace $\mathfrak{D} := \bigcup_{1 \leq i \leq h, \, n \in \zz} s_{i'}^n(D_{i, i'}) \subset \Hyp$ are equivalent modulo the action of the group $\Gamma$.  In particular, for a given $i \neq 0$, no two points in the interior of the fundamental domain given by part (a) lie in the same orbit under $\Gamma$.

\end{enumerate}

\end{lemma}

\begin{proof}
Let us first prove part (a) in the $n = 1$ case.  For brevity of notation, (only) in the proof of part (a) we write $\gamma$ for $\gamma_{i, j, 1}$; we then have by definition $\gamma = s_i s_j^{-1}$.  Now, applying \Cref{lemma action of s}, we have $\gamma(s_j(\tilde{v}_{i, j})) = s_i(\tilde{v}_{i, j}) = \tilde{v}_{i, j}$ and $[s_j(\tilde{v}_{i, j}), \tilde{v}_{i, j}] = [s_j(\tilde{v}_{i, j}), \tilde{v}_{j, i}] \cup [\tilde{v}_{j, i}, \tilde{v}_{i, j}]$ with $\delta(s_j(\tilde{v}_{i, j}), \tilde{v}_{j, i}) = \delta(\tilde{v}_{j, i}, \tilde{v}_{i, j})$.  The length of the path $[s_j(\tilde{v}_{i, j}), \tilde{v}_{i, j})]$ then equals $2\delta(\tilde{v}_{j, i}, \tilde{v}_{i, j})$.

Now let $L_\gamma$ be the axis associated to $\gamma$, and let $\xi$ (resp. $\xi'$) be the closest point in $L_\gamma$ to $s_j(\tilde{v}_{i, j})$ (resp. $\tilde{v}_{i, j}$).  Now we have $\gamma(\xi) \in L_\gamma$ and 
\begin{equation}
\delta(\tilde{v}_{i, j}, \xi') \leq \delta(\tilde{v}_{i, j}, \gamma(\xi)) = \delta(\gamma(s_j(\tilde{v}_{i, j})), \gamma(\xi)) = \delta(s_j(\tilde{v}_{i, j}), \xi).
\end{equation}
By a symmetric argument, we get the opposite inequality, so we have $\delta(s_j(\tilde{v}_{i, j}), \xi) = \delta(\tilde{v}_{i, j}, \xi')$ (and in fact $\xi' = \gamma(\xi)$).  It follows that the point $\tilde{v}_{j, i}$, being the midpoint of the path $[s_j(\tilde{v}_{i, j}), \tilde{v}_{i, j}]$, lies in the axis $L_\gamma$.  Now, applying \Cref{lemma action of s}, we have $\gamma(\tilde{v}_{j, i}) = s_i(s_j^{-1}(\tilde{v}_{j, i})) = s_i(\tilde{v}_{j, i})$ and $[\tilde{v}_{j, i}, s_i(\tilde{v}_{j, i})] = [\tilde{v}_{j, i}, \tilde{v}_{i, j}] \cup [\tilde{v}_{i, j}, s_i(\tilde{v}_{j, i})]$ with $\delta(\tilde{v}_{j, i}, \tilde{v}_{i, j}) = \delta(\tilde{v}_{i, j}, s_i(\tilde{v}_{j, i}))$.  The length of the path $[\tilde{v}_{j, i}, s_i(\tilde{v}_{j, i})] \subset L_\gamma$ then equals $2\delta(\tilde{v}_{j, i}, \tilde{v}_{i, j}) = [s_j(\tilde{v}_{i, j}), \gamma(s_j(\tilde{v}_{i, j}))]$.  This shows that the amplitude of the loxodromic element $\gamma$ equals $2\delta(\tilde{v}_{j, i}, \tilde{v}_{i, j})$ and that we have $s_j(\tilde{v}_{i, j}) \in L_\gamma$, so that the path 
\begin{equation}
[s_j(\tilde{v}_{i, j}), \gamma(s_j(\tilde{v}_{i, j}))] = [s_j(\tilde{v}_{i, j}), \tilde{v}_{i, j}] = [s_j(\tilde{v}_{i, j}), \tilde{v}_{j, i}] \cup [\tilde{v}_{j, i}, \tilde{v}_{i, j}]
\end{equation}
is a fundamental domain of $L_\gamma$ under the action of $\gamma$.  Now the assertions of part (a) for general $n$ follow by considering that $\gamma_{i, j, n} = s_j^{n-1} \gamma s_j^{1-n}$, so the axis of $\gamma_{i, j, n}$ must be $s_j^{n-1}(L_\gamma)$ and a fundamental domain for $\gamma_{i, j, n}$ in this axis for $\gamma_{i, j, n}$ is given by the image under $s_j^{n-1}$ of any fundamental domain for $\gamma$ in the axis $L_\gamma$.  Part (a) is proved.

Now choose any point $\eta \in \mathfrak{D} \smallsetminus \{\tilde{v}_{i, i'}, \tilde{v}_{i', i}\}_{1 \leq i \leq h}$, and choose any element $\gamma \in \Gamma \smallsetminus \{1\}$.  Suppose that we have $\eta \in s_{i'}^m(D_{i, i'})$ and $\gamma(\eta) \in s_{j'}^n(D_{j, j'})$ for some indices $i, j$ and some integers $m, n \in \zz$.  Then we have $s_{i'}^{-m}(\eta) \in D_{i, i'} \subset \Sigma_S$ and $[s_{j'}^{-n} \gamma s_{i'}^m](s_{i'}^{-m}(\eta)) \in D_{j, j'} \subset \Sigma_S$.  Now as the point $\eta$ is not an endpoint of the path $s_{i'}^m(D_{i, i'})$, the point $s_{i'}^{-m}(\eta)$ is not an endpoint of the path $D_{i, i'}$.  By construction of the index $i'$ from $i$, the interior of the path $D_{i, i'}$ is disjoint from all subspaces $\bbLambda_{(l)} \subset \Hyp$.  In particular, the hypothesis of \Cref{lemma off the convex hull special case} holds for the point $s_{i'}^{-m}(\eta) \in \Sigma_S$ and element $s_{j'}^{-n} \gamma s_{i'}^m \in \Gamma_0$ provided that $s_{j'}^{-n} \gamma s_{i'}^m \neq 1$, in which case the lemma tells us that $[s_{j'}^{-n} \gamma s_{i'}^m](s_{i'}^{-m}(\eta)) \in \Sigma_S$ is a contradiction.  Therefore, we must have $s_{j'}^{-n} \gamma s_{i'}^m = 1$.  Then we have $s_{i'}^{-m}(\eta) = [s_{j'}^{-n} \gamma s_{i'}^m](s_{i'}^{-m}(\eta)) \in D_{j, j'}$, so that in fact we have $s_{i'}^{-m}(\eta) \in D_{i, i'} \cap D_{j, j'}$ (with $s_{i'}^{-m}(\eta)$ not an endpoint of either $D_{i, i'}$ or $D_{j, j'}$).  Now \Cref{lemma intersection of fundamental domains}(c) implies that this is a contradiction unless $j' = i'$.  Then $s_{j'}^{-n} \gamma s_{i'}^m = 1$ implies $\gamma = s_{i'}^{n-m}$.  As $\gamma$ was chosen to be an element of $\Gamma \smallsetminus \{1\}$, and the only power of $s_{i'}$ in $\Gamma$ is $1$, this is a contradiction.  We have thus shown that no point of $\mathfrak{D} \smallsetminus \{\tilde{v}_{i, i'}, \tilde{v}_{i', i}\}_{1 \leq i \leq h}$ lies in the orbit of any other point in $\mathfrak{D}$.

To complete the proof of part (b) it remains only to show that the points $\tilde{v}_{i, i'}, \tilde{v}_{j, j'} \in \Sigma_S$ for indices $i \neq j$ are not in the same orbit.  Suppose there is some $\gamma \in \Gamma \smallsetminus \{1\}$, which we express as a word using the notation in (\ref{eq word}), such that $\gamma(\tilde{v}_{i', i}) = \tilde{v}_{j', j}$.  If we have $i_1 = i'$, we replace $\gamma$ with $\gamma s_{i_1}^{-n_1} = \gamma s_{i'}^{-n_1} \in \Gamma_0 \smallsetminus \{1\}$, noting that we still have $\gamma(\tilde{v}_{i', i}) = \tilde{v}_{j', j}$.  Then the hypotheses of \Cref{lemma off the convex hull special case} apply to $\eta$ and $\gamma$, and that lemma tells us that we have $\gamma(\tilde{v}_{i', i}) \notin \Sigma_S$, a contradiction.
\end{proof}

Now it is reasonably straightforward to prove \Cref{thm values of pairing}.

\begin{proof}[Proof (of \Cref{thm values of pairing})]
Given generatoring elements $\gamma_{i, i', m}, \gamma_{j, j', n} \in \Gamma$, we proceed to compute the value of $e(\pi_*(v, \gamma_{i, i', m}(v)), \pi_*(w, \gamma_{j, j', n}(w)))$, which by \Cref{thm drinfeld-manin} equals $v(c_{\gamma_{i, i', m}}(\gamma_{j, j', n}))$.

Choose indices $i, j \neq 0$ and integers $m, n \in \zz$, and let $v = s_{i'}^{n-1}(\tilde{v}_{i, i'}) \in \Lambda_{\gamma_{i, i', m}}$ and $w = s_{j'}^{m-1}(\tilde{v}_{j, j'}) \in \Lambda_{\gamma_{j, j', n}}$.  Now by \Cref{lemma axis}(a), the paths 
\begin{equation} \label{eq fundamental domains}
[v, \gamma_{i, i', m}(v)] = s_{i'}^m(D_{i, i'}) \cup s_{i'}^{m-1}(D_{i, i'}) \subset L_{\gamma_{i, i', m}}, \ [w, \gamma_{j, j', n}(w)] = s_{j'}^m(D_{j, j'}) \cup s_{j'}^{m-1}(D_{j, j'}) \subset L_{\gamma_{j, j', n}}
\end{equation}
 are respectively fundamental domains for the elements $\gamma_{i, i', m}, \gamma_{j, j', n} \in \Gamma$ (using the notation of that lemma).  Meanwhile, by \Cref{lemma axis}(b), the points in the interior of their union $[v, \gamma_{i, i', m}(v)] \cup [w, \gamma_{j, j', n}(w)]$ lie in pairwise distinct orbits under the action of $\Gamma$.  It follows that the integer $e(\pi_*([v, \gamma(v)]), \pi_*([w, \gamma'(w)]))$ equals the length of the intersection $[v, \gamma_{i, i', m}(v)] \cap [w, \gamma_{j, j', n}(w)]$ of these oriented paths, now considered as paths in the graph $\mathcal{T}_\Gamma$, signed according to whether they are going in the same or opposite directions along their intersection, which we write as 
\begin{equation} \label{eq intersection of fundamental domains}
(s_{i'}^m(D_{i, i'}) \cup s_{i'}^{m-1}(D_{i, i'})) \cap (s_{j'}^m(D_{j, j'}) \cup s_{j'}^{m-1}(D_{j, j'})) = \bigcup_{\substack{s \in \{m, m-1\} \\ t \in \{n, n-1\}}} (s_{i'}^s(D_{i, i'}) \cap s_{j'}^t(D_{j, j'})).
\end{equation}
By part (a) (resp. (b); resp. (c)) of \Cref{lemma intersection of fundamental domains}, the length of the intersection $s_{i'}^s(D_{i, i'}) \cap s_{j'}^t(D_{j, j'})$ equals $0$ if $s \not\equiv t$ (mod $p$) (resp. if $i' \neq j'$; resp. if $\tilde{v}_{i', i} \neq \tilde{v}_{j', j}$ -- here we have applied parts (b) and (c) by reducing to the $s = 0$ case and then applying the automorphism $s_{i'}^s$).  We are therefore done with the proof in all of those cases.  Now if $s \equiv t$ (mod $p$) for one of the sets in the union in (\ref{eq intersection of fundamental domains}), we have that $n$ is equivalent to $m$ or $m \pm 1$ modulo $p$; when this condition on $m$ and $n$ does not hold, the lengths of all the intersections in the union in (\ref{eq intersection of fundamental domains}) are $0$, and again we are done.

We assume from now on that we have $i' = j'$ and $\tilde{v}_{i', i} = \tilde{v}_{j', j}$.  Suppose we have $m \equiv n$ (mod $p$).  Then by \Cref{lemma intersection of fundamental domains}(c), we have that the length of $D_{i, i'} \cap D_{j, j'}$ equals $\delta(\tilde{v}_{i, i'} \vee \tilde{v}_{j, j'}, \tilde{v}_{i', i}) = \delta(\tilde{v}_{i, i'} \vee \tilde{v}_{j, j'}, \bbLambda_{(i')})$.  As the automorphisms $s_{i'}^m = s_{j'}^n$ and $s_{i'}^{m-1} = s_{j'}^{n-1}$ are metric-preserving homeomorphisms, this is also the length of $s_{i'}^m(D_{i, i'}) \cap s_{j'}^n(D_{j, j'})$ and $s_{i'}^{m-1}(D_{i, i'}) \cap s_{j'}^{n-1}(D_{j, j'})$.  One checks that the paths $s_{i'}^m(D_{i, i'}), s_{j'}^n(D_{j, j'})$ (resp. $s_{i'}^{m-1}(D_{i, i'}), s_{j'}^{n-1}(D_{j, j'})$) (as sub-paths of $[v, \gamma_{i, i', m}(v)]$ and $[w, \gamma_{j, j', n}(w)]$ as in (\ref{eq fundamental domains})) are oriented to end (resp. begin) at the point $\tilde{v}_{i', i} = \tilde{v}_{j', j}$.  Thus, the total length we are computing equals $2\delta(\tilde{v}_{i, i'} \vee \tilde{v}_{j, j'}, \bbLambda_{(i')})$.

Now suppose instead that we have $n \equiv m + 1$ (mod $p$).  Then the only set in the union in (\ref{eq intersection of fundamental domains}) for which $s \equiv t$ (mod $p$) is $s_{i'}^m(D_{i, i'}) \cap s_{j'}^{n-1}(D_{j, j'})$, unless $p = 2$, in which case $s_{i'}^{m-1}(D_{i, i'}) \cap s_{j'}^n(D_{j, j'})$ is also such a set.  By essentially the same use of \Cref{lemma intersection of fundamental domains}(c) as above, the length of this intersection (or of each of these intersections in the $p = 2$ case) equals $\delta(\tilde{v}_{i, i'} \vee \tilde{v}_{j, j'}, \bbLambda_{(i')})$.  But this time the orientations are different: one checks that the path $s_{i'}^m(D_{i, i'}) \subset [v, \gamma_{i, i', m}(v)]$ ends at $\tilde{v}_{i', i}$, while the path $s_{j'}^{n-1}(D_{j, j'}) \subset [w, \gamma_{j, j', n}(w)]$ begins at that point, and thus the length of their intersection is counted as $-\delta(\tilde{v}_{i, i'} \vee \tilde{v}_{j, j'}, \bbLambda_{(i')})$.  The symmetric argument yields the same signed intersection length for the paths $s_{i'}^{m-1}(D_{i, i'})$ and $s_{j'}^n(D_{j, j'})$, which appear when $n \equiv m - 1$ (mod $p$) (including when $n \equiv m + 1$ (mod $p$) and $p = 2$).  Thus in both cases of $n \equiv m \pm 1$ (mod $p$), the total length we are computing equals $-2\delta(\tilde{v}_{i, i'} \vee \tilde{v}_{j, j'}, \bbLambda_{(i')})$ if $p = 2$ and $-\delta(\tilde{v}_{i, i'} \vee \tilde{v}_{j, j'}, \bbLambda_{(i')})$ otherwise.  This completes the proof of all cases of the formula stated in the theorem.
\end{proof}

\section{The monodromy pairing and the canonical principal polarization} \label{sec monodromy}

Fix a prime $\ell$.  Let $U_\ell \subset \Aut(T_\ell(J))$ be the subset of automorphisms $u$ satisfying that the kernel (resp. image) of $u - 1 \in \End(T_\ell(J))$ contains (resp. is contained in) the submodule $\mathbb{T}_\ell \subset T_\ell(J)$ which was defined in \S\ref{sec background pairings}.

For any automorphism $u \in U_\ell \subset \Aut(T_\ell(J))$, we construct a pairing 
\begin{equation} \label{eq pairing from u}
\langle \cdot, \cdot \rangle_{u, \ell} : \Lambda \otimes \zz_\ell \times \Lambda^\vee \otimes \zz_\ell \to \zz_\ell
\end{equation}
 as follows.  The homomorphism $u - 1$, being trivial on $\mathbb{T}_\ell$, induces a homomorphism 
\begin{equation}
\mathfrak{u} : T_\ell(J) / \mathbb{T}_\ell \to \mathbb{T}_\ell.
\end{equation}
Meanwhile, in \S\ref{sec background pairings}, we defined a map $\phi_\ell: T_\ell(J) \to \Lambda \otimes \zz_\ell$ whose kernel coincides with $\mathbb{T}_\ell$, which allows us to identify $T_\ell(J) / \mathbb{T}_\ell$ with $\Lambda \otimes \zz_\ell$.  On the other side, the fact that $T$ may be identified as the character group $\Hom(\Lambda^\vee, \ggr_m)$ (see \S\ref{sec background pairings} and consider that $J$ is the dual of $J^\vee$) allows us to identify $\mathbb{T}_\ell$ with 
\begin{equation} \label{eq identifications for T_ell(T)}
\lim_{\leftarrow n} T[\ell^n] = \lim_{\leftarrow n} \Hom(\Lambda^\vee, \ggr_m)[\ell^n] = \lim_{\leftarrow n} \Hom(\Lambda^\vee / [\ell^n]\Lambda^\vee, \boldsymbol{\mu}_{\ell^n}) = \Hom(\Lambda^\vee \otimes \zz_\ell, T_\ell(\bar{K}^\times)).
\end{equation}

In light of these identifications, we view $\mathfrak{u}$ as a homomorphism $\Lambda \otimes \zz_\ell \to \Hom(\Lambda^\vee \otimes \zz_\ell, T_\ell(\bar{K}^\times)) \cong \Hom(\Lambda^\vee \otimes \zz_\ell, \zz_\ell)$, and this defines a pairing $\langle \cdot, \cdot \rangle_{u, \ell}$ as in (\ref{eq pairing from u}).

\begin{rmk} \label{rmk pairing from u}

We make the following observations about the pairing $\langle \cdot, \cdot \rangle_{u, \ell}$ defined above.

\begin{enumerate}[(a)]

\item On checking that we have $u u' - 1 = (u - 1) + (u' - 1)$ for $u, u' \in U_\ell$, it is easy to verify that the subset $U_\ell \subset \Aut(T_\ell(J))$ is a subgroup and that the mapping $u \mapsto \langle \cdot, \cdot \rangle_{u, \ell}$ is an injective homomorphism from $U_\ell$ to the (additive) group of pairings $\Lambda \otimes \zz_\ell \times \Lambda^\vee \otimes \zz_\ell \to \zz_\ell$.

\item Suppose that $\ell$ is not the residue characteristic of $K$, and let $\sigma \in I_K$ be a Galois automorphism sending $\pi^{1/\ell^n}$ to $\zeta_{\ell^n}\pi^{1/\ell^n}$ for $n \geq 1$.  One sees by tracing the construction of the pairing $\langle \cdot, \cdot \rangle_\ell : M_\ell \times M_\ell^\vee \to \zz_\ell$ defined in \cite[equation (3.6)]{papikian2013non}, verifying from the definitions of $M_\ell$ and $M_\ell^\vee$ in \cite[\S3]{papikian2013non} that they can be identified with $\Lambda^\vee \times \zz_\ell$ and $\Lambda \times \zz_\ell$ respectively, that that pairing is the same as the pairing $\langle \cdot, \cdot \rangle_{\rho_\ell(\sigma), \ell}$ as we have defined it above.  It follows from \cite[Theorem 3.3]{papikian2013non} that there is a unique pairing $\Lambda \times \Lambda^\vee \to \zz$ (known as \textit{Grothendieck's monodromy pairing}) which extends via tensoring with $\zz_\ell$ to the pairing $\langle \cdot, \cdot \rangle_{\rho_\ell(\sigma), \ell}$.

\end{enumerate}

\end{rmk}

We now define the pairing 
\begin{equation}
\langle \cdot, \cdot \rangle_\mon : \Lambda \times \Lambda^\vee \to \zz
\end{equation}
 as follows.  For any $\lambda \in \Lambda$, and for any $\gamma \in \bar{\Gamma} = \Lambda^\vee$, we set $\langle \lambda, \gamma \rangle_\mon = v(\chi_\lambda(\gamma)).$  The following proposition implies that this is in fact the monodromy pairing of Grothendieck mentioned in \Cref{rmk pairing from u}(b); it is essentially \cite[Theorem 2.1]{coleman2000monodromy} (and \cite[Theorem 5.8]{papikian2013non}) but generalized to a statement that holds for all primes $\ell$.

Below we use the following notation.  Write $K^\times_0 < K^\times$ for the (multiplicative) group of elements of $K$ of valuation $0$.  For any prime $\ell$ and integer $n \geq 0$, write $\sqrt[\ell^n]{K^\times} < \bar{K}^\times$ (resp. $\sqrt[\ell^n]{K^\times_0} < \bar{K}^\times$) for the subgroup consisting of each element whose $\ell^n$th power lies in $K^\times$ (resp. $K^\times_0$).

\begin{prop} \label{prop monodromy pairing}

Let $\ell$ be a prime (not necessarily different from the residue characteristic of $K$), and let $\sigma \in \Gal(\bar{K} / K)$ be a Galois automorphism which fixes pointwise the subgroup $\sqrt[\ell^n]{K^\times_0} < \bar{K}^\times$ and sends $\pi^{1/\ell^n}$ to $\zeta_{\ell^n} \pi^{1/\ell^n}$ for each $n \geq 0$ and for some (any) uniformizer $\pi$ and some (any) $\ell^n$th root $\pi^{1/\ell^n}$ of $\pi$.

\begin{enumerate}[(a)]

\item The automorphism $\rho_\ell(\sigma) \in \Aut(T_\ell(J))$ lies in $U_\ell$.

\item The extension via tensoring with $\zz_\ell$ of the pairing $\langle \cdot, \cdot \rangle_\mon$ defined above is the pairing $\langle \cdot, \cdot \rangle_{\rho_\ell(\sigma), \ell}$.

\end{enumerate}

\end{prop}

\begin{proof}
Choose an element $a \in \sqrt[\ell^n]{K^\times}$, so that $a^{\ell^n} \in K^\times$ can be written as $\nu \pi^{v(a^{\ell^n})}$ for some $\nu \in K^\times_0$ and a fixed choice of uniformizer $\pi$.  Then we may write $a = \tilde{\nu} (\pi^{1/\ell^n})^{v(a^{\ell^n})}$ for some $\ell^n$th roots $\tilde{\nu} \in \sqrt[\ell^n]{K^\times_0}$ and $\pi^{1/\ell^n}$ of $\nu$ and $\pi$ respectively.  The Galois automorphism $\sigma$ fixes $\tilde{\nu}$ and sends $\pi^{1/\ell^n}$ to $\zeta_{\ell^n} \pi^{1/\ell^n}$, and it follows that we have 
\begin{equation} \label{eq Galois action}
\sigma(a) = \zeta_{\ell^n}^{v(a^{\ell^n})} a.
\end{equation}

We next observe that for each $n \geq 0$, elements of $T[\ell^n]$ are those homomorphisms $\bar{\Gamma} \to \bar{K}^\times$ which take values which are powers of $\zeta_{\ell^n}$.  Viewed as elements of $\bar{\mathbb{T}}_{\ell^n} \subset J[\ell^n]$, these points are clearly each fixed by $\sigma$ as $\sigma$ fixes each element of $\sqrt[\ell^n]{K^\times}$.  It follows that $\bar{\rho}_{\ell^n}(\sigma)$ pointwise fixes the $\zz / \ell^n\zz$-submodule $\bar{\mathbb{T}}_{\ell^n} \subset J[\ell^n]$.  Meanwhile, given any element $z \in J[\ell^n]$ lifting to an element $\tilde{z} \in T(\bar{K})$ given by a homomorphism $\chi_{\tilde{z}} : \bar{\Gamma} \to \bar{K}^\times$, we compute using (\ref{eq Galois action}) that the lifting $\widetilde{\sigma(z)} \in T(\bar{K})$ of $\sigma(z) \in J[\ell^n]$ is given by the homomorphism 
\begin{equation}
\chi_{\widetilde{g(z)}} : \bar{\Gamma} \to \bar{K}^\times, \ \ \ \ \gamma \mapsto \zeta_{\ell^n}^{v(\chi_{\tilde{z}}^{\ell^n}(\gamma))} \chi_{\tilde{z}}(\gamma).
\end{equation}
Therefore, the Galois element $\sigma$ takes $z$ to $z + w$ (recall that we use additive notation for the group $J[\ell^n]$), where $w$ lifts to the homomorphism 
\begin{equation} \label{eq character w}
\chi_{\tilde{w}} : \bar{\Gamma} \to \bar{K}^\times, \ \ \ \ \gamma \mapsto \zeta_{\ell^n}^{v(\chi_{\tilde{z}}^{\ell^n}(\gamma))}.
\end{equation}
By our above characterization of the $\zz / \ell^n\zz$-submodule $\bar{\mathbb{T}}_{\ell^n} \subset J[\ell^n]$, we have $w \in \bar{\mathbb{T}}_{\ell^n}$.  Therefore, the endomorphism $\bar{\rho}_{\ell^n}(\sigma) - 1 \in \End(J[\ell^n])$ kills $\bar{\mathbb{T}}_{\ell^n}$ and has its image contained in $\bar{\mathbb{T}}_{\ell^n}$.  On passing to the inverse limits with respect to $n$, part (a) now follows from the definition of $U_\ell$.

To prove part (b), it is enough to show that for each $n \geq 1$, given elements $\lambda \in \Lambda$ and $\gamma \in \bar{\Gamma} = \Lambda^\vee$, the values $\langle \lambda, \gamma \rangle_{\rho_\ell(\sigma), \ell}$ and $\langle \lambda, \gamma \rangle_\mon$ are equivalent modulo $\ell^n$, which is to say that we have 
\begin{equation} \label{eq pairing formula reduced}
\overline{\langle \lambda, \gamma \rangle}_{\rho_\ell(\sigma), \ell^n} = \zeta_{\ell^n}^{v(\chi_\lambda(\gamma))},
\end{equation}
where $\overline{\langle \cdot, \cdot \rangle}_{\rho_\ell(\sigma), \ell^n}$ denotes the composition of $\langle \cdot, \cdot \rangle_{\rho_\ell(\sigma), \ell}$ with reduction modulo $\ell^n$.   It is not hard to see that this reduced pairing factors through a pairing 
\begin{equation}
\overline{\langle \cdot, \cdot \rangle}_{\bar{\rho}_{\ell^n}(\sigma), \ell^n} : \Lambda / [\ell^n]\Lambda \times \Lambda^\vee / [\ell^n]\Lambda^\vee \to \boldsymbol{\mu}_{\ell^n}
\end{equation}
 which is constructed in an analogous way to $\langle \cdot, \cdot \rangle_{\rho_\ell(\sigma), \ell}$, with $\rho_\ell(\sigma)$ replaced with $\bar{\rho}_{\ell^n}(\sigma)$, with $\mathbb{T}_\ell \subset \Lambda \otimes \zz_\ell$ replaced by $\bar{\mathbb{T}}_{\ell^n} \subset \Lambda / [\ell^n]\Lambda$, with the map $\bar{\phi}_{\ell^n}$ from \S\ref{sec background pairings} allowing us to identify $J[\ell^n] / \bar{\mathbb{T}}_{\ell^n}$ with $\Lambda / [\ell^n]\Lambda$, and with the identification $\bar{\mathbb{T}}_{\ell^n} = \Hom(\Lambda^\vee / [\ell^n]\Lambda^\vee, \boldsymbol{\mu}_{\ell^n})$ analogous to that shown in (\ref{eq identifications for T_ell(T)}).

Choose any element $\lambda \in \Lambda$, and choose an element $z \in J[\ell^n]$ such that $\bar{\phi}_{\ell^n}(z) \in \Lambda / [\ell^n]\Lambda$ is the reduction of $\lambda$ modulo $\ell^n$th powers of elements of $\Lambda$.  Our characterization of $\sigma(z)$ given above shows that $\sigma(z) - z = (\bar{\rho}_{\ell^n}(\sigma) - 1)(z)$ is the point $w \in \bar{\mathbb{T}}_{\ell^n}$ corresponding to the character described in (\ref{eq character w}).  By construction of the map $\bar{\phi}_{\ell^n}$, this is the character given by $\gamma \mapsto \zeta_{\ell^n}^{\chi_\lambda(\gamma))}$, which (after identifying $\bar{\Gamma}$ with $\Lambda^\vee$) factors through $\Lambda^\vee / [\ell^n] \Lambda^\vee$ to yield a character in $\Hom(\Lambda^\vee / [\ell^n]\Lambda^\vee, \boldsymbol{\mu}_{\ell^n})$.  Then it is clear by construction of $\overline{\langle \cdot, \cdot \rangle}_{\bar{\rho}_{\ell^n}(\sigma), \ell^n}$ that we have 
\begin{equation}
\overline{\langle \bar{\lambda}, \bar{\gamma} \rangle}_{\bar{\rho}_{\ell^n}(\sigma), \ell^n} = \zeta_{\ell^n}^{v(\chi_\lambda(\gamma))},
\end{equation}
where $\bar{\lambda}$ and $\bar{\gamma}$ are the reductions of $\lambda$ and $\gamma$ respectively.  The desired formula in (\ref{eq pairing formula reduced}) follows.
\end{proof}

We next need an explicit description of the canonical principal polarization of $J$.  To this end, we make the following definition (which possibly originated as \cite[Definition 1.2]{mumford1972analytic2}).

\begin{dfn} \label{dfn polarization Mumford}

A \emph{polarization (in the sense of Mumford)} is a homomorphism $\Phi : \Lambda \to \Lambda^\vee$ satisfying the following two properties:

\begin{enumerate}[(i)]

\item we have $\chi_{\Phi(\lambda)}(\lambda') = \chi_{\Phi(\lambda')}(\lambda)$ for $\lambda, \lambda' \in \Lambda$; and 

\item we have $v(\chi_{\Phi(\lambda)}(\lambda)) > 0$ for $\lambda \in \Lambda \smallsetminus \{1\}$.

\end{enumerate}

\end{dfn}

One checks using property (i) of the above definition that the map $\Phi^* : T = \Hom(\Lambda^\vee, \ggr_m) \to \Hom(\Lambda, \ggr_m) = T^\vee$ induced via pullback by a polarization $\Phi: \Lambda \to \Lambda^\vee$ restricts to $\Phi$ on $\Lambda < T$.  This in turn induces a homomorphism on quotients which (by abuse of notation) we also denote by $\Phi: J = T / \Lambda \to T^\vee / \Lambda^\vee = J^\vee$.  It follows easily from property (ii) of the above definition that the map $\Phi: \Lambda \to \Lambda^\vee$ is injective and therefore has finite cokernel; the map $\Phi: J \to J^\vee$ is therefore surjective with finite kernel.  One sees from \cite[Theorem 2.4 and its proof]{bosch1991degenerating} that the homomorphism $\Phi: J \to J^\vee$ is in fact a polarization (in the usual sense) on the abelian variety $J$, which is principal if and only if $\Phi: \Lambda \to \Lambda^\vee$ is surjective.

\begin{rmk} \label{rmk transvections on toric part}

For each $n \geq 0$, one sees from the above discussion that the canonical principal polarization on $J$, being an isomorphism, restricts to an isomorphism $\bar{\mathbb{T}}_{\ell^n} \stackrel{\sim}{\to} \bar{\mathbb{T}}_{\ell^n}^\vee$.  Therefore, the isomorphism $T_\ell(J) \stackrel{\sim}{\to} T_\ell(J^\vee)$ induced from the canonical principal polarization by taking inverse limits maps $\mathbb{T}_\ell$ to $\mathbb{T}_\ell^\vee$.  It now follows from \Cref{cor explicit Weil pairing} that for any $\mathfrak{w} \in \mathbb{T}_\ell$, the submodule $\mathbb{T}_\ell \subset T_\ell(J)$ is killed by the homomorphism $e_\ell(\cdot, \mathfrak{w}) = \mathfrak{e}_\ell(\cdot, \varphi(\mathfrak{w})) : T_\ell(J) \to T_\ell(\bar{K}^\times)$, where the pairing $e_\ell : T_\ell(J) \times T_\ell(J) \to T_\ell(\bar{K}^\times)$ is defined as in \S\ref{sec intro superelliptic}.  In particular, the restriction of the pairing $e_\ell$ to $\mathbb{T}_\ell \times \mathbb{T}_\ell$ is trivial.

\end{rmk}

Now let $\varphi : \Lambda \to \Lambda^\vee$ be the homomorphism given by $c_\lambda \mapsto \ev_\lambda$ (where $\ev_\lambda$ denotes the evaluation-at-$\lambda$ character as in \Cref{rmk Lambda dual is barGamma}).  By construction of $\Lambda$ and by \Cref{rmk Lambda dual is barGamma}, this map $\varphi$ is an isomorphism, and by \cite[Theorem 1 and Theorem 4]{drinfeld1973periods}, it is a polarization in the sense of Mumford and thus induces a principal polarization $\varphi : J \stackrel{\sim}{\to} J^\vee$.  The following fact is stated only very implicitly in \cite[\S7]{papikian2013non}.

\begin{prop} \label{prop canonical principal polarization}

The homomorphism $\varphi : J \stackrel{\sim}{\to} J^\vee$ is the canonical principal polarization of $J$.

\end{prop}

\begin{proof}
In this proof and elsewhere below, given any pairing $\langle \cdot, \cdot \rangle: \Lambda \times \Lambda^\vee \to \zz$ and any principal polarization $\Phi : \Lambda \stackrel{\sim}{\to} \Lambda^\vee$, we denote by $\langle \cdot, \cdot \rangle^\Phi : \Lambda^\vee \times \Lambda^\vee \to \zz$ the pairing given by composing $\langle \cdot, \cdot \rangle$ with the inverse $\Phi^{-1} : \Lambda^\vee \stackrel{\sim}{\to} \Lambda$ applied to the first argument.

Let $\varphi' : \Lambda \to \Lambda^\vee$ be the polarization in the sense of Mumford which induces the canonical principal polarization on $J$.  One deduces from \cite[Th\'{e}or\`{e}me 12.5]{grothendieck1972modeles} (the Picard-Lefschetz formula) a description (given in \cite[\S3.4]{papikian2013non}) of the pairing 
\begin{equation}
\langle \cdot, \cdot \rangle_\mon^{\varphi'} : \Lambda^\vee \times \Lambda^\vee \to \zz
\end{equation}
via identifying $\Lambda^\vee$ with the first singular homology group of the dual graph of the special fiber of $C$.  This in turn is equivalent to the definition of the pairing $(\gamma, \gamma') \mapsto e(\pi_*([v, \gamma(v)]), \pi_*([w, \gamma'(w)]))$ from \S\ref{sec period matrix}, on identifying $\Lambda^\vee$ with $\bar{\Gamma}$: see \cite[Theorem 6.8 and its proof; discussion in \S7.3]{papikian2013non}.  By \Cref{thm drinfeld-manin}, this pairing is the same as the one given by $(\gamma, \gamma') \mapsto v(c_\gamma(\gamma'))$.

Meanwhile, thanks to \Cref{prop monodromy pairing}(b), the pairing $\langle \cdot, \cdot \rangle_\mon^\varphi$ is given by 
\begin{equation}
(\gamma, \gamma') \mapsto v(\chi_{\varphi^{-1}(\gamma)}(\gamma')) = v(c_\gamma(\gamma')).
\end{equation}
This means that the pairings $\langle \cdot, \cdot \rangle_\mon^\varphi$ and $\langle \cdot, \cdot \rangle_\mon^{\varphi'}$ are equal.  The fact that $\varphi = \varphi'$ now follows straightforwardly from the nondegeneracy of this pairing implied by \Cref{thm drinfeld-manin}.
\end{proof}

\begin{cor} \label{cor action of sigma pairing}

Let $\sigma \in \Gal(\bar{K} / K)$ be an element satisfying the properties given in the hypothesis of \Cref{prop monodromy pairing}.  The pairing $\langle \cdot, \cdot \rangle_{\rho_\ell(\sigma), \ell}^\varphi$, viewed via the identification given in \Cref{rmk Lambda dual is barGamma} as a pairing $\bar{\Gamma} \times \bar{\Gamma} \to \zz$ tensored with $\zz_\ell$, is given by 
\begin{equation}
(\gamma, \gamma') \mapsto v(c_\gamma(\gamma')) \in \zz.
\end{equation}

\end{cor}

\begin{proof}
This is just \Cref{prop monodromy pairing}(b) combined with \Cref{prop canonical principal polarization}.
\end{proof}

As we want to equate the operator $\rho_\ell(\sigma) \in \Aut(T_\ell(J))$ with a linear operator on $T_\ell(J)$ given by a product of transvections, our final objective in this section is to provide a concrete description of the action on $T_\ell(J)$ of a transvection with respect to an element of $\mathbb{T}_\ell$.

\begin{prop} \label{prop explicit transvection}

Let $\ell$ be a prime (not necessarily different from the residue characteristic of $K$); choose an element $\mathfrak{w} \in \mathbb{T}_\ell = \Hom(\Lambda^\vee \otimes \zz_\ell, \zz_\ell)$; and let $t_{\mathfrak{w}} \in \Aut(T_\ell(J))$ be the associated transvection as defined in \S\ref{sec intro}.

\begin{enumerate}[(a)]

\item We have $t_{\mathfrak{w}} \in U_\ell$.

\item Let $\langle \cdot, \cdot \rangle_{t_{\mathfrak{w}}, \ell}^\varphi : \Lambda^\vee \times \Lambda^\vee \to \zz_\ell$ denote the extension via tensoring with $\zz_\ell$ of the pairing $\langle \cdot, \cdot \rangle_{t_{\mathfrak{w}}}^\varphi$ defined as in the proof of \Cref{prop canonical principal polarization}.  Then this pairing is given by the formula 
\begin{equation} \label{eq formula transvection}
\langle \mu, \mu' \rangle_{t_{\mathfrak{w}}, \ell}^\varphi = \chi_{\mathfrak{w}}(\mu) \cdot \chi_{\mathfrak{w}}(\mu'), \ \ \ \ \mu, \mu' \in \Lambda^\vee \otimes \zz_\ell.
\end{equation}

\end{enumerate}

\end{prop}

\begin{proof}
It follows from \Cref{rmk transvections on toric part} and the definition of the transvection $t_{\mathfrak{w}}$ that the endomorphism $t_{\mathfrak{w}} - 1 \in \End(T_\ell(J))$ kills $\mathbb{T}_\ell$.  Meanwhile, it is immediate from the definition of $t_{\mathfrak{w}}$ that the image of $t_{\mathfrak{w}} - 1 \in \End(T_\ell(J))$ lies in $\mathbb{T}_\ell$.  This proves part (a).

Choose any elements $\mu, \mu' \in \Lambda^\vee$.  Now let $\mathfrak{v} \in T_\ell(J)$ be an element satisfying $\phi_\ell(\mathfrak{v}) = \varphi^{-1}(\mu)$, where $\phi_\ell : T_\ell(J) \to \Lambda \otimes \zz_\ell$ is the homomorphism defined in \S\ref{sec background pairings}.  By \Cref{prop explicit Weil pairing}, we have 
\begin{equation}
e_\ell(\mathfrak{v}, \mathfrak{w}) = \mathfrak{e}_\ell(\mathfrak{v}, \varphi(\mathfrak{w})) = \chi_{\varphi(\mathfrak{w})}(\varphi^{-1}(\mu)) = \chi_{\varphi(\varphi^{-1}(\mu))}(\mathfrak{w}) = \chi_\mu(\mathfrak{w}) = \chi_{\mathfrak{w}}(\mu),
\end{equation}
where the third-to-last equality comes from the property \Cref{dfn polarization Mumford}(i).  Therefore, the endomorphism $t_{\mathfrak{w}} - 1 \in \End(T_\ell(J))$ sends $\mathfrak{v}$ to the element $\chi_{\mathfrak{w}}(\mu) \mathfrak{w} \in \mathbb{T}_\ell$, which itself gets identified with the character $\chi_{\mathfrak{w}}(\mu) \chi_{\mathfrak{w}} : \Lambda^\vee \to \zz_\ell$.  Now by construction, we get 
\begin{equation}
\langle \mu, \mu' \rangle_{t_{\mathfrak{w}}, \ell}^\varphi = \langle \varphi^{-1}(\mu), \mu' \rangle_{t_{\mathfrak{w}}, \ell} = \chi_{\mathfrak{w}}(\mu) \chi_{\mathfrak{w}}(\mu').
\end{equation}
\end{proof}

\section{Proof of the main claim of \Cref{thm main}} \label{sec proof of main}

This section uses the results we have built up in previous sections to prove the formula for $\rho_\ell(\sigma)$ claimed by \Cref{thm main}.  Fixing a prime $\ell$ and reprising the notation of \S\ref{sec background},\ref{sec monodromy}, we define the elements $\mathfrak{v}_0, \dots, \mathfrak{v}_i \in \mathbb{T}_\ell \subset T_\ell(J)$ in the following manner.  Identifying $\mathbb{T}_\ell$ with $\Hom(\Lambda^\vee \otimes \zz_\ell, T_\ell(\bar{K}^\times))$ as in (\ref{eq T_ell}), $\Lambda^\vee$ with $\bar{\Gamma}$ as in \Cref{rmk Lambda dual is barGamma}, and $T_\ell(\bar{K}^\times)$ with $\zz_\ell$, for $0 \leq i \leq h$, the character $\mathfrak{v}_i \in \mathbb{T}_\ell$ is the one determined by the character $\chi_{\mathfrak{v}_i} : \bar{\Gamma} \to \zz_\ell$ given by 
\begin{equation} \label{eq v_i}
\chi_{\mathfrak{v}_i} : 
\begin{cases}
\gamma_{i, i', 0} &\mapsto -1 \\
\gamma_{i, i', 1} &\mapsto 1 \\
\gamma_{j, j', m} &\mapsto 0 \text{ if } j \neq i \text{ or } m \not\equiv 0, 1 \ (\text{mod} \ p).
\end{cases}
\end{equation}
Note that this character is well defined: \Cref{cor basis of barGamma} implies that the set $\{\gamma_{j, j', m}\}_{1 \leq j \leq h, \, 0 \leq m \leq p - 1}$ generates $\bar{\Gamma}$ with the $h$ (independent) relations given by $\gamma_{j, j', p-1} \cdots \gamma_{j, j', 0} = 1$ for $1 \leq j \leq h$, and the map described by (\ref{eq v_i}) is compatible with these relations.

\begin{lemma} \label{lemma zeta_p^n v_i}

For $1 \leq i \leq h$ and $n \in \zz$, we have $\zeta_p^n \mathfrak{v}_i \in \mathbb{T}_\ell$, and this character is the one determined by the character $\chi_{\zeta_p^n \mathfrak{v}_i} : \bar{\Gamma} \to \zz_\ell$ given by 
\begin{equation} \label{eq zeta_p^n v_i}
\chi_{\zeta_p^n \mathfrak{v}_i} : 
\begin{cases}
\gamma_{i, i', -n} &\mapsto -1 \\
\gamma_{i, i', 1 - n} &\mapsto 1 \\
\gamma_{j, j', m} &\mapsto 0 \text{ if } j \neq i \text{ or } m + n \not\equiv 0, 1 \ (\text{mod} \ p).
\end{cases}
\end{equation}

\end{lemma}

\begin{proof}
This follows from applying \Cref{prop automorphism zeta_p} and using the fact that $s_{i'}^n \gamma_{i, i', m} s_{i'}^{-n} = \gamma_{i, i', m+n}$ for $1 \leq i \leq h$ and $m \in \zz$.
\end{proof}

Given an index $i \in \{1, \dots, h\}$, we define the index $i'$ and the points $v_i, \hat{v}_i, \tilde{v}_{i, i'}, \tilde{v}_{i', i}$ as in \S\ref{sec period matrix}, and we set the following notation.  Let $v_i^{(1)}, \dots, v_i^{(r)}$ be the vertices in the interior of the path $[\tilde{v}_{i, i'}, \tilde{v}_{i', i}]$, ordered so that we have 
\begin{equation}
\tilde{v}_{i', i} > v_i^{(r)} > \dots > v_i^{(1)} > \tilde{v}_{i, i'} \geq v_i^{(0)} := v_i
\end{equation}
for some $r \geq 0$.  One sees using \Cref{cor isometric connected components} that the images $\bar{v}_i^{(s)}$ of the points $v_i^{(s)}$ under $\pi_*$ for $1 \leq s \leq r$ are likewise vertices in the interior of the path $[\tilde{\bar{v}}_{i, i'}, \tilde{\bar{v}}_{i', i}]$, ordered so that we have 
\begin{equation}
\tilde{\bar{v}}_{i', i} > \bar{v}_i^{(r)} > \dots > \bar{v}_i^{(1)} > \tilde{\bar{v}}_{i, i'} \geq \bar{v}_i^{(0)} := \bar{v}_i.
\end{equation}
By \Cref{prop dictionary}(d), the (non-distinguished) vertices $v_i^{(s)}$ for $1 \leq j \leq r$ correspond to all of the (\"{u}bereven) clusters $\mathfrak{s}_i^{(j)} \in \mathfrak{C}_0$ strictly containing $\mathfrak{s}_i$ and contained in the disc corresponding to $\tilde{\bar{v}}_{i', i}$; these satisfy the inclusions 
\begin{equation} \label{eq chain of cluster inclusions}
\mathfrak{s}_i^{(r)} \supsetneq \dots \supsetneq \mathfrak{s}_i^{(1)} \supsetneq \mathfrak{s}_i^{(0)} := \mathfrak{s}_i.
\end{equation}

\begin{lemma} \label{lemma hard-won}

With the above set-up and the notation of \Cref{thm main}, we have $\delta(\tilde{v}_{i, i'}, \bbLambda_{(i')}) = \sum_{0 \leq s \leq r} m_{\mathfrak{s}_i^{(s)}}$ and for any $r_0 \in \{1, \dots, r\}$, we have $\delta(v_i^{(r_0)}, \bbLambda_{(i')}) = \sum_{r_0 \leq s \leq r} m_{\mathfrak{s}_i^{(s)}}$.

\end{lemma}

\begin{proof}
By construction, given $s \in \{0, \dots, r\}$ and for the moment writing $\bar{w}$ for $\bar{v}_i^{(s)}$, the point $\bar{w}'$ defined in \Cref{cor uniqueness used in main theorem} is $\bar{v}_i^{(s+1)}$ if $s \neq r$ and is the closest point in $\bar{\Lambda}_{(i')}$ to $\tilde{\bar{v}}_{i', i}$ if $s = r$.  In the latter case, let us denote this point by $\bar{v}_i^{(r+1)}$ and its corresponding cluster by $\mathfrak{s}_i^{(r+1)}$.  Subsequently, \Cref{cor uniqueness used in main theorem} tells us that we have $(\mathfrak{s}_i^{(s)})' = \mathfrak{s}_i^{(s+1)}$ for $0 \leq s \leq r$.

Now by construction of $\hat{\bar{v}}_i$, we have $\delta(\bar{v}_i, \hat{\bar{v}}_i) = \frac{pv(p)}{p-1}$, which thanks to \Cref{prop v_i etc.}(b) can be rewritten as $\delta(\bar{v}_i, \tilde{\bar{v}}_{i, i'}) = \frac{pv(p)}{p-1}$.  Meanwhile, by \Cref{prop v_i etc.}(c), we have $\delta(\bar{v}_i^{(r+1)}, \tilde{\bar{v}}_{i', i}) = \frac{pv(p)}{p-1}$.  We then have 
\begin{equation} \label{eq hard-won1}
\begin{aligned}
\delta(\tilde{v}_{i, i'}, \bbLambda_{(i')}) = \delta(\tilde{v}_{i, i'}, \tilde{v}_{i', i}) = \delta(\tilde{\bar{v}}_{i, i'}, \tilde{\bar{v}}_{i', i}) &\text{\indent by \Cref{thm distances between neighboring axes}} \\
= \delta(\bar{v}_i^{(0)}, \bar{v}_i^{(r+1)}) - 2\tfrac{pv(p)}{p-1} & \text{\indent by observations directly above} \\
= \sum_{s = 0}^r \big(\delta(\bar{v}_i^{(s)}, \bar{v}_i^{(s+1)}) - (2 - u_s) \tfrac{pv(p)}{p-1}\big) & \\
= \sum_{s = 0}^r \big(d(\mathfrak{s}_i^{(s)}) + d((\mathfrak{s}_i^{(s)})') - 2d(\mathfrak{s}_i^{(s)} \vee (\mathfrak{s}_i^{(s)})') - (2 - u_s) \tfrac{pv(p)}{p-1}\big) & \text{\indent using \Cref{prop dictionary}(a)} \\
= \sum_{s = 0}^r m_{\mathfrak{s}_i^{(s)}}, &
\end{aligned}
\end{equation}
where $u_s$ is a shorthand for $2 - \#(\{s, s + 1\} \cap \{0, r\}) = u_{\{\mathfrak{s}_i^{(s)}, \, \mathfrak{s}_i^{(s+1)}\}}$, the number of \"{u}bereven clusters among $\{\mathfrak{s}_i^{(s)}, \mathfrak{s}_i^{(s+1)}\}$.  Through a similar computation, for any $r_0 \in \{1, \dots, r\}$, we get 
\begin{equation} \label{eq hard-won2}
\delta(v_i^{(r_0)}, \bbLambda_{(i')}) = \delta(v_i^{(r_0)}, \tilde{v}_{i', i}) = \delta(\bar{v}_i^{(r_0)}, \tilde{\bar{v}}_{i', i}) = \delta(\bar{v}_i^{(r_0)}, \bar{v}_i^{(r+1)}) - \tfrac{pv(p)}{p-1} = \sum_{s = r_0}^r m_{\mathfrak{s}_i^{(s)}}.
\end{equation}
\end{proof}

The main step remaining in the proof is to verify the following.

\begin{lemma} \label{lemma monodromy pairing equals sum of transvection pairings}

With the elements $\mathfrak{v}_i \in \mathbb{T}_\ell$ defined as above, the notation of $\chi_\mathfrak{w} : \bar{\Gamma} \to \zz_\ell$ for any element $\mathfrak{w} \in \mathbb{T}_\ell$, and all of the notation used in the statement of \Cref{thm main}, we have the formula 
\begin{equation} \label{eq monodromy pairing equals sum of transvection pairings}
v(c_\gamma(\gamma')) = \sum_{\mathfrak{s} \in \mathfrak{C}_0} \Big(\sum_{0 \leq n \leq p - 1} m_{\mathfrak{s}} \chi_{\zeta_p^n \mathfrak{w}_{\mathfrak{s}}}(\gamma) \chi_{\zeta_p^n \mathfrak{w}_{\mathfrak{s}}}(\gamma')\Big) \ \ \text{for all} \ \gamma, \gamma' \in \bar{\Gamma}.
\end{equation}

\end{lemma}

\begin{proof}

Thanks to \Cref{cor basis of barGamma}, in order to get the equality of pairings on $\bar{\Gamma}$ given in (\ref{eq monodromy pairing equals sum of transvection pairings}), it suffices to prove that it holds for $\gamma = \gamma_{i, i', \mu}, \gamma' = \gamma_{j, j', \nu}$ for $1 \leq i, j \leq h$ and $\mu, \nu \in \zz$.  Below we freely use \Cref{lemma zeta_p^n v_i} in computing the images of these elements of $\bar{\Gamma}$ under the characters $\chi_{\zeta_p^n \mathfrak{w}_s}$.

The only clusters $\mathfrak{s} \in \mathfrak{C}_0$ and values of $n \in \zz$ such that we have $\chi_{\zeta_p^n \mathfrak{w}_{\mathfrak{s}}}(\gamma_{i, i', \mu}) \neq 0$ are the clusters $\mathfrak{s}_i^{(s)}$ appearing in (\ref{eq chain of cluster inclusions}) and the integers $n$ equivalent to $-\mu$ or $-\mu + 1$ modulo $p$ with 
\begin{equation} \label{eq summands}
\chi_{\zeta_p^{-\mu} \mathfrak{w}_{\mathfrak{s}}}(\gamma_{i, i', \mu}) = -1, \ \ \chi_{\zeta_p^{-\mu + 1} \mathfrak{w}_{\mathfrak{s}}}(\gamma_{i, i', \mu}) = 1.
\end{equation}
The only clusters $\mathfrak{s} \in \mathfrak{C}_0$ and integers $n \in \zz$ such that the summand $m_{\mathfrak{s}} \chi_{\zeta_p^n \mathfrak{w}_{\mathfrak{s}}}(\gamma_{i, i', \mu}) \chi_{\zeta_p^n \mathfrak{w}_{\mathfrak{s}}}(\gamma_{j, j', \nu})$ does not necessarily equal $0$ are therefore the clusters which either equal $\mathfrak{s}_i$ (in the case that $i = j$) or equal $\mathfrak{s}_i^{(r)} = \mathfrak{s}_j^{(t)}$ for some $r, t \geq 1$, and the integers $n$ equivalent modulo $p$ to an integer in $\{-\mu, -\mu + 1\} \cap \{-\nu, -\nu + 1\}$ (thus requiring in particular that $\nu \equiv \mu$ or $\nu \equiv \mu \pm 1$ (mod $p$)).  Note that if $i \neq j$, we have $\mathfrak{s}_i^{(r)} = \mathfrak{s}_j^{(t)}$ (or equivalently $v_i^{()} = v_j^{(t)}$) for some $r, t \geq 1$ only if the intersection $v_i^{(r)} \in [\tilde{v}_{i, i'}, \tilde{v}_{i', i}] \cap [\tilde{v}_{j, j'}, \tilde{v}_{j', j}]$ is a path with interior and $v_i^{(r)} \geq \tilde{v}_{i, i'} \vee \tilde{v}_{j, j'}$, which by \Cref{lemma intersection of fundamental domains}(b) implies that $i' = j'$.  In particular, the point $\tilde{v}_{i, i'} \vee \tilde{v}_{j, j'}$ would have to be in the interior of $[\tilde{v}_{i, i'}, \tilde{v}_{i', i} = \tilde{v}_{j', j}]$, implying that $\tilde{v}_{i, i'} \vee \tilde{v}_{j, j'} \notin \bar{\bbLambda}_{(i')}$.

The above paragraph implies that if we have $\tilde{v}_{i', i} \neq \tilde{v}_{j', j}$ or if $\nu$ is not equivalent modulo $p$ to $\mu$ or $\mu \pm 1$, then the right-hand side of the formula in (\ref{eq monodromy pairing equals sum of transvection pairings}) comes out to $0$.  Now suppose conversely that $\tilde{v}_{i, i'} \vee \tilde{v}_{j, j'} \leq \tilde{v}_{i', i} = \tilde{v}_{j', j}$ (the second equality of which implies that $i' = j'$ by \Cref{thm disjointness of neighborhoods}) and that we have $\nu \equiv \mu$ or $\nu \equiv \mu \pm 1$ (mod $p$); we proceed to compute the right-hand side of (\ref{eq monodromy pairing equals sum of transvection pairings}).

If $i = j$, then we have $\bar{v}_i \vee \bar{v}_j = \bar{v}_i = \bar{v}_i^{(0)}$.  If $i \neq j$, then we clearly have (using \Cref{prop v_i etc.}(b) and \Cref{thm disjointness of neighborhoods}) the equalities $\tilde{v}_{i, i'} \vee \tilde{v}_{j, j'} = \hat{v}_i \vee \hat{v}_j = v_i \vee v_j$.  Applying \Cref{cor isometric connected components}, we get $\tilde{\bar{v}}_{i, i'} \vee \tilde{\bar{v}}_{j, j'} = \bar{v}_i \vee \bar{v}_j$; by \Cref{prop dictionary}(a), this point is the one corresponding to the cluster $\mathfrak{s}_i \vee \mathfrak{s}_j$ and is therefore a vertex satisfying $\tilde{\bar{v}}_{j', j} = \tilde{\bar{v}}_{i', i} > \bar{v}_i \vee \bar{v}_j > \tilde{\bar{v}}_{i, i'}, \tilde{\bar{v}}_{j, j'}$.  In either case, we may therefore write $\bar{v}_i \vee \bar{v}_j = \bar{v}_i^{(r_0)} = \bar{v}_j^{(t_0)}$ for some $r_0, t_0 \geq 0$.  It follows that the only clusters $\mathfrak{s}$ such that the summands $m_{\mathfrak{s}} \chi_{\zeta_p^n \mathfrak{w}_{\mathfrak{s}}}(\gamma_{i, i', \mu}) \chi_{\zeta_p^n \mathfrak{w}_{\mathfrak{s}}}(\gamma_{j, j', \nu})$ for $0 \leq n \leq p - 1$ do not necessarily each equal $0$ are $\mathfrak{s}_i^{(r_0)} \subsetneq \dots \subsetneq \mathfrak{s}_i^{(r)}$, as these are the only clusters $\mathfrak{s} \supseteq \mathfrak{s}_i \vee \mathfrak{s}_j = \mathfrak{s}_i^{(r_0)}$ with $\mathfrak{s} \not\supset \mathfrak{s}_l$ for any $l \neq i, j$ such that $\mathfrak{s}_l$ contains $\mathfrak{s}_i$ or $\mathfrak{s}_j$.

Suppose that $\mu \equiv \nu$ (mod $p$).  Then for a given $\mathfrak{s} \in \{\mathfrak{s}_i^{(s)}\}_{r_0 \leq s \leq r}$, from (\ref{eq summands}) we have 
\begin{equation}
\begin{aligned}
m_{\mathfrak{s}} \chi_{\zeta_p^{-\mu} \mathfrak{w}_{\mathfrak{s}}}(\gamma_{i, i', \mu}) \chi_{\zeta_p^{-\mu} \mathfrak{w}_{\mathfrak{s}}}(\gamma_{j, j', \nu}) &= m_{\mathfrak{s}} (-1) (-1) = m_{\mathfrak{s}}, \\
m_{\mathfrak{s}} \chi_{\zeta_p^{-\mu + 1} \mathfrak{w}_{\mathfrak{s}}}(\gamma_{i, i', \mu}) \chi_{\zeta_p^{-\mu + 1} \mathfrak{w}_{\mathfrak{s}}}(\gamma_{j, j', \nu}) &= m_{\mathfrak{s}} (1) (1) = m_{\mathfrak{s}},
\end{aligned}
\end{equation}
and the analogous summand for $n$ not equivalent to $-\mu$ or $-\mu + 1$ modulo $p$ equals $0$.  The summands in the inner sum of the right-hand side of (\ref{eq monodromy pairing equals sum of transvection pairings}) for this cluster $\mathfrak{s}$ therefore add up to $2m_{\mathfrak{s}}$.  Then, by \Cref{lemma hard-won}, the whole sum in that formula equals 
\begin{equation}
2\sum_{s = r_0}^r m_{\mathfrak{s}_i^{(s)}} = 2\delta(\tilde{v}_{i, i'} \vee \tilde{v}_{j, j'}, \bar{\bbLambda}_{(i')}).
\end{equation}

Now suppose that $\nu \equiv \mu + 1$ (mod $p$).  Then for a given $\mathfrak{s} \in \{\mathfrak{s}_i^{(s)}\}_{r_0 \leq s \leq r}$, from (\ref{eq summands}) we have 
\begin{equation}
m_{\mathfrak{s}} \chi_{\zeta_p^{-\mu} \mathfrak{w}_{\mathfrak{s}}}(\gamma_{i, i', \mu}) \chi_{\zeta_p^{-\mu} \mathfrak{w}_{\mathfrak{s}}}(\gamma_{j, j', \nu}) = m_{\mathfrak{s}} (-1) (1) = -m_{\mathfrak{s}},
\end{equation}
and that the analogous summand for $n$ not equivalent to $-\mu$ modulo $p$ equals $0$.  The symmetric statement then clearly holds in the case that $\mu \equiv \nu + 1$ (mod $p$).  Noting that these two cases are equivalent when $p = 2$ but mutually exclusive otherwise, we see that the summands in the inner sum of the right-hand side of (\ref{eq monodromy pairing equals sum of transvection pairings}) for this cluster $\mathfrak{s}$ therefore add up to $-m_{\mathfrak{s}}$ (resp. $-2m_{\mathfrak{s}}$) if $p \neq 2$ (resp. $p = 2$) and $\nu \equiv \mu \pm 1$ (mod $p$).  Then, by \Cref{lemma hard-won}, the whole sum in that formula equals 
\begin{equation}
-\epsilon \sum_{s = r_0}^r m_{\mathfrak{s}_i^{(s)}} = -\epsilon\delta(\tilde{v}_{i, i'} \vee \tilde{v}_{j, j'}, \bar{\bbLambda}_{(i')}),
\end{equation}
with $\epsilon = 2$ if $p = 2$ and $\epsilon = 1$ otherwise.

In each of these cases, the value computed matches that given by \Cref{thm values of pairing} for $v(c_{\gamma_{i, i', \mu}}(\gamma_{j, j', \nu}))$, and the lemma is proved.
\end{proof}

Let $\sigma \in I_K$ be as in the statement of \Cref{thm main},  Now by applying \Cref{cor action of sigma pairing} to the left-hand side and \Cref{prop explicit transvection}(b) to the right-hand side of (\ref{eq monodromy pairing equals sum of transvection pairings}), and by applying the polarization $\varphi : \Lambda \stackrel{\sim}{\to} \Lambda^\vee$ (tensored with $\zz_\ell$) to the first argument for each pairing, \Cref{lemma monodromy pairing equals sum of transvection pairings} tells us that 
\begin{equation}
\langle \gamma, \gamma' \rangle_{\rho_\ell(\sigma), \ell} = \sum_{\mathfrak{s} \in \mathfrak{C}_0} \Big(\sum_{0 \leq n \leq p - 1} m_{\mathfrak{s}} \langle \gamma, \gamma' \rangle_{t_{\mathfrak{w}_{\mathfrak{s}}}, \ell} \Big) \ \ \text{for all} \ \gamma, \gamma' \in \bar{\Gamma}.
\end{equation}
Now the formula in (\ref{eq main product of transvections}) for $\rho_\ell(\sigma)$ claimed by \Cref{thm main} follows from \Cref{rmk pairing from u}(a).

\section{The $(1 - \zeta_p)$-torsion subgroup of the Jacobian} \label{sec torsion}

The goal of this section is to explicitly describe certain elements of the $p$-torsion group $J[p]$ as points in the quotient $T / \Lambda$ constructed in \S\ref{sec background jacobians}, namely those represented by divisors of the form $(\alpha_i, 0) - (\beta_i, 0) \in \Div_0(C)$.  Apart from the self-contained and elementary \Cref{lemma automorphy van steen}, the results of this section are used only to prove the property of the basis of the submodule $\mathbb{T}_\ell \subset T_\ell(J)$ asserted in the last statement of \Cref{thm main} (which only pertains to the $\ell = p$ case), and nothing else in this paper depends on them.

As the automorphism $\zeta_p$ of the superelliptic curve $C$ fixes each of its ramification points over $\proj_K^1$, any divisor in $\Div_0(C)$ of the form $(P) - (P')$, where $P, P' \in C(K)$ are ramification points, is also fixed by $\zeta_p$.  Therefore, the image of such a divisor in the Jacobian $J$ is killed by the endomorphism $1 - \zeta_p \in \End(J)$.  (In fact, the results of \cite[\S2.3]{arul2020explicit} describe the so-called $(1 - \zeta_p)$-torsion subgroup $J[1 - \zeta_p]$ as being generated by the images in $J$ of the set of degree-$0$ divisors of the form $(P) - (\infty)$ for ramification points $P \neq \infty$, where $\infty$ is the ramification point lying over the branch point $\beta_0 = \infty$; this is shown in a situation where the defining polynomial $f(x)$ appearing in (\ref{eq superelliptic}) is separable rather than for split degenerate superelliptic curves, but the arguments involved work in the same way when $C$ is determined by an equation of the form in (\ref{eq split degenerate superelliptic}).)  In particular, since we have $(1 - \zeta_p) \mid p$ in the ring $\zz[\zeta_p] \subset \End(J)$, these elements lie in the $p$-torsion subgroup $J[p]$ and generate a $\zz/p\zz$-submodule of $J[p]$ of rank $\frac{2g}{p-1} = 2h$.  (The fact that these points are $p$-torsion can also be checked directly by noting that for any ramification point $P = (\alpha, 0) \in C(K)$, the divisor $p(P) - p(\infty)$ is principal, being the divisor of the function $x - \alpha$.)

The main result of this section, \Cref{thm certain p-torsion points} below, describes the images under the analytic Abel-Jacobi map of points of $J[1 - \zeta_p] \subset J[p]$ represented by the divisors $(\alpha_i, 0) - (\beta_i, 0)$ for $0 \leq i \leq h$ (we remark that these divisors do not generate all of $J[1 - \zeta_p]$ but do generate a rank-$h$ $\zz/p\zz$-submodule of $J[1 - \zeta_p]$).  In order to state this theorem, for a given index $i \in \{0, \dots, h\}$ and a given element $\gamma \in \Gamma_0$, we define $\sigma_i(\gamma) \in \zz / p\zz$ to be the reduction modulo $p$ of the sum of exponents of the letter $s_i$ appearing in the element $\gamma$ expressed as a word on the generators $s_j$ of $\Gamma_0$.  It is clear from the fact that the only relations between these generators are given by $s_0^p = \cdots = s_h^p = 1$ that the map $\sigma_i : \Gamma_0 \to \zz / p\zz$ thus defined is a well-defined homomorphism.  Noting that this homomorphism is trivial on the commutator subgroup of $\Gamma \lhd \Gamma_0$, we see that it induces a homomorphism $\bar{\Gamma} \to \zz / p\zz$ which, by slight abuse of notation, we again denote by $\sigma_i$.  Identifying the group $\zz / p\zz$ with the subgroup $\boldsymbol{\mu}_p < K^\times$ of $p$th roots of $1$ (by identifying $1$ modulo $p$ with the primitive root $\zeta_p \in \boldsymbol{\mu}_p$), we consider $\sigma_i$ to be an element of $T = \Hom(\bar{\Gamma}, K^\times)$.

\begin{thm} \label{thm certain p-torsion points}

For $0 \leq i \leq h$, we have 
\begin{equation}
\mathrm{AJ}((\alpha_i, 0) - (\beta_i, 0)) = \sigma_i^{n_i},
\end{equation}
where $n_i \in \zz$ satisfies $m_i n_i \equiv 1$ (mod $p$) (with $m_i$ the exponent appearing in (\ref{eq split degenerate superelliptic})).

\end{thm}

The following corollary gives us the last statement of \Cref{thm main}.

\begin{cor} \label{cor integral combination of v_i's}

The generators $\mathfrak{v}_1, \dots, \mathfrak{v}_h$ given by \Cref{thm main} in the $\ell = p$ case satisfy that the image of the element $(\zeta_p + 2\zeta_p^2 + \dots + (p - 1)\zeta_p^{p - 1})\mathfrak{v}_i \in \mathbb{T}_p \subset T_p(J)$ modulo $p$ in $J[p]$ is represented by the divisor $\sum_j m_j ((\alpha_j, 0) - (\beta_j, 0)) \in \Div_0(C)$, where the sum is taken over all indices $j$ satisfying $d(\mathfrak{s}_i) - d(\mathfrak{s}_i \vee \mathfrak{s}_j) \leq \frac{pv(p)}{p-1}$ and $m_j$ is the exponent appearing in (\ref{eq split degenerate superelliptic}).

\end{cor}

\begin{proof}
Fix an index $i \in \{1, \dots, h\}$.  Recall the tree $\mathcal{I}$ defined in \S\ref{sec berk}, and define the graph $\mathcal{I}_i$ to be the sub-tree of $\mathcal{I}$ rooted at the vertex $i$ (that is, the sub-tree consisting of all vertices ``downstream" from $i$, meaning that there is a directed path from $i$ ending at each such vertex).  Then by definition, using the notation set in \S\ref{sec berk}, we have $j \in \mathcal{I}_i$ if and only if $\hat{v}_i > \hat{v}_j$.  This condition gives us $\hat{v}_i > v_i, v_j$ and therefore $\hat{v}_i \geq v_i \vee v_j \geq v_i$, which via \Cref{prop dictionary}(a) and \Cref{prop v_i etc.}(a)(b)(c) can be translated to the condition that $\delta(v_i \vee v_j, v_i) \leq \frac{pv(p)}{p-1}$.  Thus, the divisor in $\Div_0(C)$ specified in the statement can be written as $m_i \sum_{j \in \mathcal{I}_i} ((\alpha_j, 0) - (\beta_j, 0))$.  By \Cref{thm certain p-torsion points}, we have 
\begin{equation}
\mathrm{AJ}\big(\sum_{j \in \mathcal{I}_i} m_j ((\alpha_j, 0) - (\beta_j, 0))\big) = \prod_{j \in \mathcal{I}_i} \sigma_j^{m_j n_j} = \prod_{j \in \mathcal{I}_i} \sigma_j.
\end{equation}

Let us recall the set $\{\gamma_{l, l', \nu}\}_{1 \leq l \leq h, \, 1 \leq \nu \leq p -1}$ defined in \S\ref{sec period matrix} and established as a basis of $\bar{\Gamma}$ by \Cref{cor basis of barGamma}.  We have 
\begin{equation}
\sigma_j:
\begin{cases}
\gamma_{j, j', \nu} &\mapsto \zeta_p \\
\gamma_{l, l', \nu} &\mapsto \zeta_p^{-1} \text{ if } j = l' \\
\gamma_{l, l', \nu} &\mapsto 1 \text{ if } j \neq l, l'.
\end{cases}
\end{equation}
Then for $l \notin \mathcal{I}_i$ (resp. $l \in \mathcal{I}_i \smallsetminus \{i\}$ so that $l' \in \mathcal{I}_i$; resp. $l = i \in \mathcal{I}_i$ so that $l' \notin \mathcal{I}_i$), the product $\prod_{j \in \mathcal{I}_i} \sigma_j(\gamma_{l, l', \nu})$ comes out to $1$ (resp. $\zeta_p \zeta_p^{-1} = 1$; resp. $\zeta_p$).  Therefore, $\prod_{j \in \mathcal{I}_i} \sigma_j : \bar{\Gamma} \to \boldsymbol{\mu}_p$ is the character mapping $\gamma_{i, i', \nu}$ to $\zeta_p$ for all $\nu \in \zz$ and mapping $\gamma_{l, l', \nu}$ to $1$ for all $l \neq i$ and $\nu \in \zz$.

Meanwhile, for each $n \in \zz$ viewing $\zeta_p^n \mathfrak{v}_i$ as determined by the character $\chi_{\mathfrak{v}_i} : \bar{\Gamma} \to T_p(\bar{K}^\times)$ given by the formula provided in (\ref{eq zeta_p^n v_i}) by \Cref{lemma zeta_p^n v_i}, its reduction modulo $p$ is the character 
\begin{equation}
\bar{\chi}_{\zeta_p^n \mathfrak{v}_i} : 
\begin{cases}
\gamma_{i, i', -n} &\mapsto \zeta_p^{-1} \\
\gamma_{i, i', 1 - n} &\mapsto \zeta_p \\
\gamma_{l, l', m} &\mapsto 1 \text{ if } l \neq i \text{ or } m + n \not\equiv 0, 1 \ (\text{mod} \ p).
\end{cases}
\end{equation}
Now the sum $(\sum_{n = 1}^{p - 1} n \zeta_p^n) \mathfrak{v}_i$ is determined by a character whose reduction is $\prod_{n = 1}^{p - 1} \bar{\chi}_{\zeta_p^n \mathfrak{v}_i}^n$, which sends $\gamma_{i, i', \nu}$ to $(\zeta_p^{-1})^{-\nu}(\zeta_p)^{1 - \nu} = \zeta_p$ and which sends $\gamma_{l, l', \nu}$ to $1$ for $l \neq i$.  This matches the character computed above which is represented by the divisor $\sum_{j \in \mathcal{I}_i} m_j ((\alpha_j, 0) - (\beta_j, 0)) \in \Div_0(C)$.
\end{proof}

In order to prove \Cref{thm certain p-torsion points}, we need a couple of lemmas.  The first of these is essentially \cite[Lemma 1.1]{van1983note}, but since that source is difficult to access and there is a repeated typo in one of the formulas, we rephrase (a slight generalization of) the result and its proof below.

\begin{lemma} \label{lemma automorphy van steen}

Given $a, b \in \Omega$ and an automorphism $s \in \PGL_2(K)$ which normalizes the subgroup $\Gamma < \PGL_2(K)$ and which is represented by the matrix $\begin{bsmallmatrix} A & B \\ C & D \end{bsmallmatrix}$, we have 
\begin{equation}
\Theta_{a, b}(s(z)) \equiv c \Theta_{s^{-1}a, s^{-1}b}(z),
\end{equation}
where the constant $c \in K^\times$ is given by 
\begin{equation} \label{eq c}
c = \prod_{\gamma \in \Gamma} \frac{A - C\gamma(a)}{A - C\gamma(b)} = \Theta_{a, b}(AC^{-1}) = \Theta_{a, b}(s(\infty)).
\end{equation}

\end{lemma}

\begin{rmk}

Having set $\infty = b_0 \in S \subset \Omega$, we have $s(\infty) \in \Omega$, and so the product in the above formula for $c$ converges.

\end{rmk}

\begin{proof}[Proof (of \Cref{lemma automorphy van steen})]

With the symbol $\cdot$ indicating the action of (the image in $\PGL_2(K)$ of) a matrix on an element of $K \cup \{\infty\}$), we compute 
\begin{equation}
\begin{aligned}
\Theta_{a, b}(s(z)) &= \prod_{\gamma \in \Gamma} \frac{s(z) - \gamma(a)}{s(z) - \gamma(b)} = \prod_{\gamma \in \Gamma} \Big(\begin{bmatrix} 1 & -\gamma(a) \\ 1 & -\gamma(b) \end{bmatrix} \begin{bmatrix} A & B \\ C & D \end{bmatrix} \cdot z \Big) \\
&= \prod_{\gamma \in \Gamma} \Big(\begin{bmatrix} A - C\gamma(a) & 0 \\ 0 & A - C\gamma(b) \end{bmatrix} \begin{bmatrix} 1 & -s^{-1}\gamma(a) \\ 1 & -s^{-1}\gamma(b) \end{bmatrix} \cdot z \Big) \\
&= \prod_{\gamma \in \Gamma} \Big(\frac{A - C\gamma(a)}{A - C\gamma(b)}\Big) \Big(\frac{z - s^{-1}\gamma s (s^{-1}(a))}{z - s^{-1}\gamma s (s^{-1}(b))}\Big) = c \prod_{\gamma \in \Gamma} \frac{z - \gamma(s^{-1}(a))}{z - \gamma(s^{-1}(b))}.
\end{aligned}
\end{equation}
\end{proof}

\begin{lemma} \label{lemma clever cancelling}

Choose an element $\gamma \in \Gamma$ and an index $i \in \{0, \dots, h\}$.  Let $n_i \in \zz$ be as in \Cref{thm certain p-torsion points}.  We have 
\begin{equation} \label{eq clever cancelling product}
\prod_{d = 0}^{p - 1} \frac{(\zeta_p^{n_i} b_i - a_i) + (1 - \zeta_p^{n_i}) s_i^d \gamma s_i^{-d}(a_i)}{(\zeta_p^{n_i} b_i - a_i) + (1 - \zeta_p^{n_i}) s_i^d \gamma s_i^{-d}(b_i)} = 1.
\end{equation}

\end{lemma}

\begin{proof}

Throughout this proof, in order to simplify notation, we drop the index $i$ from the symbols $a_i, b_i, s_i, n_i$.

We first compute a matrix in $\GL_2(K)$ representing $s^d$ for any $m \in \zz$.  As discussed in \S\ref{sec background Schottky}, there is a matrix representing $s$ which is similar to the matrix $\begin{bsmallmatrix} \zeta_p^n & 0 \\ 0 & 1 \end{bsmallmatrix}$.  The image of such a matrix in $\PGL_2(K)$ fixes $\infty$ and $0$, and, as in the set-up of \cite[Proposition 3.2]{van1982galois}, the automorphism $s$ is obtained by conjugating this by a rational linear transformation which sends $0$ and $\infty$ respectively to $a$ and $b$; such a transformation is given by $z \mapsto \frac{bz + a}{z + 1}$.  Therefore, we may compute a matrix representing $s^d$ to be 
\begin{equation} \label{eq matrix for s^d}
\begin{bmatrix} b & a \\ 1 & 1 \end{bmatrix} \begin{bmatrix} \zeta_p^n & 0 \\ 0 & 1 \end{bmatrix}^d \begin{bmatrix} b & a \\ 1 & 1 \end{bmatrix}^{-1} = \begin{bmatrix} \zeta_p^{dn} b - a & (1 - \zeta_p^{dn})ab \\ (\zeta_p^{dn} - 1) & b - \zeta_p^{dn} a \end{bmatrix}.
\end{equation}

Now for any $N \in \zz$, let us write 
\begin{equation*}
\xi_N = \frac{(\zeta_p^N - 1) \gamma(a) + (b - \zeta_p^N a)}{(\zeta_p^N - 1) \gamma(b) + (b - \zeta_p^N a)} \in K.
\end{equation*}
Noting that as the elements $a, b \in \Omega$ are each fixed by $s$, we have $\gamma s^{-d}(a) = \gamma(a)$ and $\gamma s^{-d}(b) = \gamma(b)$, we now compute the term in the product in (\ref{eq clever cancelling product}) corresponding to a fixed $d$ as 
\begin{equation}
\begin{aligned}
&\ \ \ \ \frac{(\zeta_p^n b - a) + (1 - \zeta_p^n) s^d(\gamma(a))}{(\zeta_p^n b - a) + (1 - \zeta_p^n) s^d(\gamma(b))} = \frac{(\zeta_p^n b - a) + (1 - \zeta_p^n) \big(\frac{(\zeta_p^{dn}b - a)\gamma(a) + (1 - \zeta_p^{dn})a b}{(\zeta_p^{dn} - 1)\gamma(a) + (b - \zeta_p^{dn}a)}\big)}{(\zeta_p^n b - a) + (1 - \zeta_p^n) \big(\frac{(\zeta_p^{dn}b - a)\gamma(b) + (1 - \zeta_p^{dn})a b}{(\zeta_p^{dn} - 1)\gamma(b) + (b - \zeta_p^{dn}a)}\big)} \\
&= {\scriptstyle \Big(\frac{(\zeta_p^{dn} - 1)\gamma(b) + (b - \zeta_p^{dn}a)}{(\zeta_p^{dn} - 1)\gamma(a) + (b - \zeta_p^{dn}a)}\Big) }
{\scriptstyle \frac{(\zeta_p^n b - a)(\zeta_p^{dn} - 1)\gamma(a) + (\zeta_p^n b - a)(b - \zeta_p^{dn}a) + (1 - \zeta_p^n)(\zeta_p^{dn}b - a)\gamma(a) + (1 - \zeta_p^n)(1 - \zeta_p^{dn})ab)}{(\zeta_p^n b - a)(\zeta_p^{dn} - 1)\gamma(b) + (\zeta_p^n b - a)(b - \zeta_p^{dn}a) + (1 - \zeta_p^n)(\zeta_p^{dn}b - a)\gamma(b) + (1 - \zeta_p^n)(1 - \zeta_p^{dn})ab)} } \\
&= \Big(\frac{(\zeta_p^{dn} - 1)\gamma(b) + (b - \zeta_p^{dn}a)}{(\zeta_p^{dn} - 1)\gamma(a) + (b - \zeta_p^{dn}a)}\Big) \Big(\frac{\zeta_p^n(b - a)\big((\zeta_p^{dn - n} - 1)\gamma(a) + (b - \zeta_p^{dn - n}a)\big)}{\zeta_p^n(b - a)\big((\zeta_p^{dn - n} - 1)\gamma(b) + (b - \zeta_p^{dn - n}a)\big)}\Big) \\
&= \xi_{dn}^{-1} \xi_{dn - n}.
\end{aligned}
\end{equation}

The product in (\ref{eq clever cancelling product}) can thus be written as $\prod_{d = 0}^{p - 1} \xi_{dn}^{-1} \xi_{(d - 1)n}$, in which all terms clearly cancel (as $\xi_N$ only depends on $N$ modulo $p$) so that the product becomes $1$.
\end{proof}

\begin{lemma} \label{lemma automorphy for s}

Putting $a = a_i$, $b = b_i$, and $s = s_i$ into the statement of \Cref{lemma automorphy van steen} and letting $n_i$ be as in \Cref{thm certain p-torsion points}, the constant $c \in K^\times$ guaranteed by \Cref{lemma automorphy van steen} equals $\zeta_p^{n_i}$.

\end{lemma}

\begin{proof}

Using the formula for $c$ provided by \Cref{lemma automorphy van steen} and the fact that $s_i$ is represented by the matrix given in (\ref{eq matrix for s^d}), we have 
\begin{equation} \label{eq formula for automorphy for s}
c = \prod_{\gamma \in \Gamma} \frac{(\zeta_p^{n_i} b_i - a_i) + (1 - \zeta_p^{n_i}) \gamma(a_i)}{(\zeta_p^{n_i} b_i - a_i) + (1 - \zeta_p^{n_i}) \gamma(b_i)}.
\end{equation}
Now it follows from the group structure of $\Gamma$ discussed in \S\ref{sec background Schottky} that each orbit of $\Gamma$ under the action of conjugation by $s_i$ has cardinality $p$ except for the orbit consisting of the identity.  The terms of the product in (\ref{eq formula for automorphy for s}) consist of the term corresponding to $\gamma = 1$ and the terms corresponding to each of the $p$ elements of $\Gamma$ in each of the orbits $\{s_i^d \gamma_0 s_i^{-d}\}_{0 \leq d \leq p - 1}$ represented by an element $\gamma_0 \in \Gamma \smallsetminus \{1\}$.  For each such orbit $\{s_i^d \gamma_0 s_i^{-d}\}_{0 \leq d \leq p - 1}$, the product of the terms corresponding to elements in that orbit comes out to $1$ by \Cref{lemma clever cancelling}.  The formula for $c$ in (\ref{eq formula for automorphy for s}) therefore simplifies to just the term in the product corresponding to $\gamma = 1$; thus, we may compute 
\begin{equation}
c = \frac{(\zeta_p^{n_i} b_i - a_i) + (1 - \zeta_p^{n_i}) a_i}{(\zeta_p^{n_i} b_i - a_i) + (1 - \zeta_p^{n_i}) b_i} = \frac{\zeta_p^{n_i}b_i - \zeta_p^{n_i}a_i}{b_i - a_i} = \zeta_p^{n_i}.
\end{equation}
\end{proof}

\enlargethispage{\baselineskip}

We are now ready to prove \Cref{thm certain p-torsion points}; most of the ideas of our argument are a variation on the proof of \cite[Proposition 1.4]{van1983note}.

\begin{proof}[Proof (of \Cref{thm certain p-torsion points})]

Fix an index $i$.  Recalling the elements $\gamma_{j, i, n} = \gamma_i^{n - 1} \gamma_j \gamma_i^{-n} \in \Gamma$ in \S\ref{sec period matrix} for $j \neq  i$ and $1 \leq n \leq p - 1$ (note the reversed labeling of the indices), one easily sees that these $(p - 1)h$ elements generate the group $\Gamma$ (and thus that their images in $\bar{\Gamma}$, also denoted $\gamma_{j, i, d}$, generate $\bar{\Gamma}$).

Recalling from \S\ref{sec background Schottky} that the images of the points $a_i, b_i \in \Omega$ modulo the action of $\Gamma_0$ are respectively $\alpha_i, \beta_i \in K \cup \{\infty\}$, the image of the divisor $(\alpha_i, 0) - (\beta_i, 0)$ under the Abel-Jacobi map is $c_{a_i, b_i} \in T$.  Our task is therefore to show that, for $1 \leq n \leq p - 1$ and $j \neq i$, we have 
\begin{equation}
c_{a_i, b_i}(\gamma_{j, i, n}) = \zeta_p^{-n_i} = \sigma_i^{n_i}(\gamma_{j, i, n}).
\end{equation}
Fix an integer $n \in \{0, \dots, p - 1\}$ and an index $j \neq i$.  To compute $c_{a_i, b_i}(\gamma_{j, i, n})$ in terms of theta functions, we may put $z = s_i^n(a_j) \in \Omega$ and exploit the fact that $a_j$ is fixed by $s_j$.  Using Lemmas \ref{lemma automorphy van steen} and \ref{lemma automorphy for s}, we compute 
\begin{equation}
\begin{aligned}
c_{a_i, b_i}(\gamma_{j, i, n}) = \Theta_{a_i, b_i}(\gamma_{j, i, n}(z)) \Theta_{a_i, b_i}(z)^{-1} &= \big[\Theta_{a_i, b_i}(s_i^{n - 1}(a_j))\big] \big[\Theta_{a_i, b_i}(s_i^n(a_j))\big]^{-1} \\
&= \big[\zeta_p^{(n - 1)n_i} \Theta_{a_i, b_i}(a_j)\big] \big[\zeta_p^{nn_i} \Theta_{a_i, b_i}(a_j)\big]^{-1} = \zeta_p^{-n_i}.
\end{aligned}
\end{equation}

\end{proof}

\bibliographystyle{plain}
\bibliography{bibfile}

\end{document}